\documentclass[a4paper,reqno]{amsart}
\usepackage[english]{babel}

\usepackage{amssymb}
\usepackage{mathtools}

\usepackage{hyperref}

\usepackage{enumitem}
\setlist[enumerate]{leftmargin=*}

\newtheorem{theorem}{Theorem}[section]
\newtheorem{lemma}[theorem]{Lemma}
\newtheorem{corollary}[theorem]{Corollary}
\newtheorem{prop}[theorem]{Proposition}

\theoremstyle{definition}
\newtheorem{definition}[theorem]{Definition}

\theoremstyle{remark}
\newtheorem{remark}[theorem]{Remark}

\numberwithin{equation}{section}

\newcommand{\ZZ}{\mathbb{Z}}
\newcommand{\NN}{\mathbb{N}}
\newcommand{\RR}{\mathbb{R}}
\newcommand{\CC}{\mathbb{C}}

\newcommand{\Npos}{\NN_+}

\newcommand{\Rz}{\mathcal{R}}

\newcommand{\opL}{\mathcal{L}}

\newcommand{\opAv}{\mathcal{A}}
\newcommand{\opM}{\mathcal{M}}

\newcommand{\modE}{\mathcal{E}}

\newcommand{\Comp}{\mathcal{C}}

\newcommand{\Dy}{\mathcal{D}}
\newcommand{\Int}{\mathcal{I}}

\DeclareMathOperator{\diam}{diam}

\newcommand{\Hol}{\mathcal{H}}

\newcommand{\Pol}{\mathcal{P}}

\newcommand{\fin}{\mathrm{fin}}

\newcommand{\cO}{\mathcal{O}}

\newcommand{\baK}{\mathcal{K}}

\newcommand{\Bdd}{\mathcal{B}}

\newcommand{\Abel}{\mathcal{J}}

\newcommand{\id}{\mathrm{id}}

\newcommand{\myth}{\omega_*}

\renewcommand{\colon}{\,:\,}

\newcommand{\pred}{\mathfrak{p}}
\renewcommand{\succ}{\mathfrak{s}}

\newcommand{\chr}{\mathbf{1}}

\newcommand{\defeq}{\mathrel{:=}}
\newcommand{\eqdef}{\mathrel{=:}}

\newcommand{\ord}{\mathrm{ord}}

\DeclareMathOperator{\supp}{supp}

\newcommand{\convcc}{\xrightarrow[\mathrm{cc}]{}}

\newcommand{\TT}[1]{\mathbb{T}_{#1}}
\newcommand{\Tq}{\TT{q}}
\newcommand{\mTT}[1]{m_{\TT{#1}}}
\newcommand{\mq}{\mTT{q}}
\newcommand{\Lq}{\opL_{\Tq}}
\newcommand{\Rq}{\Rz_{\Tq}}

\begin{document}

\title[Riesz transform and spectral multipliers for the flow Laplacian]{Riesz transform and spectral multipliers for the flow Laplacian on nonhomogeneous trees}

\author[A. Martini]{Alessio Martini}
\address[A. Martini]{Dipartimento di Scienze Matematiche ``Giuseppe Luigi Lagrange'' \\   Politecnico di Torino \\ Corso Duca degli Abruzzi 24 \\ 10129 Torino \\ Italy}
\email{alessio.martini@polito.it}

\author[F. Santagati]{Federico Santagati}
\address[F. Santagati]{Dipartimento di Matematica e Applicazioni \\   Universit\`a degli Studi di Milano-Bicocca  \\ Via Cozzi 55 \\ 20125 Milano \\ Italy}
\email{federico.santagati@unimib.it}

\author[A. Tabacco]{Anita Tabacco}
\address[A. Tabacco]{Dipartimento di Scienze Matematiche ``Giuseppe Luigi Lagrange'' \\  Politecnico di Torino \\ Corso Duca degli Abruzzi 24 \\ 10129 Torino \\ Italy}
\email{anita.tabacco@polito.it}

\author[M. Vallarino]{Maria Vallarino}
\address[M. Vallarino]{Dipartimento di Scienze Matematiche ``Giuseppe Luigi Lagrange'' \\  Politecnico di Torino \\ Corso Duca degli Abruzzi 24 \\ 10129 Torino \\ Italy}
\email{maria.vallarino@polito.it}

\thanks{Work partially supported by the Project ``Harmonic analysis on continuous and discrete structures'' (bando Trapezio Compagnia di San Paolo CUP E13C21000270007). The first-named author gratefully acknowledges the financial support of Compagnia di San Paolo through the ``Starting Grant'' programme. The authors are members of the Gruppo Nazionale per l'Analisi Matema\-tica, la Probabilit\`a e le loro Applicazioni (GNAMPA) of the Istituto Nazionale di Alta Matematica (INdAM)}


\begin{abstract}
Let $T$ be a locally finite tree equipped with a flow measure $m$. Let $\opL$ be the flow Laplacian on $(T,m)$. We prove that the first order Riesz transform $\nabla \opL^{-1/2}$ is bounded on $L^p(m)$ for $p\in (1,\infty)$. Moreover, we prove a sharp $L^p$ spectral multiplier theorem of Mihlin--H\"ormander type for $\opL$. In the case where $m$ is locally doubling, we also prove corresponding weak type and Hardy space endpoint bounds. This generalises results by Hebisch and Steger for the canonical flow Laplacian on homogeneous trees to the setting of nonhomogeneous trees with arbitrary flow measures. The proofs rely on approximation and perturbation arguments, which allow one to transfer to any flow tree a number of $L^p$ bounds that hold on homogeneous trees of arbitrarily large degree and are uniform in the degree.
\end{abstract}

\maketitle

\section{Introduction}

\subsection{Summary of the results}
Let $T$ denote an infinite tree, i.e., an infinite connected graph with no loops, equipped with the usual discrete distance $d$. We identify $T$ with its set of vertices and say that $x,y \in T$ are neighbours if $d(x,y)=1$; in this case, we write $x \sim y$. We shall assume throughout that $T$ is \emph{locally finite}, i.e., every vertex has finitely many neighbours; the number of neighbours of a vertex is also known as its \emph{degree}.

Let $\partial T$ be the boundary at infinity of $T$ (defined, e.g., as in \cite[Section I.1]{FTN}). We choose a boundary point $\myth \in \partial T$ that we think of as the root of $T$. We shall call $T$ a \emph{tree with root at infinity} whenever a root $\myth \in \partial T$ has been fixed. Such a choice induces a natural partial order on $T$: namely, for any $x,y \in T$, we say that $x \leq y$ if $y$ belongs to the semi-infinite geodesic from $x$ to $\myth$. We shall think of $T$ as hanging from its root $\myth$, so if $x \leq y$ we say that $x$ is below $y$ and $y$ is above $x$. In this way, any $x \in T$ has exactly one neighbour lying above $x$, which shall be referred to as the \emph{predecessor} of $x$ and denoted by $\pred(x)$; the remaining neighbours lie below $x$ and form the set $\succ(x)$ of \emph{successors} of $x$. 

We set $q(x) = \# \succ(x)$, where $\#$ is the counting measure.
The tree $T$ is called \emph{homogeneous} if $x \mapsto q(x)$ is constant, i.e., if every vertex $x$ has the same degree, which is then said to be the degree of the tree.
For any positive integer $q$, we shall denote by $\Tq$ the homogeneous tree of degree $q+1$ with a fixed root at infinity (so $q(x) = q$ for all $x \in \Tq$); of course, for any positive integer $q$ there is only one such tree $\Tq$, up to isomorphisms. 
While homogeneous trees play an important role in our discussion, one of the main objectives of the present paper is to develop an analysis also encompassing the case of \emph{nonhomogeneous} trees $T$ with root at infinity.

Given a semi-infinite geodesic $(x_j)_{j \geq 0}$ in $T$ with $\myth$ as an endpoint, we define the \emph{level function} $\ell: T \to \ZZ$  as
\[
\ell(x) = \lim_{j \to \infty} \left[ j-d(x,x_j) \right], \qquad x \in T.
\]
It is easily seen that the level function $\ell$ is uniquely determined by $\myth$ up to an additive shift; in particular, the collection of the level sets of $\ell$, also known as \emph{horocycles}, does not depend on the geodesic $(x_j)_{j \geq 0}$. When working with a tree $T$ with root at infinity, we shall assume that a level function $\ell$ has been chosen. Notice that, for any $x,y \in T$, we have $x \leq y$ if and only if $d(x,y) = \ell(y)-\ell(x)$.

The following definition describes a natural class of measures on trees with root at infinity, which are the object of our study.

\begin{definition}\label{d: flow}
A \emph{flow measure} on a tree $T$ with root at infinity is a function  $m: T \to (0,\infty)$ such that
\[
m(x) = \sum_{y \in \succ(x)} m(y), \qquad x \in T.
\]
We say that the pair $(T,m)$ is a \emph{flow tree} if $T$ is a tree with root at infinity and $m$ is a flow measure on $T$.
\end{definition}

Variants of the above definition can be found in the literature, also encompassing the case where the root is a vertex and not a boundary point; we refer the reader to \cite{LP} for a wide-ranging account of flows in Probability and Analysis on trees, with connections to Computer Science and Operations Research. In this work we focus on a Harmonic Analysis perspective and, specifically, the study of certain singular integral operators naturally associated with flow trees; of course, as we are working with a discrete measure, here local integrability is not an issue, so the term ``singular integral'' refers to a lack of integrability at infinity.

The existence of a flow measure on a tree $T$ with root at infinity, in the sense of Definition \ref{d: flow},  implies that $T$ has no leaves, i.e., $q(x) > 0$ for all $x \in T$. We shall identify a flow measure $m$ on $T$ with the discrete measure with density $m$ with respect to the counting measure $\#$ on $T$. In this way, we can consider Lebesgue spaces $L^p(m)$ and other function spaces on $T$ associated with a flow measure $m$.

In the case of the homogeneous tree $\Tq$, the \emph{canonical flow measure} $m_{\Tq}$ is given by $m_{\Tq}(x) = q^{\ell(x)}$ for all $x \in \Tq$. When $q \geq 2$, it is readily seen that the $m_{\Tq}$-measure of balls with a fixed centre grows exponentially with the radius. This example shows that, for an arbitrary flow tree $(T,m)$, the metric measure space $(T,d,m)$ need not satisfy the doubling property; indeed, \cite[Theorem 2.5]{LSTV} shows that the doubling property fails for most flow trees, so many standard techniques for the analysis of singular integrals are not directly available in this context. On the other hand, in recent years there has been considerable interest in extending aspects of the theory of singular integrals to nondoubling settings (see, e.g., \cite{CMM,CeMe,hs,MMV,NTV,T,To,Ve} and references therein), and the present work can be thought of as a contribution to this effort.

A flow measure $m$ is called \emph{locally doubling} if for every $R>0$ there exists a constant $D_R$ such that 
\[
m(B_{2r}(x))\leq D_R \, m(B_r(x)) \qquad \forall x\in T, \, r \in (0,R],
\]
where $B_r(x)$ denotes the ball centred at $x$ of radius $r$ with respect to $d$. It is known that, if a tree with root at infinity can be equipped with a locally doubling flow measure, then $T$ has \emph{bounded degree}, i.e., $\sup_{x \in T} q(x) < \infty$ \cite[Corollary 2.3]{LSTV}.
Following the seminal work \cite{hs}, a Calder\'on--Zygmund theory on a locally doubling flow tree $(T,m)$ was developed in \cite{LSTV}, and suitable Hardy and BMO spaces $H^1(m)$ and $BMO(m)$ were introduced (see also \cite{ATV2,ATV1,LSTV2,San}). Such theory is applied in this paper to study boundedness properties of singular integral operators associated with a natural Laplacian on flow trees. 

Let $(T,m)$ be a flow tree, and denote by $\CC^{T}$ the set of complex-valued functions defined on $T$. We  define the \emph{flow gradient} $\nabla$ as 
\[
\nabla f(x)=f(x)-f(\pred(x)), \qquad f \in \CC^T, \, x\in T,
\]
and the \emph{flow Laplacian} $\opL$ as $\frac{1}{2} \nabla^*\nabla$ (see Section \ref{s: Laplacian} for more details). We prove in Corollary \ref{c: spectrum} that for every $p\in [1,\infty]$ the spectrum of $\opL$ on $L^p(m)$ is $[0,2]$. In particular, $\opL$ has no spectral gap on $L^2(m)$.

One of our main results deals with $L^p$ boundedness properties of the first order Riesz transform associated with $\opL$, which is defined as $\Rz=\nabla \opL^{-1/2}$. The boundedness properties of $\Rz$ were studied in \cite{hs, LMSTV, MSV} in the setting of homogeneous trees equipped with the canonical flow Laplacian. In this paper, we obtain an analogous result for any tree equipped with a locally doubling flow measure, and actually the non-endpoint boundedness properties hold true for arbitrary flow trees.

\begin{theorem}\label{t: Riesz}
Let $(T,m)$ be a flow tree.
Then, the Riesz transform $\Rz$ is bounded on $L^p(m)$ for every $p \in (1,\infty)$.
Moreover, if $m$ is locally doubling, then $\Rz$ is of weak type $(1,1)$ and bounded from $H^1(m)$ to $L^1(m)$.
\end{theorem}

We also prove that, apart from trivial cases, $\Rz$ is unbounded from $L^\infty(m)$ to $BMO(m)$ whenever $(T,m)$ is locally doubling (see Proposition \ref{p: RunboundedLinftyBMO}). This shows that the case $p \geq 2$ of Theorem \ref{t: Riesz} cannot be simply obtained by duality considerations from the case $p \leq 2$, and a different approach is needed.

$L^p$ boundedness properties of Riesz transforms in discrete settings have received considerable attention in the literature. For example, in \cite{BaRu,fe,Ru2,Ru} several results are proved in the context of graphs satisfying the doubling property. The results of \cite{CeMe,CCH}, instead, do not require the doubling condition, but assume that the corresponding Laplacian has a spectral gap. As the doubling property may fail for arbitrary flow measures, while the flow Laplacian $\opL$ has no spectral gap, our Theorem \ref{t: Riesz} does not fall under the scope of those results.

Beside Riesz transforms, in this work we also deal with spectral multipliers of $\opL$. First of all, we establish that the flow Laplacian has a differentiable $L^p$ functional calculus: namely, we prove the $L^p$ boundedness for $p \in [1,\infty]$ of operators of the form $F(\opL)$ whenever 
$F$ lies in the standard inhomogeneous Sobolev space $L^2_s(\RR)$ of order $s>3/2$.

\begin{theorem}\label{t: multipliers}
Let $(T,m)$ be a flow tree. Let $s>3/2$. Then, there exists a positive constant $C_s$ such that, if $F\in L^2_s(\RR)$, then $F(\opL)$ is bounded on $L^p(m)$ for every $p\in [1,\infty]$ and
\begin{equation}\label{f: multipliers_bound}
\|F(\opL)\|_{L^p \to L^p} \le C_s\|F\|_{L^2_s}.
\end{equation}
\end{theorem}

Notice that, as the spectrum of $\opL$ is $[0,2]$, one can change $F : \RR \to \CC$ outside $[0,2]$ without changing $F(\opL)$, thus the Sobolev bound on $F$ is effectively only required on a neighbourhood of the interval $[0,2]$.

We also obtain a singular integral version of the above result, i.e., the following spectral multiplier theorem of Mihlin--H\"ormander type for the flow Laplacian.

\begin{theorem}\label{t: MHmultipliers}
Let $\chi\in C^{\infty}_c(\RR)$ be supported in $(3/4,5/4)$ and such that $0\leq \chi\leq 1$ and $\chi(1)=1$.
Let $(T,m)$ be a flow tree. Let $s>3/2$ and let $F:\RR \to \CC$ be a Borel measurable function.
\begin{enumerate}[label=(\roman*)]
\item\label{en: MHmultipliers_twist}
If
\begin{equation}\label{f: MH_cond_twist}
\sup_{0<t\leq 1}\|F(t\cdot)\chi\|_{L^2_s}+\sup_{0<t\leq 1}\|F(2-t\cdot)\chi\|_{L^2_s}<\infty,
\end{equation}
then $F(\opL)$ is bounded on $L^p(m)$ for every $p\in (1,\infty)$. If moreover $m$ is locally doubling, then  $F(\opL)$ is also of weak type $(1,1)$.
\item\label{en: MHmultipliers_std}
If $m$ is locally doubling and $F$ satisfies the more restrictive condition
\begin{equation}\label{f: MH_cond_std}
\sup_{0 < t \leq 2}\|F(t\cdot)\chi\|_{L^2_s}<\infty,
\end{equation}
then $F(\opL)$ is also bounded from $H^1(m)$ to $L^1(m)$, as well as from $L^{\infty}(m)$ to $BMO(m)$.
\end{enumerate}
\end{theorem}

The regularity threshold $3/2$ in Theorems \ref{t: multipliers} and \ref{t: MHmultipliers} is sharp, that is, it cannot be replaced by any smaller number, at least in the case $(T,m) = (\Tq,\mq)$ with $q \geq 2$ (see Proposition \ref{p: mult_sharpness} and Remark \ref{r: sharpness} below).

Part \ref{en: MHmultipliers_twist} of Theorem \ref{t: MHmultipliers} was essentially proved in \cite[Theorem 2.3]{hs} for the canonical flow Laplacian on homogeneous trees.
The assumption on the multiplier $F$ in part \ref{en: MHmultipliers_twist} is weaker than that in \ref{en: MHmultipliers_std}, as the former allows $F$ to be singular not only at $0$ but also at $2$. The assumption in part \ref{en: MHmultipliers_std} is analogous to the scale-invariant smoothness assumption for Fourier and spectral multipliers on $\RR^d$ and other settings (see, e.g., \cite{Hor0,MMu,Mih}); the restriction of the supremum in \eqref{f: MH_cond_std} to $t \leq 2$ is just due to the boundedness of the spectrum $[0,2]$ of $\opL$.

The fact that the condition on the multiplier $F$ in part \ref{en: MHmultipliers_twist} is invariant under the change of spectral variable $\lambda \mapsto 2-\lambda$ is a natural consequence of a certain ``modulation symmetry'' of the flow Laplacian $\opL$, see \eqref{f: modul_opL} below. However, while this modulation preserves Lebesgue and Lorentz spaces, it does not preserve the Hardy and BMO spaces on $(T,m)$, as it may destroy cancellations. Indeed (see Proposition \ref{p: twist_MH_counterex} below) the $H^1 \to L^1$ and $L^\infty \to BMO$ endpoint bounds, which hold under the assumption \eqref{f: MH_cond_std}, may fail under the weaker assumption \eqref{f: MH_cond_twist}.

An important tool in the proof of the above results is a class of weighted $L^1$ estimates for the heat kernel of $\opL$, denoted by $K_{e^{-t\opL}}$, and its gradient. These estimates, extending those obtained in \cite{MSV} in the particular case of homogeneous trees, are of independent interest and read as follows.

\begin{theorem}\label{t: est on T}
Let $(T,m)$ be a flow tree. Then, for every $\varepsilon>0$ and $t \ge 1$,
\begin{align*}  
\sup_{y \in T} \sum_{x \in T} e^{\varepsilon d (x,y)/\sqrt{t}} \, |K_{e^{-t\opL}}(x,y)| \, m(x) &\le c_\varepsilon, \\ 
\sup_{y \in T} \sum_{x \in T} e^{\varepsilon d (x,y)/\sqrt{t}} \, |K_{\nabla e^{-t\opL}}(x,y)| \, m(x) &\le \frac{c_\varepsilon}{\sqrt{t}}, \\ 
\sup_{y \in T} \sum_{x \in T} e^{\varepsilon d (x,y)/\sqrt{t}} \, |K_{e^{-t\opL}\nabla^*}(x,y)| \, m(x) &\le \frac {c_\varepsilon}{\sqrt{t}}, \\
\sup_{y \in T} \sum_{x \in T} e^{\varepsilon d (x,y)/\sqrt{t}} \, |K_{\nabla  e^{-t\opL} \nabla^*}(x,y)| \, m(x)&\le \frac{c_\varepsilon}{t},
\end{align*}
where $c_\varepsilon>0$ is a constant independent of $(T,m)$. Moreover, for every $t>0$,
\begin{equation}\label{f: levelestimates}
\begin{aligned}
\sup_{l \in \ZZ} \sup_{x \in T} \sum_{z \in T \colon \ell(z)=l} |K_{\nabla e^{-t\opL}}(x,z)| \, m(z) &\le \frac{C}{1+t},\\ 
\sup_{l \in \ZZ} \sup_{x \in T} \sum_{z \in T \colon \ell(z)=l} |K_{\nabla e^{-t\opL}}(z,x)| \, m(z) &\le \frac{C}{1+t},
\end{aligned}
\end{equation}
where $C>0$ is a constant independent of $(T,m)$.
\end{theorem}

\subsection{Proof strategy}
One of the fundamental techniques underlying the proof of our results is \emph{transference}. Namely, as we show in Sections \ref{s: submersions} and \ref{s: perturbation} below, any flow tree $(T,m)$ can be approximated by a sequence of flow trees $(T_j,m_j)$, each of which is a quotient of a homogeneous tree $(\TT{q_j},m_{\TT{q_j}})$. In addition, we show that, for a large class of operators in the functional calculus for the flow Laplacian, $L^p$ estimates and weighted $L^1$ kernel estimates can be transferred through this quotienting and approximation procedure. This results in the following ``universal transference'' result for $L^p$ bounds,
\[
\|F(\opL)\|_{L^p(m) \to L^p(m)} \leq \sup_q \|F(\Lq)\|_{L^p(\mq) \to L^p(\mq)},
\]
and in a similar transference result for weighted $L^1$ estimates of the integral kernel $K_{F(\opL)}$, which are valid for any sufficiently regular function $F$ (see Theorem \ref{t: unif_transf_opL} below). In particular, any bound of the above form that holds on homogeneous trees \emph{uniformly in $q$} can be transferred to an arbitrary flow tree.

We also prove analogous transference results for joint functions of the flow gradient $\nabla$ and its adjoint $\nabla^*$ (see Theorem \ref{t: unif_transf} below); a technical difficulty here is that $\nabla$ and $\nabla^*$ do not commute in general, however one can make sense of a ``joint functional calculus'' for $\nabla$ and $\nabla^*$ by means of noncommutative power series (see Section \ref{ss: noncommutative}). This allows us to transfer estimates for operators such as $\nabla \exp(-\opL)$ or $\nabla \exp(-\opL) \nabla^*$, so indeed (once the transference results are established) the heat kernel bounds of Theorem \ref{t: est on T} are a direct consequence of those proved in \cite{MSV} for homogeneous trees.

The $L^p$ transference results for quotients of flow trees developed in Section \ref{s: submersions} below are similar in spirit to the classical $L^p$ transference results for actions of amenable groups (see, e.g., \cite{BPW,CW,dL}). However, also due to the nonhomogeneity of the involved trees, in our context there does not seem to be an obvious group action to which our transference results can be reduced. The fact that transference methods can be applied even in a non-group-invariant context may be another reason of interest for the present work. The results of Section \ref{s: perturbation} are somewhat different in nature, as they are based on a perturbative argument that is affine to those used in other contexts for the transplantation of $L^p$ estimates (see, e.g., \cite{CMMP,KST,AMsphere,Mit}).

We stress once more that, in order for our transference strategy to yield results on arbitrary flow trees, the original bounds on the homogeneous trees $\Tq$ must hold uniformly in $q$. As it turns out, one source of such $q$-uniform bounds is a connection between the functional calculus for the flow Laplacian on $\Tq$ with that on $\TT{1}$, which is expressed in terms of a certain \emph{discrete Abel transform} (see, e.g., \cite{CMS0}). This connection was implicitly used in \cite{MSV} to derive $q$-uniform weighted $L^1$ bounds for the heat kernel and its gradient on $\Tq$, starting from similar estimates on $\TT{1}$. In Section \ref{s: homogeneous} below we use a similar strategy to obtain $q$-uniform weighted $L^1$ estimates for more general functions $F(\Lq)$ of the flow Laplacian, which are at the basis of Theorems \ref{t: multipliers} and \ref{t: MHmultipliers} above.

Notice that the flow tree $\TT{1}$ can be identified with $\ZZ$ with the usual discrete Laplacian. As the latter can be thought of as a discrete version of $\RR$ with the standard Laplacian, one may expect that the multiplier theorems stated above (Theorems \ref{t: multipliers} and \ref{t: MHmultipliers}) would hold under a weaker smoothness assumption, i.e., $s>1/2$ instead of $s>3/2$: indeed, the smoothness condition in the classical Mihlin--H\"ormander theorem for Fourier multipliers on $\RR^d$ \cite{Hor0,Mih} is $s>d/2$. As a matter of fact, the $s>1/2$ improvement of Theorems \ref{t: multipliers} and \ref{t: MHmultipliers} is possible if one restricts to the case of $\TT{1}$; this is likely well known to experts (cf.\ \cite{Alex,dL}) and could also be proved by adapting the arguments below. However, the ``Abel transform connection'' between $\TT{1}$ and $\Tq$ effectively introduces an additional degree-one weight in $L^1$ estimates (see Proposition \ref{p: stimaL1pesataTq} below), which explains the shift from $1/2$ to $3/2$. As already mentioned, this increased smoothness requirement is not just due to the proof, because we can show (see Proposition \ref{p: mult_sharpness}) that, for the canonical flow Laplacian on homogeneous trees $\Tq$ with $q \geq 2$, the threshold $3/2$ is sharp and cannot be lowered.

The above transference results do not apply to weak-type or Hardy space bounds. This is the main reason why the various endpoint results for singular integrals in the above theorems are only proved in the context of locally doubling flow trees. Indeed, in that context, the Calder\'on--Zygmund theory of \cite{hs,LSTV} is available; so one can use it to prove the desired bounds, by standard dyadic decompositions of the operators of interest, whereby each dyadic piece satisfies better estimates, which are amenable to transference techniques.

Even in relation to $L^p$ bounds, our transference results do not apply directly to singular integrals such as Riesz transforms or Mihlin--H\"ormander multipliers. In other words, for these operators, one cannot directly transfer $L^p$ bounds from homogeneous trees to nonhomogeneous ones, so we cannot just use as a black box the results in the homogeneous setting available in the literature. Nevertheless, we can approximate those singular integrals with appropriate ``truncations'', which are no longer singular, and show that the latter satisfy uniform $L^p$ bounds in the truncation parameter; moreover, by a suitable adjustment of the Calder\'on--Zygmund theory of \cite{hs,LSTV}, discussed in Section \ref{ss: CZ} below, we also obtain the $q$-uniformity of these $L^p$ bounds on $(\Tq,\mq)$. By transferring the $L^p$ bounds for the truncations, and then passing to the limit, one eventually recovers the desired $L^p$ boundedness results for singular integrals on arbitrary flow trees.

\subsection{Some open problems}
Homogeneous trees $\Tq$ with $q \geq 2$ are often seen as discrete counterparts of real hyperbolic spaces. Indeed, analogues of Theorems \ref{t: Riesz} to \ref{t: est on T} are known to hold for distinguished (sub-)Laplacians on hyperbolic spaces and more general solvable Lie groups \cite{GS,hs,AMriesz,MOV,MaPl,MaVa,Sj,SV2,Va}. The results of this paper show that the homogeneity constraint on trees can be dispensed with in the discrete setting; it would be interesting to investigate whether the continuous counterparts also admit more robust versions, with less rigid assumptions on the operator and the underlying manifold.

The smoothness threshold $3/2$ in our multiplier theorems (Theorems \ref{t: multipliers} and \ref{t: MHmultipliers}) also appears in their continuous counterparts on hyperbolic spaces and other solvable Lie groups \cite{hs,MOV,MaPl,Va}, where it can be interpreted as half the ``pseudodimension'' of the group (see also \cite{CGHM}). As homogeneous trees are expected to capture the ``coarse structure'' of hyperbolic spaces, it is not surprising that what plays the role of the ``dimension at infinity'' in the continuous setting appears in the discrete setting too. In this work we actually prove that the threshold $3/2$ is optimal for the homogeneous tree $(\Tq,\mq)$ with $q \geq 2$, while for $q=1$ we know it can be lowered to $1/2$; it remains an open problem to characterise the optimal threshold for any given flow tree $(T,m)$ and to determine whether it may attain intermediate values between $1/2$ and $3/2$ for certain nonhomogeneous trees (values below $1/2$ are not possible, as $\TT{1} = \ZZ$ is a quotient of any flow tree via the level function).

As a matter of fact, on homogeneous trees, we can prove the optimality of $3/2$ in reference to $L^1 \to L^1$ and $H^1 \to L^1$ bounds (see Proposition \ref{p: mult_sharpness}), but we do not know whether a lower smoothness requirement could be enough for weak type $(1,1)$ bounds or $L^p$ bounds for all $p \in (1,\infty)$. On the other hand, our proof of Theorem \ref{t: MHmultipliers} is fundamentally based on $L^1$ bounds for the dyadic pieces in which the singular integral operator is decomposed (see Theorems \ref{t: lemmahebisch} and \ref{t: lemmahebisch_sharp}), so in any case nothing better than $3/2$ could be achieved without a substantial change of strategy to prove bounds for singular integrals.

As already mentioned, the endpoint bounds in Theorems \ref{t: Riesz} and \ref{t: MHmultipliers} are only proved in the case of locally doubling flow trees $(T,m)$. It would be interesting to know whether, for example, the weak type $(1,1)$ endpoint bounds may also be valid on arbitrary flow trees. This would appear to be out of reach for the existing Calder\'on--Zygmund theory on flow trees \cite{hs,LSTV} and may require new ideas.

The negative $H^1 \to L^1$ boundedness results for the adjoint Riesz transform $\Rz^*$ and certain spectral multipliers $F(\opL)$ of the flow Laplacian (see Propositions \ref{p: RunboundedLinftyBMO} and \ref{p: twist_MH_counterex}) appear to indicate some limitations of the atomic Hardy space $H^1(m)$ on locally doubling flow trees $(T,m)$ defined in \cite{LSTV}. It remains an open problem whether a different definition of Hardy space could be given this setting, in order to better capture endpoint boundedness properties of natural singular integral operators; some related investigations can be found in \cite{CeMe,San}.

In any case, our results do not provide any explicit $p=1$ endpoint bounds for the adjoint Riesz transform $\Rz^*$. In particular, the question whether $\Rz^*$ is of weak type $(1,1)$ is open even in the setting of homogeneous trees $(\Tq,\mq)$ with $q \geq 2$ \cite{LMSTV}, as well as in the continuous counterpart discussed in \cite{AMriesz}.

\subsection*{Notation}
We write $\chr_S$ for the characteristic function of a set $S$.
The symbol $\NN$ denotes the set of natural numbers, including zero; we write $\Npos$ for the set $\NN \setminus \{0\}$ of positive integers.
For a real number $x$, we denote by $x_+$ its positive part $\max\{x,0\}$, and by $\lfloor x \rfloor$ its integer part $\max\{ k \in \ZZ \colon k \leq x \}$.

Given two nonnegative quantities $A$ and $B$, $A \lesssim B$ means that there exists a finite positive constant $C$ such that $A \le C B$, while $A \approx B$ means $A \lesssim B$ and $B \lesssim A$. Moreover, $A \lesssim_{x_1,x_2,\dots,x_n} B$, for some parameters $x_1,\dots,x_n$, means that the implicit constant may depend on $x_1,\dots,x_n$.

\section{Flow trees and flow Laplacians}\label{s: Laplacian}  

In this section we recall a few basic properties of flow Laplacians, obtaining in particular a description of their $L^p$ spectra. We also introduce some important spaces of functions on flow trees that will be used throughout, as well as a ``joint functional calculus'' for noncommuting operators via power series.

\subsection{The flow Laplacian and its spectrum}\label{ss: spectrum}

Let $(T,m)$ be a flow tree. Define the \emph{shift operator} $\Sigma : \CC^{T} \to \CC$ as 
\begin{equation}\label{f: shift}
\Sigma f(x)=f(\pred(x)), \qquad f \in \CC^T, \ x\in T.
\end{equation}
From Definition \ref{d: flow} it is clear that $\Sigma$ is a linear isometry on $L^p(m)$ for every $p \in [1,\infty]$. Moreover, the adjoint operator $\Sigma^*$ with respect to the $L^2(m)$ pairing acts on functions $f \in \CC^T$ as follows:
\[
\Sigma^* f(x)= \frac{1}{m(x)} \sum_{y \in \succ(x)} f(y) \, m(y), \qquad f \in \CC^T, \ x\in T.
\]
 
Given a flow tree $(T,m)$, we define the \emph{flow gradient} $\nabla$ as $I-\Sigma$ and the \emph{flow Laplacian} $\opL$ as $\frac{1}{2} \nabla^*\nabla$, where $\nabla^*$ is the adjoint operator of $\nabla$ on $L^2(m)$. An easy calculation shows that $\opL= I - \opAv$, where $\opAv = (\Sigma+\Sigma^*)/2$ is the averaging operator given by
\begin{equation}\label{f: mathcalA}
\opAv f(x)= \frac{1}{2} f(\pred(x))+ \frac{1}{2m(x)} \sum_{y \in \succ(x)} f(y) \, m(y), \qquad f \in \CC^T, \ x\in T.
\end{equation}

In the next result, we show that $\opL$ has no spectral gap.
 
\begin{prop}\label{p: spectrum}
Let $(T,m)$ be a flow tree and $\opL$ be the corresponding flow Laplacian. For every $p\in [1,\infty]$ the real points of the spectrum of $\opL$ on $L^p(m)$ are the points of the interval $[0,2]$. In particular, the $L^2$ spectrum of $\opL$ is $[0,2]$.
\end{prop}
\begin{proof}
As $\opL$ is selfadjoint, by duality (see, e.g., \cite[Section VIII.6]{Y}) we may assume that $p \leq 2$, and in particular $p<\infty$.

We first prove that the spectrum of $\opAv$ on $L^p(m)$ contains $[-1,1]$. Given $\theta\in \RR$, define $f_\theta : T \to \CC$ by
\[
f_{\theta}(x) = e^{i\theta \ell(x)}, \qquad x \in T.
\]
Choose any $o \in T$ with $\ell(o) = 0$. For every integer $d \geq 2$, define the sets
\[\begin{aligned}
V_d &= \{x\in T \colon x\leq o, \ \ell(x)\geq -d\}, \\
\overline{V_d} &= \{x\in T \colon x \leq o, \ \ell(x)\geq -d-1\}\cup \{\pred(o)\}, \\
{V}^{\circ}_d &= \{x\in T \colon x \leq o, \ -d+1 \le \ell(x)\leq -1\} .
\end{aligned}\] 
Consider the function $f_{\theta,d} = f_{\theta} \chr_{V_d}$, whose $L^p$ norm is $[m(o)(d+1)]^{1/p}$. It is easy to see that 
\[
\opAv f_{\theta,d}(x) = \begin{cases}
\cos\theta \, f_{\theta,d}(x) & \text{if } x\in {V}^{\circ}_d, \\
0 & \text{if } x\notin \overline{V_d}, \\
\frac{m(o)}{2m(\pred(o))} & \text{if } x=\pred(o), \\
\frac{1}{2} e^{-i\theta} & \text{if } x=o, \\ 
\frac{1}{2} e^{i\theta (\ell(x)+1)} & \text{if } x\in \overline{V_d} \text{ and } \ell(x)=-d,-d-1. \\
\end{cases} 
\]
It follows that 
\[
\frac{\| \opAv f_{\theta,d}-\cos\theta f_{\theta,d}\|_p}{\|f_{\theta,d}\|_{p}}
\lesssim \frac{m(\overline{V_d}\setminus {V}^{\circ}_d)^{1/p}}{[m(o)(d+1)]^{1/p}}
= \frac{[m(\pred(o))+3m(o)]^{1/p}}{[m(o)(d+1)]^{1/p}},
\]
which tends to $0$ as $d \to \infty$. This implies that $\cos\theta$ is in the spectrum of $\opAv$ for every $\theta\in\RR$. Hence $[-1,1]$ is contained in the spectrum of $\opAv$ on $L^p(m)$.

As $\opAv = (\Sigma+\Sigma^*)/2$ is clearly $L^p$ bounded with norm at most $1$, the spectrum of $\opAv$ on $L^p(m)$ is contained in the closed unit ball centred at the origin of $\CC$. As the spectrum contains $[-1,1]$, this interval exhausts the real points of the spectrum. In the case $p=2$, as $\opAv$ is selfadjoint, its spectrum is real and therefore $[-1,1]$ is the whole $L^2$ spectrum.

Since $\opL = I-\opAv$, we finally deduce that the real points of the $L^p$ spectrum of $\opL$ are the points of the interval $[0,2]$, which is the whole spectrum when $p=2$.
\end{proof}

We shall complete the characterisation of the $L^p$ spectrum of $\opL$ in Corollary \ref{c: spectrum} below.

As mentioned in the Introduction, an important example of flow tree is the homogeneous tree $\Tq$ with the so-called canonical flow measure $\mq$, for any $q \in \Npos$. Namely, $\Tq$ is a tree with root at infinity and $q(x) = q$ for all $x \in \Tq$, while
\begin{equation}\label{f: can_flow}
\mq(x) = q^{\ell(x)}, \qquad x \in \Tq.
\end{equation}
We shall denote by $\Sigma_{\Tq}$, $\nabla_{\Tq}$ and $\opL_{\Tq}$ the shift operator, the flow gradient and flow Laplacian on $(\Tq,\mq)$.

\subsection{Function spaces on a flow tree}

We now introduce some spaces of functions on a flow tree $(T,m)$, as well as classes of operators between these spaces, which will be relevant in the subsequent discussion.

We denote by $c_{00}(T)$ the set of the functions $f \in \CC^T$ with finite support.
For a linear operator $\cO : c_{00}(T) \to \CC^T$, we denote by $K_{\cO}$ its integral kernel, i.e., $K_{\cO} : T \times T \to \CC$ is such that, for all $f\in c_{00}(T)$,
\begin{equation}\label{eq:kernel_op}
\cO f(x)=\sum_{y\in T} K_{\cO}(x,y) \, f(y) \, m(y)  \qquad \forall x\in T.
\end{equation}
For any such operator $\cO$, the (formal) adjoint $\cO^* : c_{00}(T) \to \CC^T$ is the operator with integral kernel
\[
K_{\cO^*}(x,y) = \overline{K_{\cO}(y,x)}, \qquad x,y \in T,
\]
and satisfies
\[
\langle \cO^* f,g \rangle_{L^2(m)} = \langle f, \cO g \rangle_{L^2(m)} \qquad\forall f,g \in c_{00}(T).
\]
We point out that, if $\cO : L^p(m) \to L^p(m)$ is bounded for some $p \in [1,\infty)$, then $\cO$ is uniquely determined by its restriction $\cO|_{c_{00}(T)}$, and therefore by its integral kernel $K_{\cO}$. In particular, if $p=1$, then
\[
\|\cO\|_{1 \to 1} = \sup_{y \in T} \sum_{x \in T} |K_{\cO}(x,y)| \, m(x),
\]
and the formula \eqref{eq:kernel_op} is valid for any $f \in L^1(m)$.

We denote by $\Bdd(m)$ the set of the bounded operators $\cO : L^1(m) \to L^1(m)$ such that the adjoint $\cO^*$ is also bounded on $L^1(m)$. It is readily checked that $\Bdd(m)$ is a unital Banach $*$-algebra with the norm
\begin{equation}\label{eq:Bddm_norm}
\|\cO\|_{\Bdd(m)} = \max\{ \|\cO\|_{1 \to 1} , \|\cO^*\|_{1 \to 1}\}.
\end{equation}
Moreover, any $\cO \in \Bdd(m)$ can also be thought of as a bounded operator on $L^\infty(m)$, where $\cO f$ is given by \eqref{eq:kernel_op} for any $f \in L^\infty(m)$, and clearly
\[
\|\cO\|_{\infty \to \infty} = \|\cO^*\|_{1 \to 1} = \sup_{x \in T} \sum_{y \in T} |K_{\cO}(x,y)| \, m(y).
\]
So, we can rewrite \eqref{eq:Bddm_norm} as
\[
\|\cO\|_{\Bdd(m)} = \max\{ \|\cO\|_{1 \to 1} , \|\cO\|_{\infty \to \infty}\}.
\]
By interpolation, any $\cO \in \Bdd(m)$ is also bounded on $L^p(m)$ for all $p \in [1,\infty]$, with
\begin{equation}\label{f: BddLp}
\|\cO\|_{p \to p} \leq \|\cO\|_{\Bdd(m)}.
\end{equation}

Finally, it is not difficult to check that the set
\[
\Bdd_{\fin}(m) = \{ \cO \in \Bdd(m) \colon \sup \{d(x,y) \colon K_{\cO}(x,y) \neq 0\}  < \infty \}.
\]
is a unital $*$-subalgebra of $\Bdd(m)$, and that $\cO(c_{00}(T)) \subseteq c_{00}(T)$ for all $\cO \in \Bdd_{\fin}(m)$.
Moreover, by means of the formula \eqref{eq:kernel_op}, each $\cO \in \Bdd_\fin(m)$ extends to an operator $\cO : \CC^T \to \CC^T$.
Clearly the shift operator $\Sigma$, the flow gradient $\nabla$ and the flow Laplacian $\opL$ are all members of $\Bdd_\fin(m)$.

\subsection{Noncommutative polynomials and power series}\label{ss: noncommutative}

It is convenient to introduce some notation for noncommutative polynomials and power series in multiple indeterminates. This will allow us to consider a sort of ``joint functional calculus'' for the two noncommuting operators $\Sigma$ and $\Sigma^*$, including, e.g., operators of the form $H(\opL)$ for an entire function $H$ on $\CC$, as well as more complicated expressions such as
\begin{equation}\label{f: esempio_FC}
\nabla^k H(\opL) (\nabla^*)^h
\end{equation}
for all $k,h \in \NN$.

Let $d \in \Npos$. A formal power series in the $d$ noncommutative indeterminates $Z_1,\dots,Z_d$ is an expression of the form
\begin{equation}\label{f: formalpow}
F(Z_1,\dots,Z_d) = \sum_{N \in \NN} \sum_{\alpha \in \{1,\dots,d\}^N} c_\alpha (Z_1,\dots,Z_d)^\alpha
\end{equation}
for some coefficients $c_\alpha \in \CC$, where $\alpha \in \{1,\dots,d\}^N$ is a noncommutative multi-index of length $|\alpha| \defeq N \in \NN$
and
\[
(Z_1,\dots,Z_d)^\alpha \defeq Z_{\alpha_1} \cdots Z_{\alpha_N}
\]
denotes the noncommutative monomial of multi-degree $\alpha$. In case only finitely many of the coefficients $c_\alpha$ are nonzero, then $F(Z_1,\dots,Z_d)$ is called a noncommutative polynomial. More generally, if $R \in (0,\infty)$ and
\[
\|F\|_{(R)} \defeq \sum_{N \in \NN} R^N \sum_{\alpha \in \{1,\dots,d\}^N} |c_\alpha| < \infty,
\]
then we say that the power series $F(Z_1,\dots,Z_d)$ is $R$-absolutely convergent.

The formal power series in $d$ noncommutative indeterminates form an algebra $\Hol(d)$ with the natural operations; notice that the product of two noncommutative monomials is given by
\[
(Z_1,\dots,Z_d)^\alpha (Z_1,\dots,Z_d)^\beta = (Z_1,\dots,Z_d)^{\alpha \cup \beta}
\]
where $\alpha \cup \beta$ is the concatenation of the multi-indices $\alpha$ and $\beta$. For any $R \in (0,\infty)$, we shall write $\Hol(d,R)$ for the subalgebra of the $R$-absolutely convergent power series, and $\Pol(d)$ for the subalgebra of noncommutative polynomials. It is readily checked that $\Hol(d,R)$ is a Banach algebra with the norm $\|\cdot\|_{(R)}$, and that $\|(Z_1,\dots,Z_d)^\alpha\|_{(R)} = R^{|\alpha|}$ for any multi-index $\alpha$. Thus, for any $F \in \Hol(d,R)$, the series in the right-hand side of \eqref{f: formalpow}, thought of as an infinite sum in the Banach algebra $\Hol(d,R)$, converges absolutely to $F$ in $\Hol(d,R)$. In particular, $\Pol(d)$ is dense in $\Hol(d,R)$.

Much like their commutative counterparts, noncommutative power series and polynomials can be used to define, via substitutions, certain ``joint functions'' of a tuple of noncommuting elements of a (Banach) algebra. Namely, from the above definitions one immediately deduces the following result.

\begin{lemma}\label{l: subs_pws}
Assume that $\baK$ is a Banach algebra and $M_1,\dots,M_d \in \baK$ satisfy $\|M_j\|_{\baK} \leq R$ for $j=1,\dots,d$, where $R \in (0,\infty)$. If $F \in \Hol(d,R)$ is given by \eqref{f: formalpow}, then the series
\[
F(M_1,\dots,M_d) \defeq \sum_{N \in \NN} \sum_{\alpha \in \{1,\dots,d\}^N} c_\alpha (M_1,\dots,M_d)^\alpha
\]
converges absolutely in $\baK$, and $\|F(M_1,\dots,M_d)\|_{\baK} \leq \|F\|_{(R)}$.
\end{lemma}

We highlight two particular cases of the above assertion, which are especially significant for our discussion.
\begin{enumerate}
\item As $\|\Sigma\|_{\Bdd(m)} = \|\Sigma^*\|_{\Bdd(m)} = 1$, for any $F \in \Hol(2,1)$ the series
$F(\Sigma,\Sigma^*)$
converges absolutely in $\Bdd(m)$, and $\|F(\Sigma,\Sigma^*)\|_{\Bdd(m)} \leq \|F\|_{(1)}$.
\item As $\|\opL\|_{\Bdd(m)} = 2$, for any $H \in \Hol(1,2)$, the series $H(\opL)$ converges absolutely in $\Bdd(m)$ and $\|H(\opL)\|_{\Bdd(m)} \leq \|H\|_{(2)}$.
\end{enumerate}
The latter statement actually reduces to the former, because
\begin{equation}\label{f: opL_Sigma}
\opL = \frac{1}{2} (I-\Sigma^*)(I-\Sigma),
\end{equation}
and because of the following result.

\begin{lemma}\label{l: chvar}
If $H \in \Hol(1,2)$ and $F$ is defined by 
\begin{equation}\label{f: chvar}
F(Z_1,Z_2) = H((1-Z_2)(1-Z_1)/2),
\end{equation}
then $F \in \Hol(2,1)$ and $\|F\|_{(1)} \leq \|H\|_{(2)}$.
\end{lemma}
\begin{proof}
Notice that $\|Z_1\|_{(1)} = \|Z_2\|_{(1)} = 1$, so
\[
\|(1-Z_2)(1-Z_1)/2\|_{(1)} \leq \|1-Z_2\|_{(1)} \|1-Z_1\|_{(1)} / 2 \leq 2.
\]
Therefore, Lemma \ref{l: subs_pws}, applied to the Banach algebra $\baK = \Hol(2,1)$, its element $M_1 = (1-Z_2)(1-Z_1)/2$ and the power series $H \in \Hol(1,2)$, shows that the substitution \eqref{f: chvar} indeed defines an element $F$ of $\Hol(2,1)$ with the desired norm estimate.
\end{proof}

From Lemma \ref{l: chvar} and the expression \eqref{f: opL_Sigma} we deduce that the operators
of the form
\[
F(\Sigma,\Sigma^*), \qquad \text{where } F \in \Hol(2,1),
\]
include, among others, all the operators \eqref{f: esempio_FC} for any $k,h \in \NN$ and any $H \in \Hol(1,2)$. Thus, the ``noncommutative functional calculus'' for $(\Sigma,\Sigma^*)$ based on $\Hol(2,1)$ is sufficiently rich to include many of the operators of interest related to the flow gradient and the flow Laplacian.

\section{Flow submersions and transference}\label{s: submersions}

In this section we develop a transference theory for appropriately defined quotients of flow trees. We also characterise those flow trees $(T,m)$ that are quotients of the homogeneous tree $(\Tq,\mq)$, in terms of rationality properties of the flow $m$.

\subsection{Submersions and compatible operators}
We introduce a particular class of mappings between trees with root at infinity, which preserve the underlying structure.

\begin{definition}
Let $T_1,T_2$ be trees with root at infinity.
We say that $\pi: T_1 \to T_2$ is a \emph{submersion} if 
\begin{align}
\label{f: 1}
\pi(\pred(x)) &= \pred(\pi(x)) \qquad \forall x \in T_1, \\ 
\label{f: 2}
\pi(\succ(x)) &= \succ(\pi(x)) \qquad \forall x \in T_1 .
\end{align} 
\end{definition}

Observe that \eqref{f: 1} and \eqref{f: 2} imply that a submersion $\pi$ is surjective and strictly increasing, i.e., if $x < y$ on $T_1$ then $\pi(x) < \pi(y)$ on $T_2$.
Moreover, iteration of \eqref{f: 2} gives
\[
\pi(\succ^n(x)) = \succ^n(\pi(x)) \qquad\forall n \in \NN,
\]
where $\succ^n(x)$ is the set of the $n$th-generation descendants of $x$.
In particular,
\[
\pi(\Delta_x) = \Delta_{\pi(x)},
\]
where
\begin{equation}\label{f: tenda}
\Delta_z \defeq \{y \in T_j \colon y \le z\} = \bigcup_{n \in \NN} \succ^n(z)
\end{equation}
is the set of all the descendants of $z \in T_j$.

From the above properties, it is readily seen that a submersion $\pi$ is a $1$-Lipschitz map, i.e., 
\begin{equation}\label{f: 1-Lip}
d_2(\pi(x),\pi(y)) \le d_1(x,y) \qquad \forall x,y \in T_1,
\end{equation}
where $d_j$ is the distance function on $T_j$, and moreover
\[
\ell(\pi(x)) = \ell(x) + c_\pi \qquad\forall x \in T_1
\]
for some constant $c_\pi \in \ZZ$; in particular, up to an appropriate shifting of the level functions, one may assume that $c_\pi = 0$, i.e., the submersion $\pi$ is \emph{level-preserving}.

\begin{definition}
Let $T_1,T_2$ be trees with root at infinity.
\begin{enumerate}[label=(\alph*)]
\item Given a submersion $\pi : T_1 \to T_2$, we say that two flow measures $m_1$ and $m_2$, on $T_1$ and $T_2$, respectively, are \emph{$\pi$-compatible} if
\begin{equation}\label{f: daga}
\frac{m_2(\pi(x))}{m_2(\pred(\pi(x)))}=\frac{m_1(\pi^{-1}\{\pi(x)\} \cap \succ(\pred(x)))}{m_1(\pred(x))} \qquad \forall x \in T_1.
\end{equation}

\item Let $(T_1,m_1)$ and $(T_2,m_2)$ be flow trees. We say that $\pi: (T_1, m_1) \to (T_2,m_2)$ is a \emph{flow submersion} if $\pi : T_1 \to T_2$ is a submersion and the flow measures $m_1$ and $m_2$ are $\pi$-compatible. In this case we also say that $(T_2,m_2)$ is a \emph{flow quotient} of $(T_1,m_1)$.
\end{enumerate}
\end{definition}

The compatibility property \eqref{f: daga} says, roughly speaking, that the flow measure $m_2$ can be thought of as the push-forward of the flow measure $m_1$ via $\pi$, provided both flows are restricted to appropriate subsets of the trees. This idea is made precise in the following statement.

\begin{prop}\label{p: pushforward}
Let $(T_1,m_1)$ and $(T_2,m_2)$ be flow trees and let $\pi:(T_1,m_1) \to (T_2,m_2)$ be a flow  submersion.  Let $x \in T_1$, and let $\pi_{x}$ be the restriction of $\pi$ to $\Delta_x$.
Then, for every $z \in \pi(\Delta_x) = \Delta_{\pi(x)}$,
\begin{equation}\label{f: pushforward}
m_2(z)=c(x) \, m_1(\pi^{-1}_{x}\{z\}),
\end{equation}
where $c(x)=\frac{m_2(\pi(x))}{m_1(x)}$. In other words, the measure $m_2|_{\Delta_{\pi(x)}}$ is $c(x)$ times the push-forward of the measure $m_1|_{\Delta_x}$ via $\pi_x$.
\end{prop}
\begin{proof}
We prove \eqref{f: pushforward} by induction on the level of $z$.

If $\ell(z)=\ell(x)$, i.e., $z=\pi(x)$, then \eqref{f: pushforward} is trivially verified.

Now assume that \eqref{f: pushforward} holds for all the vertices in $\Delta_{\pi(x)}$ whose level is between $\ell(x)$ and $\ell(x)-n$, where $n \in \NN$,  and pick $z \in \Delta_{\pi(x)}$ at level $\ell(x)-n-1$. We can then choose $y_1,\dots,y_N \in \pi_x^{-1}\{z\}$ such that $\pred(y_j) \ne \pred(y_k)$ if $j \ne k$ and $\pi_x^{-1}\{\pred(z)\} = \{\pred(y_j)\}_{j=1}^N$.
Then, by \eqref{f: daga},
\[\begin{split}
m_1(\pi_{x}^{-1}\{z\})
&= \sum_{j=1}^N m_1(\pi^{-1}\{z\} \cap \succ(\pred(y_j)))=\sum_{j=1}^{N} \frac{m_2(z)}{m_2(\pred(z))} m_1(\pred(y_j)) \\ 
&= \frac{m_2(z)}{m_2(\pred(z))} m_1(\pi^{-1}_{x}(\pred(z))) =\frac{m_2(z)}{c(x)},
\end{split}\]
where in the last step we used the inductive hypothesis.
\end{proof}

\begin{definition}
Given two flow trees $(T_1,m_1)$ and $(T_2,m_2)$ and a flow submersion $\pi:(T_1,m_1) \to (T_2,m_2)$, we define the \emph{lifting operator} $\Phi_\pi :\CC^{T_2} \to \CC^{T_1}$  by 
\[
\Phi_\pi f=f \circ \pi.
\]
We say that an operator $\cO \in \Bdd(m_1)$ is \emph{$\pi$-compatible} if there exists $\tilde\cO \in \Bdd(m_2)$ such that
\begin{equation}\label{f: M1M2comp}
\cO \Phi_\pi =\Phi_\pi \tilde\cO   \qquad\text{and}\qquad \cO^*\Phi_\pi = \Phi_\pi \tilde\cO^* \qquad\text{on } L^\infty(m_2).
\end{equation}
\end{definition}

\begin{remark}
By taking adjoints, the condition \eqref{f: M1M2comp} is equivalent to
\begin{equation}\label{f: M1M2comp_adj}
\Phi_\pi^* \cO = \tilde\cO \Phi_\pi^*   \qquad\text{and}\qquad \Phi_\pi^* \cO^* = \tilde\cO^* \Phi_\pi^* \qquad\text{on } L^1(m_1),
\end{equation}
where $\Phi_\pi^* : L^1(m_1) \to L^1(m_2)$ is the adjoint of the lifting operator $\Phi_\pi : L^\infty(m_2) \to L^\infty(m_1)$, i.e., the fibre-averaging operator
\begin{equation}\label{f: adj_lifting}
\Phi_\pi^* h(x) = \frac{1}{m_2(x)} \sum_{\overline{x} \in \pi^{-1}\{x\}} h(\overline{x}) \, m_1(\overline{x}) .
\end{equation}
\end{remark}

\begin{prop}\label{p: K2K1}
Let $\pi : (T_1,m_1) \to (T_2,m_2)$ be a flow submersion and let $\cO \in \Bdd(m_1)$ be $\pi$-compatible. Then $\Phi_\pi(L^\infty(m_2))$ is both $\cO$- and $\cO^*$-invariant.
Moreover, the operator $\tilde\cO \in \Bdd(m_2)$ satisfying \eqref{f: M1M2comp} is uniquely determined by $\cO$, and
\[
\|\tilde \cO\|_{1 \to 1} \leq \|\cO\|_{1 \to 1}, \qquad \|\tilde \cO\|_{\infty \to \infty} \leq \|\cO\|_{\infty \to \infty}.
\]
Furthermore, for all $x,y \in T_2$,
\begin{align}
\label{f: rel ker} 
K_{\tilde \cO}(x,y) &= \frac{1}{m_2(y)} \sum_{z \in \pi^{-1}\{y\}} K_{\cO}(\overline{x},z) \, m_1(z) \qquad \forall \overline{x}\in\pi^{-1}\{x\}, \\ 
\label{f: rel ker2}
K_{\tilde \cO}(x,y) &= \frac{1}{m_2(x)}\sum_{z \in \pi^{-1}\{x\}} K_{\cO}(z,\overline{y}) \, m_1(z) \qquad \forall \overline{y}\in\pi^{-1}\{y\},
\end{align}
where the sums converge absolutely.
\end{prop}
\begin{proof}
As $\cO$ is $\pi$-compatible, the $\cO$- and $\cO^*$-invariance of $\Phi_\pi(L^\infty(m_2))$ is clear from \eqref{f: M1M2comp}. Moreover, as $\Phi_\pi : L^\infty(m_2) \to L^\infty(m_1)$ is an isometric embedding, from the first equality in \eqref{f: M1M2comp} we deduce that $\tilde \cO = \Phi_{\pi}^{-1} \cO \Phi_\pi$ on $L^\infty(m_2)$, so $\tilde\cO$ is uniquely determined by $\cO$, and clearly $\|\tilde \cO\|_{\infty \to \infty} \leq \|\cO\|_{\infty \to \infty}$.

In addition, for any $f \in L^\infty(m_2)$, $x \in T_2$ and $\overline{x} \in \pi^{-1}\{x\}$,
\[
\tilde\cO f(x) = \cO \Phi_\pi f(\overline{x})  = \sum_{z \in T_1} K_{\cO}(\overline{x},z) \, f(\pi(z)) \, m_1(z),
\]
and a rearrangement of this expression readily shows that the integral kernel of $\tilde\cO$ is given by 
\eqref{f: rel ker}. Notice that all the above sums converge absolutely, as
\[
\sum_{y \in T_2} |K_{\tilde\cO}(x,y)| \, m_2(y) \leq \sum_{z \in T_1} |K_{\cO}(\bar x,z)| \, m_1(z) \leq \|\cO\|_{\infty \to \infty} < \infty.
\]

The remaining assertions follow by repeating the above argument with $\cO^*$ in place of $\cO$.
\end{proof}

\begin{definition}
Let $\pi : (T_1,m_1) \to (T_2,m_2)$ be a flow submersion.
We denote by $\Comp(\pi)$ the set of the $\pi$-compatible operators in $\Bdd(m_1)$. Moreover, for any $\cO \in \Comp(\pi)$, we denote by $\pi(\cO)$ the unique operator $\tilde \cO$ satisfying \eqref{f: M1M2comp}. Furthermore, we set $\Comp_{\fin}(\pi) = \Bdd_{\fin}(m_1) \cap \Comp(\pi)$, and denote by $\overline{\Comp_{\fin}(\pi)}$ the closure of $\Comp_{\fin}(\pi)$ with respect to the norm of $\Bdd(m_1)$.
\end{definition}

In light of Proposition \ref{p: K2K1} and the inequality \eqref{f: 1-Lip}, it is easy to verify the following result. 

\begin{prop}\label{p: Cpi}
Let $\pi : (T_1,m_1) \to (T_2,m_2)$ be a flow submersion. Then
$\Comp(\pi)$ is a closed unital $*$-subalgebra of $\Bdd(m_1)$, and the mapping
\[
\pi : \Comp(\pi) \to \Bdd(m_2)
\]
is a unital Banach $*$-algebra homomorphism of norm $1$. Furthermore, $\Comp_{\fin}(\pi)$  and its closure $\overline{\Comp_{\fin}(\pi)}$ are unital $*$-subalgebras of $\Comp(\pi)$, and $\pi(\Comp_{\fin}(\pi)) \subseteq \Bdd_{\fin}(m_2)$.
\end{prop}

The following result provides some notable examples of $\pi$-compatible operators and shows why their theory is relevant to the study of the flow gradient and the flow Laplacian. Recall from Section \ref{ss: noncommutative} the definition of a joint functional calculus for noncommuting operators via power series.

\begin{prop}\label{p: Intral}
Let $\pi : (T_1,m_1) \to (T_2,m_2)$ be a flow submersion. Let $\Sigma_j$ denote the shift operator on $T_j$ for $j=1,2$. Then, for any $F \in \Hol(2,1)$, we have $F(\Sigma_1,\Sigma_1^*) \in \overline{\Comp_\fin(\pi)}$ and
\[
\pi(F(\Sigma_1,\Sigma_1^*)) = F(\Sigma_2,\Sigma_2^*).
\]
Moreover, if $F \in \Pol(2)$, then $F(\Sigma_1,\Sigma_1^*) \in \Comp_{\fin}(\pi)$.
\end{prop}
\begin{proof}
Thanks to Proposition \ref{p: Cpi} and the absolute convergence in $\Bdd(m_1)$ of the series $F(\Sigma_1,\Sigma_1^*)$ for any $F \in \Hol(2,1)$, the above statement is an immediate consequence of the assertion that $\Sigma_1 \in \Comp_\fin(\pi)$ and $\pi(\Sigma_1) = \Sigma_2$, which we now proceed to prove.

Notice that, by \eqref{f: 1}, for all $x \in T_1$,
\[
\Sigma_1(f\circ \pi)(x)=f(\pi(\pred(x)))=f(\pred(\pi(x)))=(\Sigma_2f) \circ \pi(x),
\]
that is $\Sigma_1 \Phi_\pi = \Phi_\pi \Sigma_2$. Moreover, by \eqref{f: 2},
\begin{multline*}
\Sigma_1^*(f \circ \pi) (z)
= \sum_{w \in \succ(z)}f(\pi(w)) \frac{m_1(w)}{m_1(z)}
= \sum_{y \in \succ(\pi(z))}f(y) \frac{m_1(\succ(z) \cap \pi^{-1}\{y\})}{m_1(z)}\\
= \sum_{y \in \succ(\pi(z))} f(y) \frac{m_2(y)}{m_2(\pi(z))}
= (\Sigma_2^*f) \circ \pi (z),
\end{multline*}
where the third equality follows by applying \eqref{f: daga} to any $x \in \pi^{-1}\{y\} \cap \succ(z)$. This shows that $(\Sigma_1)^* \Phi_\pi = \Phi_\pi (\Sigma_2)^*$, thus $\Sigma_1 \in \Comp(\pi)$ and $\pi(\Sigma_1) = \Sigma_2$. As clearly $\Sigma_1 \in \Bdd_{\fin}(m_1)$, this completes the proof of the above assertion.
\end{proof}  

\subsection{Transference to flow quotients}\label{ss: transference}

Let $\pi$ be a flow submersion.
From Proposition \ref{p: K2K1} we know that
\begin{equation}\label{eq:L1Linfty_transference}
\|\pi(\cO)\|_{1 \to 1} \leq \|\cO\|_{1 \to 1} \qquad\text{and}\qquad \|\pi(\cO)\|_{\infty \to \infty} \leq \|\cO\|_{\infty \to \infty}
\end{equation}
for any $\pi$-compatible operator $\cO$, i.e., $L^1$ and $L^\infty$ bounds transfer from $\cO$ to $\pi(\cO)$. The relation between the integral kernels of $\cO$ and $\pi(\cO)$ actually allows us to prove a weighted variant of these transference estimates.

\begin{prop}\label{p: stimeK2K1}
Let $\pi:(T_1,m_1) \to (T_2,m_2)$ be a level-preserving flow submersion and $\cO \in \Comp(\pi)$. Then, for every weight $w:\NN \times \ZZ\times \ZZ \to [0,\infty)$ which is increasing in the first variable,
\begin{equation}\label{f: weightedestimateK1K2}
\begin{split}
&\sup_{y \in T_2} \sum_{x \in T_2} w(d_2(x,y),\ell(x),\ell(y)) \, |K_{\pi(\cO)}(x,y)| \, m_2(x)\\
&\le \sup_{y \in T_1} \sum_{x \in T_1} w(d_1(x,y),\ell(x),\ell(y)) \, |K_\cO(x,y)| \, m_1(x),
\end{split}
\end{equation}
and
\begin{equation}\label{f: weightedestimateK1K2bis}
\begin{split}
&\sup_{x \in T_2} \sum_{y \in T_2} w(d_2(x,y),\ell(x),\ell(y)) \, |K_{\pi(\cO)}(x,y)| \, m_2(y)\\
&\le \sup_{x \in T_1} \sum_{y \in T_1} w(d_1(x,y),\ell(x),\ell(y)) \, |K_\cO(x,y)| \, m_1(y).
\end{split}
\end{equation}
\end{prop}
\begin{proof}
By \eqref{f: rel ker2} and \eqref{f: 1-Lip},
\[\begin{split}
&\sum_{x \in T_2}  w(d_2(x,y),\ell(x),\ell(y)) \, |K_{\pi(\cO)}(x,y)| \, m_2(x)  \\
&\le \sum_{x \in T_2} w(d_2(  {x},y),\ell(x),\ell(y))  \sum_{z \in \pi^{-1}\{x\}} |K_\cO(z,\overline{y})| \, m_1(z) \\ 
&\leq \sum_{x \in T_1}  w(d_1( {x}, \overline{y}),\ell(x),\ell(\overline{y})) \, |K_\cO(x, \overline{y})| \, m_1(x)  ,
\end{split}\]
for any $\overline{y}\in\pi^{-1}\{y\}$, where we used that $w$ is increasing in the first variable. Taking the supremum over all $y\in T_2$ gives \eqref{f: weightedestimateK1K2}. The estimate \eqref{f: weightedestimateK1K2bis} is proved analogously, using \eqref{f: rel ker} and \eqref{f: 1-Lip}.
\end{proof}

Finally, we show that an $L^p$ variant of the transference estimates \eqref{eq:L1Linfty_transference} holds for a subclass of $\pi$-compatible operators $\cO$. The following result can be thought of as an $L^p$ transference result, which may be compared, e.g., to those in \cite{BPW,CW} for actions of amenable groups.

\begin{prop}\label{p: Lptransf}
Let $\pi : (T_1,m_1) \to (T_2,m_2)$ be a flow submersion and $\cO \in \overline{\Comp_\fin(\pi)}$. Then, for all $p \in [1,\infty]$,
\[
\|\pi(\cO)\|_{p \to p} \leq \|\cO\|_{p \to p}.
\]
\end{prop}
\begin{proof}
Thanks to \eqref{f: BddLp} and the boundedness of $\pi : \Comp(\pi) \to \Bdd(m_2)$, it is enough to prove the assertion under the assumption that $\cO \in \Comp_\fin(\pi)$, i.e.,
\begin{equation}\label{f: Cfin_ker}
N \defeq \sup \{ d_1(x,y) \colon K_{\cO}(x,y) \neq 0\} < \infty.
\end{equation}
In addition, in light of \eqref{eq:L1Linfty_transference}, it is enough to consider $p \in (1,\infty)$. As $c_{00}(T_2)$ is dense in $L^p(m_2)$, we only need to check that
\begin{equation}\label{f: Lptrest}
\|\pi(\cO) f\|_{L^p(m_2)} \leq \|\cO\|_{p \to p} \|f\|_{L^p(m_2)}
\end{equation}
for all $f \in c_{00}(T_2)$.

If $S$ is a subset of $T_j$, let us write $c_{00}(S)$ for the set of the functions in $c_{00}(T_j)$ supported in $S$. Recall moreover the notation $\Delta_x$ from \eqref{f: tenda}, and define, for any $x \in T_j$,
\[
\Delta_x^N = \{ y \in \Delta_x \colon d_j(x,y) \geq N \} = \bigcup_{y \in \succ^N(x)} \Delta_y.
\]
Then clearly
\[
c_{00}(T_2) = \bigcup_{z \in T_2} c_{00}(\Delta_z) = \bigcup_{z \in T_2} c_{00}(\Delta^N_z) = \bigcup_{w \in T_1} c_{00}(\Delta^N_{\pi(w)}),
\]
where we used the surjectivity of $\pi : T_1 \to T_2$. So we are reduced to proving that, for any $w \in T_1$, the estimate \eqref{f: Lptrest} holds for all $f \in c_{00}(\Delta^N_{\pi(w)})$.

Let us now fix $w \in T_1$. As $L^p$ operator norms are unchanged if the underlying measure is scaled, by appropriately scaling the flow measures we may assume that $m_2(\pi(w)) = m_1(w)$ when proving \eqref{f: Lptrest} for all $f \in c_{00}(\Delta^N_{\pi(w)})$. Under this assumption, Proposition \ref{p: pushforward} then tells us that the measure $m_2|_{\Delta_{\pi(w)}}$ is the push-forward via $\pi|_{\Delta_w}$ of the measure $m_1|_{\Delta_w}$. In particular, for all $g \in c_{00}(\Delta_{\pi_{w}})$ and all $q \in [1,\infty]$,
\[
\|\chr_{\Delta_w} \Phi_\pi g\|_{L^q(m_1)} = \| g \|_{L^q(m_2)},
\]
and moreover
\begin{equation}\label{f: inv_lifting}
\Phi_\pi^* \chr_{\Delta_w} \Phi_\pi g = g,
\end{equation}
where $\Phi_\pi^* : L^1(m_1) \to L^1(m_2)$ is as in \eqref{f: adj_lifting}.
Furthermore, for all $h \in c_{00}(\Delta_{w})$ and $q \in [1,\infty]$,
\[
\|\Phi_\pi^* h\|_{L^q(m_2)} \leq \|h\|_{L^q(m_1)}.
\]
Finally, by \eqref{f: M1M2comp_adj},
\[
\Phi_\pi^* \cO = \pi(\cO) \Phi_\pi^* \qquad\text{on } L^1(m_1).
\]

Given now $f \in c_{00}(\Delta_{\pi(w)}^N)$, by \eqref{f: inv_lifting} we can write
\[
f = \Phi_\pi^* \chr_{\Delta_w} \Phi_\pi f,
\]
and moreover, as $\pi$ is a submersion, $\supp(\chr_{\Delta_w} \Phi_\pi f) \subseteq \Delta_w^N$. From \eqref{f: Cfin_ker} we then deduce that $\cO \chr_{\Delta_w} \Phi_\pi f \in c_{00}(\Delta_w)$, so
\begin{multline*}
\|\pi(\cO) f\|_{L^p(m_2)} = \|\pi(\cO) \Phi_\pi^* \chr_{\Delta_w} \Phi_\pi f\|_{L^p(m_2)} = \|\Phi_\pi^* \cO \chr_{\Delta_w} \Phi_\pi f\|_{L^p(m_2)} \\
\leq \|\cO \chr_{\Delta_w} \Phi_\pi f\|_{L^p(m_1)} \leq \|\cO\|_{p \to p} \|\chr_{\Delta_w} \Phi_\pi f\|_{L^p(m_1)} = \|\cO\|_{p \to p} \|f\|_{L^p(m_2)},
\end{multline*}
as desired.
\end{proof}

\subsection{Uniformly rational flows}\label{s: rationalflows}

The results of Section \ref{ss: transference} show that weighted kernel estimates and $L^p$ bounds for operators related to the flow gradient and the flow Laplacian can be transferred from a flow tree to any of its flow quotients. It is therefore of interest to know when a given flow tree is the flow quotient of another. In this section, we characterise the class of flow quotients of homogeneous trees.

\begin{definition}
Let $(T,m)$ be a flow tree. We say that $m$ is a \emph{uniformly rational flow measure} if there exists an integer $q \in \Npos$ such that
\[
q\frac{m(x)}{m(\pred(x))} \in \Npos \qquad\forall x \in T.
\]
In this case, we shall refer to $m$ as a \emph{$q$-uniformly rational flow measure}, and correspondingly $(T,m)$ will be said a \emph{$q$-uniformly rational flow tree}.
\end{definition} 

\begin{remark}\label{r: rational_bdddeg}
If $T$ admits a $q$-uniformly rational flow measure, then $q(x) \le q$ for every $x \in T$. So, in this case, $T$ has bounded degree.
\end{remark}

\begin{remark}\label{r: rational_constraint}
From \eqref{f: can_flow} it is clear that the homogeneous tree $(\Tq,\mq)$ is $q$-uniformly rational. Moreover, from the $\pi$-compatibility condition \eqref{f: daga} it follows that any flow quotient of a $q$-uniformly rational flow tree is also $q$-uniformly rational. Therefore, if a flow tree $(T,m)$ is a flow quotient of the homogeneous tree $(\Tq,\mq)$, then $m$ is $q$-uniformly rational.
\end{remark}

As we shall see, being $q$-uniformly rational is not only necessary, but also a sufficient condition for a flow tree to be a flow quotient of $(\Tq,\mq)$.

We start with an auxiliary result, which provides, for any tree $T$ with root at infinity, a convenient way to enumerate the successors of any given vertex of $T$; this will be of use when constructing a submersion with image $T$.

\begin{prop}\label{p: enumerator}
Let $T$ be a tree with root at infinity.
Then, there exists a function $\ord : T \to \NN$ 
such that
\begin{equation}\label{f: enumerator}
\begin{gathered}
\ord(\succ(x)) =\{0,\dots,q(x)-1\} \qquad\forall x \in T,\\
\lim_{k \to \infty} \ord(\pred^{k}(x))=0 \qquad\forall x \in T.
\end{gathered}
\end{equation}
Furthermore, for any given vertex $x_0 \in T$, there exists such a function $\ord$ with the additional property that
\[
\ord(\pred^k(x_0)) = 0 \qquad \forall k \in \NN.
\]
\end{prop}
\begin{proof}
Let $\pred^*(x_0) = \{\pred^n(x_0)\}_{n \in \NN}$. Then, for all $x \in T$, the set $\pred^*(x_0)$ intersects $\succ(x)$ in at most one point. Therefore, for any $x \in T$, we can choose a bijection $\ord_x : \succ(x) \to \{0,\dots,q(x)-1\}$ in such a way that $\ord_x(z) = 0$ if $z \in \succ(x) \cap \pred^*(x_0)$. Gluing together all the functions $\ord_x$ yields a function $\ord : T \to \NN$ with the desired properties. Indeed, for any $x \in T$, the vertices $x$ and $x_0$ have a common ancestor, thus $\pred^k(x) \in \pred^*(x_0)$ for all sufficiently large $k \in \NN$, and therefore $\ord(\pred^k(x)) = 0$ for all $k$ sufficiently large.
\end{proof}

\begin{definition}
Let $T$ be a tree with root at infinity. A function $\ord : T \to \NN$ with the properties \eqref{f: enumerator} will be referred to as an \emph{enumerator} of $T$.
\end{definition}

\begin{remark}
If $\ord$ is an enumerator of $T$, then the set
\[
\Gamma_\ord \defeq \{ x \in T \colon \ord(\pred^{(k)}(x)) = 0 \ \forall k \in \NN \}
\]
is a bi-infinite geodesic in $T$ with the root $\myth$ as one endpoint. Proposition \ref{p: enumerator} therefore tells us that, for any given $x_0 \in T$, we can find an enumerator $\ord$ of $T$ such that $x_0 \in \Gamma_{\ord}$; a simple variation of the proof would actually allow us to construct an enumerator $\ord$ so that $\Gamma_{\ord}$ is any prescribed bi-infinite geodesic with endpoint $\myth$.
\end{remark}

The next result, together with Remark \ref{r: rational_constraint}, shows that being a flow quotient of $(\Tq,\mq)$ is a characterisation of $q$-uniformly rational flow trees.

\begin{prop}\label{p: comp unif rat}
Let $m$ be a $q$-uniformly rational flow measure on a tree $T$. Then $T$ is a flow quotient of the homogeneous tree $\Tq$ equipped with the canonical flow. More precisely, for any given $\overline{w_0} \in \Tq$ and $w_0 \in T$, we can find a flow submersion $\pi : \Tq \to T$ such that $\pi(\overline{w_0}) = w_0$.
\end{prop}
\begin{proof}
By Proposition \ref{p: enumerator}, we can find an enumerator $\ord$ of $T$ such that $w_0 \in \Gamma_{\ord}$. In particular, for all $x \in T$, we can write
\[
\succ(x) = \{ \succ_0(x),\dots,\succ_{q(x)-1}(x)\},
\]
where $\succ_j(x) \in \succ(x)$ is uniquely determined by $\ord(\succ_j(x)) = j$. We now modify this enumeration of the elements of $\succ(x)$, by repeating each element a number of times proportional to its relative $m$-measure within $\succ(x)$. As the relative measures $m(y)/m(x)$ for $y \in \succ(x)$ are all rational numbers with common denominator $q$ and add up to one, we can construct such a noninjective enumeration of $\succ(x)$ as a list of length $q$. In other words, there exist functions $\tilde \succ_j : T \to T$ for $j=0,\dots,q-1$ such that
\begin{equation}\label{f: noninj_enum}
\succ(x)=\{\tilde\succ_0(x), \dots,\tilde\succ_{q-1}(x)\} \qquad\forall x \in T
\end{equation}
and moreover, for all $x \in T$,
the function $j \mapsto \phi_x(j) \defeq \ord(\tilde\succ_j(x))$ is increasing and 
\begin{equation}\label{f: asterisco}
\#\phi_x^{-1}\{\ord(y)\}=q \frac{m(y)}{m(x)} \qquad \forall y\in \succ(x).
\end{equation} 
In particular, as $\phi_x$ is increasing,
\begin{equation}\label{f: ord_zero}
\ord(\tilde\succ_0(x)) = 0 \qquad \forall x \in T.
\end{equation}

Let now $\ord_q$ be an enumerator of $\Tq$ such that $\overline{w_0} \in \Gamma_{\ord_q}$.
We now claim that there exists a unique map $\pi: \Tq \to T$ such that $\pi(\overline{w_0}) = w_0$ and
\begin{equation}\label{f: pi}
\pi(x)=\tilde \succ_{\ord_q(x)}(\pi(\pred(x))) \qquad \forall x \in \Tq
\end{equation}
Indeed, if $x \in \Tq$ is such that $\ord_q(x) = 0$, then \eqref{f: ord_zero} and \eqref{f: pi} imply that $\ord(\pi(x)) = 0$; in other words, $\pi(\Gamma_{\ord_q}) \subseteq \Gamma_\ord$. As the bi-infinite geodesics $\Gamma_{\ord}$ and $\Gamma_{\ord_q}$ in $T$ and $\Tq$ contain one element for each level, there is only one way to define a map $\pi : \Gamma_{\ord_q} \to \Gamma_{\ord}$ in such a way that $\pi(\overline{w_0}) = w_0$ and $\pi$ preserves the predecessor-successor relation, i.e., so that \eqref{f: pi} is satisfied for all $x \in \Gamma_{\ord_q}$. Clearly such a map satisfies $\pi(\Gamma_{\ord_q}) = \Gamma_{\ord}$. We can now extend $\pi$ to the whole $\Tq$ by iterated applications of \eqref{f: pi}: more precisely, we can write $\Tq$ as the increasing union
\[
\Tq = \bigcup_{n \in \NN} \Gamma_{\ord_q}^n, \qquad \Gamma_{\ord_q}^n \defeq \{ x \in \Tq \colon d_{\Tq}(x,\Gamma_{\ord_q}) \leq n\} = \bigcup_{x \in \Gamma_{\ord_q}} \succ^n(x),
\]
and then inductively extend $\pi$ to each of the sets $\Gamma_{\ord_q}^n$, by using \eqref{f: pi} and the fact that $\pred(\Gamma_{\ord_q}^{n+1}) = \Gamma_{\ord_q}^{n}$. This completes the proof of the claim.

Now, from \eqref{f: pi} it is clear that $\pred(\pi(x)) = \pi(\pred(x))$; moreover, from \eqref{f: pi} and \eqref{f: noninj_enum} we also deduce that, for all $z \in \Tq$,
\[
\pi(\succ(z))=\{\pi(x) \colon x \in \succ(z)\}=\{\tilde\succ_j(\pi(z)) \colon j=0,\dots,q-1\}=\succ(\pi(z)),
\]
namely $\pi : \Tq \to T$ is a submersion.

It remains to check that $m$ and the canonical flow $\mq$ are $\pi$-compatible. For any $x \in \Tq$, by \eqref{f: asterisco},
\[\begin{split}
q \frac{m(\pi(x))}{m(\pred(\pi(x)))}
&= \#\phi^{-1}_{\pred(\pi(x))}\{\ord(\pi(x))\}\\ 
&= \#\{j\in\{0,\dots,q-1\} \colon \ord(\tilde \succ_j(\pred(\pi(x))))=\ord(\pi(x))\}\\ 
&= \#\{j\in\{0,\dots,q-1\} \colon \tilde \succ_j(\pred(\pi(x)))=\pi(x)\}\\ 
&= \#\{y \in \succ(\pred(x)) \colon \tilde \succ_{\ord_q(y)}(\pred(\pi(x)))=\pi(x)\}\\ 
&= \#\{y \in \succ(\pred(x)) \colon \pi(y)=\pi(x)\} \\ 
&= \#(\pi^{-1}\{\pi(x)\} \cap \succ(\pred(x))),
\end{split}\]
where in the second last equality we used that $\{\ord_q(y)\}_{y \in \succ(\pred(x))}=\{0,\dots,q-1\}$  and \eqref{f: pi}. Thus,
\[
\frac{m(\pi(x))}{m(\pred(\pi(x)))}
= \frac{\#(\pi^{-1}\{\pi(x)\} \cap \succ(\pred(x)))}{q}
= {\frac{\mq(\pi^{-1}\{\pi(x)\} \cap \succ(\pred(x)))}{\mq(\pred(x))}},
\]
as desired.
\end{proof}

\section{Perturbation of flow measures}\label{s: perturbation}

The construction in Section \ref{s: rationalflows} shows that any uniformly rational flow tree is a flow quotient of a homogeneous tree; in light of the results of Section \ref{s: submersions}, this means that a number of estimates for the flow gradient and the flow Laplacian on homogeneous trees may be transferred to analogous estimates on uniformly rational flow trees.

From Remark \ref{r: rational_constraint} we know that the uniform rationality constraint on $(T,m)$ is necessary in order for the flow tree to be a quotient of a homogeneous tree. On the other hand, this constraint is quite restrictive, as it rules out, e.g., any flow measure $m$ such that the ratio $m(x)/m(\pred(x))$ is irrational for some $x \in T$.

We now show how a perturbative argument can be used, in some cases, to get rid of such rationality constraint and obtain a sort of transference result that applies to any flow tree.
The key idea is that any flow measure on a tree $T$ can be approximated by uniformly rational flow measures on suitable subtrees of $T$, and moreover the estimates we are interested in are preserved by this approximation process. Of course, in order to be able to approximate a flow measure $m$ with irrational ratios $m(x)/m(\pred(x))$, we will need to use $q$-uniformly rational measures with $q$ larger and larger; as a consequence, we will be able to transfer to an arbitrary flow tree $(T,m)$ only those estimates that hold on $(\Tq,\mq)$ \emph{uniformly in $q$}.

\subsection{A perturbative argument}
We start by presenting a perturbative argument, showing that many estimates for the joint functional calculus of $(\Sigma,\Sigma^*)$ are preserved under pointwise convergence of the underlying flow measure. Similar arguments have been used in different contexts for the purpose of transplanting $L^p$ estimates (see, e.g., \cite{KST}, \cite[Theorem 5.2]{AMsphere} or \cite[Lemma 2.3]{CMMP}).

\begin{definition}
Let $T$ be a tree with root at infinity. A \emph{$\pred$-subtree} of $T$ is a subset $S$ of $T$ such that, if $x \in S$, then $\pred(x) \in S$ too. With the structure induced by $T$, any such $S$ is also a tree with root at infinity.
\end{definition}

\begin{definition}
Let $(T,m)$ be a flow tree. An \emph{approximating sequence} for $(T,m)$ is a sequence $((T_j,m_j))_j$ such that:
\begin{enumerate}[label=(\alph*)]
\item the $T_j$ form an increasing sequence of $\pred$-subtrees of $T$, and $T = \bigcup_j T_j$;
\item each $m_j$ is a flow measure on $T_j$;
\item if we extend by zeros each $m_j$ to a function on the whole $T$, then $m_j \to m$ pointwise on $T$.
\end{enumerate} 
\end{definition}

Let $((T_j,m_j))_j$ be an approximating sequence of a flow tree $(T,m)$. We shall denote by $\Sigma_j,\nabla_j,\opL_j$ the shift operator, flow gradient and flow Laplacian on $(T_j,m_j)$, while $\Sigma,\nabla,\opL$ denote the corresponding operators on $(T,m)$.

Recall that $c_{00}(T)$ denotes the space of finitely supported functions on $T$. We shall identify any $f \in c_{00}(T_j)$ with its extension by zeros to $T$; in this way, $c_{00}(T_j)$ is the subspace of $c_{00}(T)$ of the functions supported in $T_j$. As the $T_j$ are an increasing sequence with $T = \bigcup_j T_j$, any function $f \in c_{00}(T)$ also belongs to $c_{00}(T_j)$ for sufficiently large $j$.

We say that a $c_{00}(T)$-valued sequence $(f_j)_j$ converges to $f$ in $c_{00}(T)$ if $f_j \to f$ pointwise and $\bigcup_{j} \supp(f_j)$ is finite. In this case, we shall write $f_j \convcc f$.

Notice that, if $f_j \convcc f$, then $\bigcup_{j} \supp(f_j)$ is contained in $T_k$ for any sufficiently large $k$. In particular, $f_j \in c_{00}(T_j)$ for all $j$ large enough. Therefore, if $\cO_j \in \Bdd_\fin(T_j)$, then $\cO_j f_j \in c_{00}(T_j) \subseteq c_{00}(T)$ is well defined for any $j$ large enough, and we can construct a new sequence $(\cO_j f_j)_j$ in $c_{00}(T)$; while this sequence is only defined for $j$ large enough, this will not be a problem for our discussion, as we shall only be interested in asymptotic properties as $j \to \infty$.

Our perturbative argument for approximating sequences of flow trees is encoded in the following two statements.

\begin{prop}\label{p: conv_hker_basic}
Let $(T,m)$ be a flow tree.
Let $((T_j,m_j))_j$ be an approximating sequence.
Let $f,g \in c_{00}(T)$, and let $(f_j)_j,(g_j)_j$ be sequences in $c_{00}(T)$ such that $f_j \convcc f$ and $g_j \convcc g$. Then:
\begin{enumerate}[label=(\roman*)]
\item\label{en: conv_hker_a}
$\|f_j\|_{L^p(m_j)} \to \|f\|_{L^p(m)}$ for all $p\in [1,\infty]$;
\item\label{en: conv_hker_b}
$\langle f_j, g_j\rangle_{L^2(m_j)} \to \langle f,g\rangle_{L^2(m)}$;
\item\label{en: conv_hker_c}
$F(\Sigma_j,\Sigma_j^*) f_j \convcc F(\Sigma,\Sigma^*) f$ for any $F \in \Pol(2)$.
\end{enumerate}
\end{prop}
\begin{proof}  
Parts \ref{en: conv_hker_a} and \ref{en: conv_hker_b} and trivial.

As for part \ref{en: conv_hker_c}, it readily follows by iteration from the two particular cases
\begin{equation}\label{f: conv_pert_Sigma}
\Sigma_j f_j \convcc \Sigma f, \qquad \Sigma_j^* f_j \convcc \Sigma^* f.
\end{equation}
The first convergence in \eqref{f: conv_pert_Sigma} is clear, since by \eqref{f: shift} it follows that $\Sigma_j f_j(x) = f_j(\pred(x)) = \Sigma f_j(x)$ for all $j$ large enough. As for the second one, for any fixed $x \in T$, we have $\{x\} \cup \succ(x) \subseteq T_j$ for all $j$ large enough; thus, for any large enough $j$,
\[
\Sigma_j^* f_j(x) = \frac{1}{m_j(x)} \sum_{y \in \succ(x)} f_j(y) \, m_j(y),
\]
and the right-hand side clearly converges to $\Sigma^* f(x)$ as $j \to \infty$.
\end{proof}

\begin{prop}\label{p: conv_hker_adv}
Let $(T,m)$ be a flow tree.
Let $((T_j,m_j))_j$ be an approximating sequence.
Let $f,g \in c_{00}(T)$.
\begin{enumerate}[label=(\roman*)]
\item\label{en: conv_hker_d}
For every $F \in \Hol(2,1)$,
\begin{equation}\label{f: d1}
\langle F(\Sigma_j,\Sigma_j^*) f, g \rangle_{L^2(m_j)} \to \langle F(\Sigma,\Sigma^*) f , g \rangle_{L^2(m)}
\end{equation}
and
\begin{equation}\label{f: d2}
K_{F(\Sigma_j,\Sigma_j^*)} \to K_{F(\Sigma,\Sigma^*)}
\end{equation} 
pointwise on $T\times T$; here $K_{F(\Sigma_j,\Sigma_j^*)}$ denotes the integral kernel with respect to $m_j$, extended by zeros to the whole $T \times T$, while $K_{F(\Sigma,\Sigma^*)}$ denotes the integral kernel with respect to $m$.
\item\label{en: conv_hker_e}
For every $F \in C[0,2]$,
\[
\langle F(\opL_j) f, g \rangle_{L^2(m_j)} \to \langle F(\opL) f , g \rangle_{L^2(m)}
\]
and
\[
K_{F(\opL_j)} \to K_{F(\opL)}
\]
pointwise on $T\times T$; here $K_{F(\opL_j)}$ denotes the integral kernel with respect to $m_j$, extended by zeros to the whole $T \times T$, while $K_{F(\opL)}$ denotes the integral kernel with respect to $m$.
\end{enumerate}
\end{prop}
\begin{proof}
Notice first that the convergence \eqref{f: d1} follows from Proposition \ref{p: conv_hker_basic} whenever $F \in \Pol(2)$. We shall now extend this result to any $F \in \Hol(2,1)$. Recall first from Section \ref{ss: noncommutative} that
\begin{equation}\label{f: norm_est}
\|F(\Sigma_j,\Sigma_j^*)\|_{L^2(m_j) \to L^2(m_j)} \leq \|F(\Sigma_j,\Sigma_j^*)\|_{\Bdd(m_j)} \leq \|F\|_{(1)} \qquad\forall F \in \Hol(2,1).
\end{equation}
Define now
\[
A = \{ F \in \Hol(2,1) \colon \langle F(\Sigma_j,\Sigma_j^*) f,g \rangle_{L^2(m_j)} \to \langle F(\Sigma,\Sigma^*) f,g \rangle_{L^2(m)} \ \forall f,g \in c_{00}(T) \}.
\]
So, proving \eqref{f: d1} is the same as proving that $A = \Hol(2,1)$. As we already know that $\Pol(2) \subseteq A$, and $\Pol(2)$ is dense in $\Hol(2,1)$, it will be enough to show that $A$ is a closed linear subspace of $\Hol(2,1)$.

It is straightforward to check that $A$ is a linear subspace of $\Hol(2,1)$. In addition, if $F_n \in A$ and $F_n \to F$ in $\Hol(2,1)$, then, for all $f,g \in c_{00}(T)$, by \eqref{f: norm_est},
\[\begin{split}
&|\langle F(\Sigma_j,\Sigma_j^*) f,g \rangle_{L^2(m_j)} - \langle F(\Sigma,\Sigma^*) f,g \rangle_{L^2(m)}| \\
 &\leq C_{f,g} \|F_n-F\|_{(1)}  + |\langle F_n(\Sigma_j,\Sigma_j^*) f,g \rangle_{L^2(m_j)} - \langle F_n(\Sigma,\Sigma^*) f,g \rangle_{L^2(m)}|
\end{split}
\]
where $C_{f,g} = \sup_j \|f\|_{L^2(m_j)} \|g\|_{L^2(m_j)} + \|f\|_{L^2(m)} \|g\|_{L^2(m)} < \infty$ by Proposition \ref{p: conv_hker_basic}\ref{en: conv_hker_a}; from this estimate it is easy to conclude that $\langle F(\Sigma_j,\Sigma_j^*) f,g \rangle_{L^2(m_j)} \to \langle F(\Sigma,\Sigma^*) f,g \rangle_{L^2(m)}$, and therefore, by the arbitrariness of $f,g \in c_{00}(T)$, that $F \in A$ too. This shows that $A$ is closed in $\Hol(2,1)$, as desired.

Finally, \eqref{f: d2} readily follows by taking $f=\chr_{\{y\}}$ and $g=\chr_{\{x\}}$ in \eqref{f: d1} for any $x,y \in T$. This completes the proof of part \ref{en: conv_hker_d}.

The proof of part \ref{en: conv_hker_e} is similar. Here one considers instead the set
\[
B = \{ F \in C[0,2] \colon \langle F(\opL_j) f,g \rangle_{L^2(m_j)} \to \langle F(\opL) f,g \rangle_{L^2(m)} \ \forall f,g \in c_{00}(T) \}
\]
and shows that $B$ is a closed linear subspace of $C[0,2]$; in place of \eqref{f: norm_est}, one can use the estimate
\[
\|F(\opL_j)\|_{L^2(m_j) \to L^2(m_j)} \leq \|F\|_{\infty} \qquad\forall F \in C[0,2]
\]
from the Borel functional calculus. As we already know from the above discussion that $B$ contains all polynomials, the Stone--Weierstrass Theorem implies that $B = C[0,2]$.
\end{proof}

From Propositions \ref{p: conv_hker_basic} and \ref{p: conv_hker_adv} we finally deduce the following crucial result.

\begin{corollary}\label{c: conv_hker}
Let $(T,m)$ be a flow tree, and let $((T_j,m_j))_j$ be an approximating sequence. 
\begin{enumerate}[label=(\roman*)]
\item\label{en: c_conv_hker_Sigma} Let $F\in \Hol(2,1)$. Then, for all $p \in [1,\infty]$,
\begin{equation}\label{f: Lp_pert_Sigma}
\|F(\Sigma,\Sigma^*)\|_{L^p(m) \to L^p(m)} \leq \liminf_{j \to \infty} \|F(\Sigma_j,\Sigma_j^*)\|_{L^p(m_j) \to L^p(m_j)}.
\end{equation}
Moreover, for every weight $w : T \times T \to [0,\infty)$,
\begin{equation}\label{f: weight_pert_Sigma}
\begin{split}
&\sup_{y \in T} \sum_{x \in T}  w(x,y) \, |K_{F(\Sigma,\Sigma^*)}(x,y)| \, m(x)  \\
&\le \liminf_{j \to \infty} \sup_{y \in T} \sum_{x \in T} w(x,y) \, |K_{F(\Sigma_j,\Sigma_j^*)}(x,y)| \, m_j(x).
\end{split}
\end{equation}
\item\label{en: c_conv_hker_opL} Let $F\in C[0,2]$. Then, for all $p \in [1,\infty]$,
\[
\|F(\opL)\|_{L^p(m) \to L^p(m)} \leq \liminf_{j \to \infty} \|F(\opL_j)\|_{L^p(m_j) \to L^p(m_j)}.
\]
Moreover, for every weight $w : T \times T \to [0,\infty)$,
\[\begin{split}
&\sup_{y \in T} \sum_{x \in T}  w(x,y) \, |K_{F(\opL)}(x,y)| \, m(x)  \\
&\le \liminf_{j \to \infty} \sup_{y \in T} \sum_{x \in T} w(x,y) \, |K_{F(\opL_j)}(x,y)| \, m_j(x).
\end{split}\]\end{enumerate}
\end{corollary} 
\begin{proof}
We only prove part \ref{en: c_conv_hker_Sigma}, as the proof of part \ref{en: c_conv_hker_opL} is analogous.

Notice that
\[
\| F(\Sigma,\Sigma^*) \|_{L^p(m)\to L^p(m)} = \sup_{\substack{f,g \in c_{00}(T) \\ \|f\|_{L^p(m)} = \|g\|_{L^{p'}(m)} = 1}} |\langle F(\Sigma,\Sigma^*) f,g \rangle_{L^2(m)}|,
\]
where $p'$ is the conjugate exponent to $p$; this is clear when $p<\infty$, and for $p=\infty$ it follows by duality from the case $p=1$, as $F(\Sigma,\Sigma^*)^*= \overline{F}(\Sigma^*,\Sigma)$.
Moreover, by Propositions \ref{p: conv_hker_basic} and \ref{p: conv_hker_adv},
\begin{multline*}
|\langle F(\Sigma,\Sigma^*) f,g \rangle_{L^2(m)}| = \lim_{j \to \infty} |\langle F(\Sigma_j,\Sigma_j^*) f,g \rangle_{L^2(m_j)}| \\
\leq \|f\|_{L^p(m)} \|g\|_{L^{p'}(m)}  \liminf_{j \to \infty} \|F(\Sigma_j,\Sigma_j^*)\|_{L^p(m_j) \to L^p(m_j)},
\end{multline*}
whence \eqref{f: Lp_pert_Sigma} follows.

In addition, as $K_{F(\Sigma_j,\Sigma_j^*)} \to K_{F(\Sigma,\Sigma*)}$ pointwise on $T \times T$ by Proposition \ref{p: conv_hker_adv}, the estimate \eqref{f: weight_pert_Sigma} is an immediate consequence of Fatou's Lemma.
\end{proof}

\subsection{Uniformly rational approximation}

We now show that whenever $(T,m)$ is a flow tree with bounded degree it is possible to approximate $m$ with a sequence of uniformly rational flow measures on $T$. The bounded degree assumption is needed in order for uniformly rational flow measures to exist on $T$ (see Remark \ref{r: rational_bdddeg}). 

\begin{prop}\label{p: point approx}
Let $T$ be a tree with bounded degree and let $m$ be a flow measure on $T$.
Then, there is a sequence of uniformly rational flow measures on $T$ converging pointwise to $m$ on $T$.
\end{prop}
\begin{proof}
By the bounded degree assumption, there exists $q_0 \in \Npos$ such that $q(x) \le q_0$ for every $x \in T$. We shall construct, for any $q \geq q_0$, a $q$-uniformly rational flow measure $m_q$ on $T$ so that $m_q \to m$ pointwise as $q \to \infty$.

Notice that, as $m$ is a flow measure on $T$, for every $x \in T$ we have a $q(x)$-tuple $\left(m(y)/m(x)\right)_{y \in \succ(x)}$ of strictly positive numbers whose sum is $1$. The first step in our construction is to approximate these $q(x)$-tuples with analogous tuples with rational entries with common denominator $q$.

For any $q \geq q_0$, define
\begin{equation}\label{f: cond x}
W_q = \left\{ x \in T \colon \min_{y \in \succ(x)} \frac{m(y)}{m(x)} \geq \frac{1}{q} \right\}.
\end{equation}
Clearly the sets $W_q$ grow with $q$ and $T = \bigcup_{q \geq q_0} W_q$, because $1/q \to 0$ as $q \to \infty$.

If $x \in W_q$, then we can pick any $y_x \in \succ(x)$ and define a $q(x)$-tupe $(w_q(y))_{y \in \succ(x)}$ by setting
\[
w_{q}(y) = \begin{cases} \frac{1}{q} \left\lfloor q \frac{m(y)}{m(x)}\right\rfloor &\text{if } y \in \succ(x)\setminus \{y_x\},\\
1-\sum_{z \in \succ(x) \setminus \{y_x\}} w_q(z) &\text{if } y=y_x.
\end{cases}
\]
By \eqref{f: cond x} we deduce that, for every $y \in \succ(x) \setminus \{y_x\}$, we have $w_q(y) > 0$, thus
\[
\frac{1}{q} \le w_q(y) \le  \frac{m(y)}{m(x)} \qquad \text{and} \qquad \left| \frac{m(y)}{m(x)}-w_q(y)\right| \le \frac{1}{q};
\]
in addition,
\[
1-w_{q}(y_x)=\sum_{y \in \succ(x)\setminus \{y_x\}} w_q(y) \leq \sum_{y \in \succ(x)\setminus \{y_x\}} \frac{m(y)}{m(x)} = 1-\frac{m(y_x)}{m(x)}\le 1-\frac{1}{q},
\]
so $w_{q}(y_x) \ge 1/q > 0$ too. Moreover, 
\[
\left|w_q(y_x)- \frac{m(y_x)}{m(x)}\right|\le \sum_{y \in \succ(x) \setminus \{y_x\}}  \left| w_q(y)- \frac{m(y)}{m(x)} \right| \le \frac{q(x)-1}{q}\le \frac{q_0-1}{q}.
\]
Thus, the $q(x)$-tuple $(w_q(y))_{y \in \succ(x)}$ is made of positive rational numbers with common denominator $q$ adding up to $1$. Moreover, as $q \to \infty$, the tuple $(w_q(y))_{y \in \succ(x)}$ converges to $(m(y)/m(x))_{y \in \succ(x)}$.

We can extend the definition of the tuple $(w_q(y))_{y \in \succ(x)}$ to all $x \in T$ and $q \geq q_0$, by picking, when $x \in T \setminus W_q$, any $q(x)$-tuple of positive rational numbers with common denominator $q$ and adding up to $1$. As $x \in W_q$ for any sufficiently large $q$, the convergence of $(w_q(y))_{y \in \succ(x)}$ to $(m(y)/m(x))_{y \in \succ(x)}$ as $q \to \infty$ remains true after the extension. In other words,
\begin{equation}\label{f: wj_conv}
\lim_{q \to \infty} w_q(y) = \frac{m(y)}{m(\pred(y))} \qquad\forall y \in T.
\end{equation}

Fix a vertex $o \in T$. We now define $m_q$ as the flow measure such that $m_q(o)=m(o)$ and 
\[
\frac{m_q(y)}{m_q(\pred(y))}=w_q(y) \qquad\forall y \in T.
\]
In particular, $m_q$ is a $q$-uniformly rational flow measure for every $q \geq q_0$. It remains to verify that $m_q \to m$ pointwise on $T$ as $q \to \infty$. Pick $x \in T$, and let $z \in T$ be a common ancestor of $x$ and $o$. By definition, 
\begin{equation}\label{f: mq_prod}
m_q(x)
= \left( \prod_{h=0}^{n-1} \frac{1}{w_q(\pred^h(o))}\right) \left( \prod_{k=0}^{m-1} w_q(\pred^k(x))\right) 
m(o),
\end{equation}
where $z=\pred^{m}(x)=\pred^n(o)$. Since $\{\pred^h(o)\}_{h < n} \cup \{\pred^k(x)\}_{k<m}$ is a finite set of vertices independent of $q$, from \eqref{f: wj_conv} and \eqref{f: mq_prod} we readily deduce that $m_q(x) \to m(x)$, as desired.
\end{proof}

Via a covering argument, we can now get rid of the bounded degree assumption and construct a uniformly rational approximating sequence for any flow tree.

\begin{corollary}\label{c: point approx}
Let $(T,m)$ be a flow tree. Then there exists an approximating sequence of $(T,m)$ made of uniformly rational flow trees.
\end{corollary}
\begin{proof}
Choose a vertex $o \in T$ and an enumerator $\ord$ such that $o \in \Gamma_{\ord}$. For any $x \in T$, let $\succ_0(x)$ denote the only successor of $x$ with $\ord(x) = 0$. Finally, let $\gamma_x = \{ \succ_0^k(x) \colon k \in \NN\}$ for any $x \in T$; in other words, $\gamma_x$ is the infinite descending geodesic starting from $x$ and where any subsequent vertex is the zeroth successor of the preceding one, according to $\ord$.

Then, for any $w \in T$ and $n \in \NN$, the set
\[
\tilde T_{w,n} = \{ \pred^{k}(w) \colon k \in \NN \} \cup \{ x \in T \colon x \leq w, \ d(x,w) \leq n \} \cup \bigcup_{x \in T \colon x \leq w, \, d(x,w) = n} \gamma_x
\]
is a $\pred$-subtree of $T$ of bounded degree, and
\[
\tilde m_{w,n}(x) = \begin{cases}
m(w) &\text{if } x \geq w,\\
m(x) &\text{if } x \leq w, \ d(x,w) \leq n,\\
m(y) &\text{if } x \in \gamma_y \text{ for some } y \leq w \text{ with } d(w,y) = n
\end{cases}
\]
defines a flow measure on $\tilde T_{w,n}$, which coincides with $m$ on $\tilde S_{w,n} = \{ x \in T \colon x \leq w, \ d(x,w) \leq n \}$.

In particular, if we set, for all $n \in \NN$,
\[
T_n = \tilde T_{\pred^n(o),2n}, \qquad \tilde m_n = \tilde m_{\pred^n(o),2n}, \qquad S_n = \tilde S_{\pred^{n}(o),2n},
\]
then $T_n$ is an increasing sequence of $\pred$-subtrees of $T$, $S_n$ is an increasing sequence of finite subsets of $T$, and
\[
\bigcup_{n \in \NN} T_n = \bigcup_{n \in \NN} S_n = T.
\]
Moreover, $\tilde m_n$ is a flow measure on $T_n$ which coincides with $m$ on $S_n$.

By applying Proposition \ref{p: point approx} to $(T_n,\tilde m_n)$, we can find a uniformly rational flow measure $m_n$ on $T_n$ such that $|m_n(x) - \tilde m_n(x)| \leq 2^{-n}$ for all $x \in S_n$ (here we use that $S_n$ is finite). As $m|_{S_n} = \tilde m_n|_{S_n}$, and the $S_n$ form an increasing sequence with $\bigcup_n S_n = T$, this proves that $m_n \to m$ pointwise on $T$, where $m_n$ is extended by zeros to $T$.

Thus, $((T_n,m_n))_n$ is the desired approximating sequence of $(T,m)$.
\end{proof}

\subsection{Universal transference}

For any flow tree $(T,m)$, Corollary \ref{c: point approx} gives us an approximating sequence $(T_j,m_j)$, where $m_j$ is $q_j$-uniformly rational for some $q_j \in \NN$. 
Therefore, according to Proposition \ref{p: comp unif rat}, $(T_j,m_j)$ is a flow quotient of the homogeneous tree $(\TT{q_j},m_{\TT{q_j}})$. As a consequence of the transference results for quotients of Section \ref{s: submersions} and the perturbation results of the present section, we can effectively transfer to $(T,m)$ those estimates that hold on $(\Tq,\mq)$ uniformly in $q$.

Namely, by combining Propositions \ref{p: stimeK2K1}, \ref{p: Lptransf} and \ref{p: comp unif rat}, and Corollaries \ref{c: conv_hker}\ref{en: c_conv_hker_Sigma} and \ref{c: point approx}, we deduce the following ``universal transference'' result for the joint functional calculus of $(\Sigma,\Sigma^*)$.

\begin{theorem}\label{t: unif_transf}
Let $(T,m)$ be a flow tree. 
Let $F\in \Hol(2,1)$. Then, for all $p \in [1,\infty]$,
\begin{equation}\label{f: unif_transf_Lp}
\|F(\Sigma,\Sigma^*)\|_{L^p(m) \to L^p(m)} \leq \liminf_{q \to \infty} \|F(\Sigma_{\Tq},\Sigma_{\Tq}^*)\|_{L^p(\mq) \to L^p(\mq)}.
\end{equation}
Moreover, for every weight $w:\NN \times \ZZ\times \ZZ \to [0,\infty)$ which is increasing in the first variable,
\begin{multline}\label{f: unif_transf_weight}
\sup_{y \in T} \sum_{x \in T} w(d(x,y),\ell(x),\ell(y)) \, |K_{F(\Sigma,\Sigma^*)}(x,y)| \, m(x) \\
\le \liminf_{q \to \infty} \sup_{y \in \Tq} \sum_{x \in \Tq} w(d_{\Tq}(x,y),\ell(x),\ell(y)) \, |K_{F(\Sigma_{\Tq},\Sigma_{\Tq}^*)}(x,y)| \, \mq(x).
\end{multline}
\end{theorem}

\begin{remark}
From Section \ref{ss: noncommutative} we know that the $\Hol(2,1)$-based joint functional calculus for $(\Sigma,\Sigma^*)$ includes, among others, the operators of the form $F(\opL)$ for $F \in \Hol(1,2)$. As we shall see, if we restrict to functions of $\opL$, the analyticity condition on $F$ can be substantially relaxed (see Theorem \ref{t: unif_transf_opL} below).
\end{remark}

\begin{remark}
In Theorem \ref{t: unif_transf} above we can take as $(T,m)$ any homogeneous tree $(\Tq,m_{\Tq})$. As a consequence, for any $F \in \Hol(2,1)$ and $p \in [1,\infty]$, from \eqref{f: unif_transf_Lp} we deduce that 
\[
\liminf_{q \to \infty} \|F(\Sigma_{\Tq},\Sigma_{\Tq}^*)\|_{L^p(\mq) \to L^p(\mq)} = \sup_{q \in \Npos} \|F(\Sigma_{\Tq},\Sigma_{\Tq}^*)\|_{L^p(\mq) \to L^p(\mq)}.
\]
In other words, Theorem \ref{t: unif_transf} says that, for a given $F \in \Hol(2,1)$ and $p \in [1,\infty]$, the supremum 
\[
\sup_{(T,m)} \|F(\Sigma,\Sigma^*)\|_{L^p(m) \to L^p(m)}
\]
is the same, irrespective of whether $(T,m)$ ranges among
\begin{itemize}
\item all flow trees, or
\item all the homogeneous trees $(\Tq,\mq)$, or
\item any infinite subclass of the $(\Tq,\mq)$.
\end{itemize}
Analogous considerations hold for the weighted estimates \eqref{f: unif_transf_weight}.
\end{remark}

\section{Deriving estimates from homogeneous trees}\label{s: homogeneous}

In light of the universal transference result of Theorem \ref{t: unif_transf}, in order to obtain estimates for the functional calculus of $(\Sigma,\Sigma^*)$ on any nonhomogeneous flow tree, it is enough to prove analogous estimates on homogeneous trees $(\Tq,\mq)$ which are uniform in $q$. As we shall see, one way to obtain such uniform estimates is reducing the analysis of the flow Laplacian and related operators on $\Tq$ to that of similar operators on $\TT{1}$, that is, the discrete group $\ZZ$ equipped with the counting measure and the discrete Laplacian.

The key technical tool that makes this reduction possible is a \emph{discrete Abel transform}, taking advantage of the symmetries of $\Tq$ to reduce matters to the ``one-dimensional case'' of $\ZZ$. The discrete Abel transform on $\Tq$ has already been introduced and used elsewhere, especially in the context of the analysis on homogeneous trees with the counting measure and the combinatorial Laplacian (see \cite{CMS0,CMS} and references therein). A possibly novel aspect of our analysis, beside the fact that we work with the flow Laplacian in place of the combinatorial one, is the focus on the uniformity in $q$ of the estimates on $\Tq$ obtained from $\ZZ$.

\subsection{Analysis on \texorpdfstring{$\ZZ$}{Z}}
The homogeneous tree $\TT{1}$ of degree $2$ can be identified with $\ZZ$ via the level map $\ell : \TT{1} \to \ZZ$, which in this case is a bijection. Via this identification, the canonical flow measure $\mTT{1}$ of \eqref{f: can_flow}, that is, the counting measure on $\TT{1}$, corresponds to the counting measure $\#$ on $\ZZ$.

As in \cite[Section 2.3]{LMSTV}, for every function $\psi$ in $\CC^{\ZZ}$ we define the symmetric gradient $\tilde\nabla_\ZZ \psi$ and the combinatorial Laplacian $\Delta_{\ZZ} \psi$ by
\begin{equation*}
\tilde{\nabla}_{\ZZ}\psi(n) =\psi(n-1)-\psi(n+1), \qquad
\Delta_{\ZZ} \psi(n) = \psi(n)-\frac{1}{2} \left( \psi(n-1)+\psi(n+1) \right)
\end{equation*}
for all $n \in \ZZ$. By comparing the above expression with \eqref{f: mathcalA}, and recalling that the underlying flow measure here is the counting measure, it is clear that the combinatorial Laplacian $\Delta_\ZZ$ is the same as the flow Laplacian $\opL_{\TT{1}}$. Instead, $\tilde\nabla_{\ZZ}$ corresponds to $\nabla_{\TT{1}}-\nabla_{\TT{1}}^*$, i.e., (twice) the skewsymmetric part of the flow gradient.

Of course, the operators $\tilde\nabla_\ZZ$ and $\Delta_\ZZ$ are translation-invariant on $\ZZ$. In particular, in terms of the usual group convolution $*$ on $\ZZ$, we can write
\[
\Delta_{\ZZ} f = f * k_{\Delta_{\ZZ}},
\]
where, for all $n \in \ZZ$,
\begin{equation}\label{f: ker_LapZ}
k_{\Delta_{\ZZ}}(n) =
\begin{cases}
1 & \text{if } n=0, \\
-1/2 & \text{if } n=\pm 1, \\
0 & \text{otherwise.}
\end{cases}
\end{equation}

We shall take advantage of Fourier analysis to intertwine translation-invariant operators on $\ZZ$ with multiplication operators on the torus $\RR/(2\pi \ZZ)$.
Specifically, for a (nice) $2\pi$-periodic function $M : \RR \to \CC$, we denote by $\widehat{M}$ its Fourier transform, that is, the sequence $(\widehat{M}(n))_{n \in \ZZ}$ of the Fourier coefficients of $M$:
\[
\widehat{M}(n) = \frac{1}{2\pi} \int_{-\pi}^\pi M(\theta) \, e^{i n \theta} \,d\theta, \qquad n \in \ZZ.
\]

From \eqref{f: ker_LapZ} it follows that the inverse Fourier transform of $k_{\Delta_{\ZZ}}$, i.e., the Fourier multiplier corresponding to $\Delta_\ZZ$, is the function 
\begin{equation}\label{f: Fmult_DeltaZ}
H(\theta) = 1-\frac{1}{2} e^{i\theta}-\frac12 e^{-i\theta} = 1-\cos\theta, \qquad \theta\in \RR/(2\pi \ZZ).
\end{equation}
Notice that the image of $H$ is the interval $[0,2]$, which corresponds to the $L^2$ spectrum of $\Delta_\ZZ$ (see Proposition \ref{p: spectrum}).

An elementary computation then allows us to obtain an expression for the skewsymmetric gradient of the convolution kernel of operators in the functional calculus for the Laplacian on $\ZZ$.

\begin{lemma}\label{l: kernelFZ}
For every bounded Borel function $F : [0,2] \to \CC$,
\begin{equation}\label{f: kernelFZ}
\tilde \nabla_{\ZZ} k_{F(\Delta_{\ZZ})}(n)
= - \frac{1}{2\pi} \int_{-\pi}^{\pi} 2i\sin\theta \, F(1-\cos\theta) \, e^{in\theta} \, d\theta = -2i \widehat{M}_F(n),
\end{equation}
where
\begin{equation}\label{f: multFZ}
M_F(\theta)=\sin\theta F(1-\cos\theta), \qquad \theta\in \RR/(2\pi\ZZ),
\end{equation}
and $\widehat M_F$ denotes the Fourier transform of $M_F$. 
\end{lemma} 
\begin{proof}
For every $n\in\ZZ$, by \eqref{f: Fmult_DeltaZ},
\[
k_{F(\Delta_{\ZZ})}(n) = \frac{1}{2\pi} \int_{-\pi}^{\pi} F(1-\cos\theta) \, e^{in\theta} \, d\theta
\]
and therefore
\[\begin{split}
\tilde\nabla_{\ZZ} k_{F(\Delta_{\ZZ})}(n)
&= k_{F(\Delta_{\ZZ})}(n-1)- k_{F(\Delta_{\ZZ})}(n+1) \\
&= \frac{1}{2\pi} \int_{-\pi}^{\pi} F(1-\cos\theta) \, [e^{i(n-1)\theta}-e^{i(n+1)\theta}] \, d\theta \\
&= \frac{1}{2\pi} \int_{-\pi}^{\pi} F(1-\cos\theta) \, (-2i\sin\theta) \, e^{in\theta} \, d\theta,
\end{split}\]
as desired.
\end{proof}

The reason why we are interested in the kernels $\tilde\nabla_\ZZ k_{F(\Delta_{\ZZ})}$ will become clear in the next section: the skewsymmetric gradient $\tilde\nabla_\ZZ$ appears in the inversion formula for the Abel transform (see Proposition \ref{p: KF(L)} below).

Thanks to the expression in Lemma \ref{l: kernelFZ}, together with the Plancherel formula and other properties of the Fourier transform, we can establish the following weighted $L^2$ estimates.

\begin{lemma}\label{l: FDeltaZsenzasupporto}
For every $\alpha\in [0,\infty)$ and $F\in L^2_\alpha(\RR)$,
\begin{equation}\label{f: alphain02senzat}
\sum_{n\in\ZZ}(1+{|n|})^{2\alpha} \, |\tilde\nabla_{\ZZ} k_{F(\Delta_{\ZZ})}(n)|^2 \lesssim_\alpha \|F\|_{L^2_{\alpha}}^2.
\end{equation}
\end{lemma}
\begin{proof}
By interpolation, it is sufficient to prove the estimate for $\alpha \in 2\NN$.

In addition, from the expression \eqref{f: kernelFZ} it is clear that the kernel $\tilde\nabla_{\ZZ} k_{F(\Delta_{\ZZ})}$ does not change if $F$ is modified outside the interval $[0,2]$. Thus, up to multiplying $F$ with an appropriate smooth cutoff function, we may assume that $\supp F \subseteq [-1,3]$.

A further splitting by means of cutoff functions allows us to reduce to the cases where $\supp F \subseteq [-1,7/4]$ and $\supp F \subseteq [1/4,3]$, which we shall consider separately.

Let us first assume that $\supp F \subseteq [-1,7/4]$. By differentiating \eqref{f: multFZ}, one can readily check that, for all $h \in \NN$,
\begin{equation}\label{f: iterated_der_MF}
\begin{aligned}
\partial_\theta^{2h} M_F(\theta) &= \sum_{\ell=0}^{2h} \sin^{1+2(\ell-h)_+}(\theta) \, p_{h,\ell}(\cos\theta) \, F^{(\ell)}(1-\cos\theta),\\
\partial_\theta^{2h+1} M_F(\theta) &= \sum_{\ell=0}^{2h+1} \sin^{2(\ell-h)_+}(\theta) \, q_{h,\ell}(\cos\theta) \, F^{(\ell)}(1-\cos\theta)
\end{aligned}
\end{equation}
for appropriate real polynomials $p_{h,\ell}$ and $q_{h,\ell}$ independent of $F$; the proof goes by induction on $h$, using the Chain and Leibniz Rules, as well as the fact that $\sin^2\theta = 1-\cos^2 \theta$ (so any even power of $\sin\theta$ can be absorbed into the polynomial in $\cos\theta$ where needed).

Thus, by Lemma \ref{l: kernelFZ}, the Plancherel formula and \eqref{f: iterated_der_MF}, for all $h \in \NN$,
\[\begin{split}
\sum_{n\in\ZZ} |n|^{4h} \, |\tilde\nabla_{\ZZ}k_{F(\Delta_{\ZZ})}(n)|^2 
&\approx_h \|\partial_\theta^{2h} M_{F}\|_{L^2(-\pi,\pi)}^2 \\
&\lesssim_h \sum_{\ell=0}^{2h}  \int_{-\pi}^\pi \left|\sin^{1+2(\ell-h)_+}(\theta) \, F^{(\ell)}(1-\cos\theta)\right|^2 \,d\theta.
\end{split}\]
The change of variables $x^2=1-\cos\theta$, i.e., $x = \sqrt{2} \sin(\theta/2)$, then shows that
\begin{equation}\label{f: stimaell=2h}
\begin{split}
\sum_{n\in\ZZ} |n|^{4h} \, |\tilde\nabla_{\ZZ}k_{F(\Delta_{\ZZ})}(n)|^2 
&\lesssim_h \sum_{\ell=0}^{2h} \int_{\RR} |x^{1+2(\ell-h)_+} \, F^{(\ell)}(x^2)|^2 \,dx \\
&\approx \sum_{\ell=0}^{2h} \int_{\RR} |\lambda^{(\ell-h)_+} \, F^{(\ell)}(\lambda)|^2 \lambda^{1/2} \,d\lambda \\
&\lesssim_h \|F\|_{L^2_{2h}}^2,
\end{split}
\end{equation}
due to our support assumption on $F$. Combining this estimate with that for $h=0$ proves the case $\alpha = 2h$ of \eqref{f: alphain02senzat}.

Assume now instead that $\supp F \subseteq [1/4,3]$. Then \eqref{f: multFZ} shows that $M_F(\theta+\pi) = -M_G(\theta)$, where $G(\lambda) \defeq F(2-\lambda)$, and in particular $\supp G \subseteq [-1,7/4]$. As $\|\partial_\theta^{2h} M_F\|_{L^2(-\pi,\pi)} = \|\partial_\theta^{2h} M_G\|_{L^2(-\pi,\pi)}$ and $\|F\|_{L^2_{2h}} = \|G\|_{L^2_{2h}}$, the desired estimates are obtained by applying the above computations with $G$ in place of $F$.
\end{proof}

We now prove a ``scale-invariant version'' of the above estimate, which will be useful when discussing spectral multipliers of Mihlin--H\"ormander type.

\begin{lemma}\label{l: FtDeltaZconsupporto}
For every $\alpha\in [0,\infty)$, $F\in L^2_\alpha(\RR)$ supported in $[1/4,7/4]$ and $t\geq 1$, 
\begin{equation}\label{f: alphain02}
\sum_{n \in \ZZ} \left(1+\frac{|n|}{\sqrt{t}}\right)^{2\alpha} \, |\tilde\nabla_{\ZZ} k_{F(t\Delta_{\ZZ})}(n)|^2
\lesssim_\alpha t^{-3/2} \, \|F\|_{L^2_{\alpha} }^2. 
\end{equation}
\end{lemma} 
\begin{proof}
Again, by interpolation, it is sufficient to consider the case $\alpha \in 2\NN$.

Notice that, for any $F$ supported in $[-1,7/4]$, from \eqref{f: stimaell=2h} we deduce that, if $\alpha = 2h$, $h \in \NN$, then
\[
\sum_{n\in\ZZ} |n|^{2\alpha} \, |\tilde\nabla_{\ZZ} k_{F(\Delta_{\ZZ})}(n)|^2 \lesssim_\alpha \sum_{j=0}^{\alpha} \int_0^\infty |\lambda^j F^{(j)}(\lambda)|^2 \, \lambda^{3/2-\alpha} \, \frac{d\lambda}{\lambda}.
\]

We now apply this bound to $F(t\cdot)$ in place of $F$, under the assumption that $\supp F \subseteq [1/4,7/4]$; as $t \geq 1$, we still have $\supp F(t\cdot) \subseteq [-1,7/4]$, so the bound can be applied. Thus we deduce that, for every $t\geq 1$ and $\alpha \in 2\NN$,
\[
\sum_{n\in\ZZ} |n|^{2\alpha} \, |\tilde\nabla_{\ZZ} k_{F(t\Delta_{\ZZ})}(n)|^2 \lesssim_\alpha t^{\alpha-3/2} \, \|F\|^2_{{L^2_\alpha}},
\]
which clearly implies the bound \eqref{f: alphain02}.
\end{proof}

\subsection{The Abel transform connection}
We shall now consider the homogeneous tree $\Tq$ with $q \geq 2$. As before, we equip $\Tq$ with the canonical flow measure $\mq$ given by \eqref{f: can_flow}, and denote by $\opL_{\Tq}$ the corresponding flow Laplacian.

The aim of this section is to describe a relation between the functional calculi of $\opL_{\Tq}$ and $\Delta_\ZZ$, which is illustrated by the following result. The proof that we present below is based in a fundamental way on results and ideas from \cite{CMS0}, particularly in regards to the use of the discrete Abel transform on $\Tq$.

Recall that $\Pol(1)$ denotes the set of polynomials with complex coefficients in one indeterminate. We initially prove the formula \eqref{f: kernelF(Lq)} below for $F \in \Pol(1)$, in order to avoid convergence problems; however, by a density argument, the formula extends to any continuous functions $F$, and indeed, by applying it to $F(\lambda) = e^{-t\lambda}$, one recovers the heat kernel formula of \cite[Proposition 2.4]{LMSTV}.

\begin{prop}\label{p: KF(L)}
For every $F$ in $\Pol(1)$,
\begin{equation}\label{f: kernelF(Lq)}
K_{F(\Lq)}(x,y)=q^{-\frac12(\ell(x)+\ell(y))} E_F(d(x,y)),\qquad x,y\in\Tq,
\end{equation}
where $E_F:\NN\to \CC$ is defined by 
\begin{equation}\label{f: E_F}
E_F(k)=\sum_{j\geq 0} q^{-\frac{k+2j}{2}} \, \tilde{\nabla}_{\ZZ} k_{F(\Delta_{\ZZ})}(k+2j+1) \qquad k\in\NN.
\end{equation}
\end{prop}
\begin{proof}
On the homogeneous tree $\Tq$, we shall denote by $\opAv_q$ the averaging operator from \eqref{f: mathcalA}, i.e.,
\[
\opAv_q f(x)=\frac{1}{2} f(\pred(x))+ \frac{1}{2q} \sum_{y \in \succ(x)} f(y), \qquad f \in \CC^{\Tq}, \ x\in \Tq,
\]
so that $\Lq = I-\opAv_q$. Let $\opM_q$ denote moreover the ``isotropic'' averaging operator
\[
\opM_q f(x)=\frac{1}{q+1} \sum_{y\sim x} f(y), \qquad f\in \CC^{\Tq}, \ x\in \Tq,
\]
which is adapted to the counting measure $\#$ on $\Tq$; indeed $I - \opM_q$ is the combinatorial Laplacian on $\Tq$.
Consider also the operator $H_q$ defined by 
\[
H_q f(x)=q^{-\ell(x)/2}f(x), \qquad f\in \CC^{\Tq}, \ x\in \Tq,
\]
which is an isometry from $L^2(\Tq,\#)$ to $L^2(\Tq,\mq)$. 

An easy computation shows that
\[
\opAv_q H_q = \frac{q+1}{2\sqrt q} H_q \opM_q,
\]
so that, for every $F\in\Pol(1)$,
\begin{equation}\label{f: rel_comb_flow}
K_{F(\opAv_q)}(x,y)=q^{-\frac12(\ell(x)+\ell(y))}K_{F(\frac{q+1}{2\sqrt q} \opM_q)}(x,y),\qquad x,y\in\Tq ;
\end{equation}
here $K_{F(\opAv_q)}$ denotes the integral kernel with respect to the flow measure $m_{\Tq}$, while $K_{F(\frac{q+1}{2\sqrt q} \opM_q)}$ is the kernel with respect to the counting measure $\#$.

As is well known (see \cite{CMS0,CMS}), thanks to the symmetries of $(\Tq,\#)$, the kernel $K_{F(\frac{q+1}{2\sqrt q} \opM_q)}(x,y)$ only depends on $d(x,y)$. So, from \eqref{f: rel_comb_flow}, we deduce that \eqref{f: kernelF(Lq)} holds for some $E_F : \NN \to \CC$, and it only remains to compute the expression of such $E_F$ in terms of $F \in \Pol(1)$. Notice that, as $F \in \Pol(1)$, we know that $F(\opL) \in \Bdd_\fin(m_{\Tq})$, thus we must have $E_F \in c_{00}(\NN)$ in this case.

To compute $E_F$, consider the map $\pi : \Tq \to \ZZ$ defined by $\pi(x)=\ell(x)$ for every $x\in\Tq$. It is easy to see that $\pi$ is a flow submersion for the canonical flows on $\Tq$ and $\ZZ$, respectively. Thus from Proposition \ref{p: Intral} we deduce that $F(\opL_{\Tq}) \in \Comp_\fin(\pi)$ and $\pi(F(\opL_{\Tq})) = F(\Delta_\ZZ)$. In particular, Proposition \ref{p: K2K1} gives us a relation between the integral kernels of $F(\opL_{\Tq})$ on $(\Tq,\mq)$ and of $F(\Delta_\ZZ)$ on $\ZZ$. Taking into account that $F(\Delta_\ZZ)$ is a convolution operator, i.e., $K_{F(\Delta_\ZZ)}(n,m) = k_{F(\Delta_\ZZ)}(n-m)$, this relation can be written as
\begin{equation}\label{f: first_abel}
\begin{split}
k_{F(\Delta_{\ZZ})}(m-n) 
&= \sum_{x\in\Tq \colon \ell(x)=m} K_{F(\Lq)}(x,y) \, q^{\ell(x)}\\
&= q^{\frac{m-n}{2}}\sum_{x\in\Tq \colon \ell(x)=m} E_F(d(x,y))
\end{split}
\end{equation}
for all $m,n \in \ZZ$ and $y \in \Tq$ with $\ell(y) = n$, where we also used \eqref{f: kernelF(Lq)}.

Now, thanks to the symmetries of $\Tq$, it is easily checked that, for all $y \in \Tq$, $d \in \NN$ and $r \in \ZZ$,
\begin{equation}\label{f: vertices_number}
\# \{ x \in \Tq \colon d(x,y) = d, \ \ell(x)-\ell(y) = r \}  = \begin{cases}
q^{(d-r)/2} &\text{if } |r|=d,\\
(q-1) q^{(d-r)/2-1} &\text{if } d-|r| \in 2\Npos,\\
0 &\text{otherwise}.
\end{cases}
\end{equation}
From \eqref{f: first_abel} and \eqref{f: vertices_number} we then deduce that
\[
\begin{split}
&k_{F(\Delta_{\ZZ})}(m-n) \\
&= q^{\frac{m-n}{2}} \sum_{k=0}^{\infty} E_F(|m-n|+2k) \, \#\{x\in\Tq \colon \ell(x)=m, \ d(x,y)=|m-n|+2k\}    \\
&= q^{\frac{|m-n|}{2}} \left[E_F(|m-n|)+\frac{q-1}{q}\sum_{k=1}^{\infty}q^kE_F(|m-n|+2k) \right].
\end{split}
\]

The above formula can be rephrased as follows: 
\begin{equation}\label{f: abel_rel}
k_{F(\Delta_{\ZZ})}(n) = \Abel_q (E_F)(|n|), \qquad n \in \ZZ,
\end{equation}
where for every $\phi \in c_{00}(\NN)$, the Abel transform $\Abel_q(\phi) \in c_{00}(\NN)$ of $\phi$ is given by
\[
\Abel_q(\phi)(j) = q^{j/2} \left[\phi(j) + \frac{q-1}{q} \sum_{k=1}^{\infty}q^k\phi(j+2k) \right],  \qquad j\in \NN.
\]
By \cite[Theorem 2.3]{CMS} we know that the Abel transform $\Abel_q : c_{00}(\NN) \to c_{00}(\NN)$ is invertible, and we have the inversion formula
\[\begin{split}
\Abel_q^{-1} \psi(n)
&= q^{-n/2} \psi(n) - (q-1) \sum_{j>0} q^{-\frac{n+2j}{2}} \, \psi(n+2j)\\
&= \sum_{j\geq 0} q^{-\frac{n+2j}{2}} \, \tilde{\nabla}_{\ZZ}\psi(n+2j+1), \qquad n \in \NN.
\end{split}\]
By using this formula, we can invert the relation \eqref{f: abel_rel} and obtain \eqref{f: E_F}.
\end{proof}

A density argument allows us to extend the previous formula beyond the class of polynomials.

\begin{corollary}\label{c: KF(L)}
The formula \eqref{f: kernelF(Lq)} holds true for all $F \in C[0,2]$.
\end{corollary}
\begin{proof}
From the expression \eqref{f: kernelFZ} it is clear that $\|\tilde \nabla_\ZZ k_{F(\Delta_\ZZ)}\|_\infty \lesssim \|F\|_\infty$; in particular, the sum in \eqref{f: E_F} converges absolutely whenever $F \in C[0,2]$, so $E_F$ is well defined and 
\begin{equation}\label{f: EF_cont}
\|E_F\|_\infty \lesssim \|\tilde \nabla_\ZZ k_{F(\Delta_\ZZ)}\|_\infty \lesssim \|F\|_\infty.
\end{equation}

Let $A$ be the set of the functions $F \in C[0,2]$ such that the formula \eqref{f: kernelF(Lq)} holds true. It is immediately checked that $A$ is a linear subspace of $C[0,2]$, and from Proposition \ref{p: KF(L)} we know that $A$ contains all the polynomials. To conclude that $A = C[0,2]$, by the Stone--Weierstrass Theorem it is enough to prove that $A$ is closed under uniform convergence.

On the other hand, if $F_n \in A$ and $F_n \to F$ uniformly on $[0,2]$, then $F_n(\Lq) \to F(\Lq)$ in the $L^2$ operator norm, thus $K_{F_n(\Lq)} \to K_{F(\Lq)}$ pointwise on $\Tq \times \Tq$. 
Moreover, since the map $F \mapsto E_F$ given by \eqref{f: E_F} is linear, the bound \eqref{f: EF_cont} shows that the uniform convergence $F_n \to F$ implies the uniform convergence $E_{F_n} \to E_F$. Thus, we can take the limit in both sides of \eqref{f: kernelF(Lq)} and from the identity for the $F_n$ we deduce that for $F$.
\end{proof}

We shall now use the formula from Proposition \ref{p: KF(L)} to derive, for any $F \in C[0,2]$, a weighted $L^1$ estimate for the kernel of $F(\opL_{\Tq})$ in terms of a suitable weighted $L^1$ norm 
of $\tilde\nabla_\ZZ k_{F(\Delta_\ZZ)}$; a crucial aspect of the estimate below is its uniformity in $q$.

\begin{prop}\label{p: stimaL1pesataTq}
For all $F \in C[0,2]$ and all increasing weights $w : \NN \to [0,\infty)$,
\[
\sup_{y\in\Tq} \sum_{x\in\Tq} |K_{F(\opL_{\Tq})}(x,y)| \, w(d(x,y)) \,\mq(x)
\lesssim \sum_{n\in\NN} n \, w(n) \, |\tilde\nabla_{\ZZ}k_{F(\Delta_{\ZZ})}(n)|,
\]
where the implicit constant does not depend on $q$.
\end{prop}
\begin{proof}
By Corollary \ref{c: KF(L)}, we can apply the formula \eqref{f: kernelF(Lq)} and deduce that, for every $y\in \Tq$,
\[\begin{split}
&\sum_{x\in\Tq}|K_{F(\opL_{\Tq})}(x,y)| \, w(d(x,y)) \,\mq(x)\\
&\leq \sum_{x\in\Tq} w(d(x,y)) \, q^{\frac{\ell(x)-\ell(y)}{2}} \sum_{j\in\NN} q^{-\frac{d(x,y)+2j}{2}} \, |\tilde\nabla_{\ZZ}k_{F(\Delta_{\ZZ})}(d(x,y)+2j+1)| \\
&\leq \sum_{x\in\Tq} q^{\frac{\ell(x)-\ell(y)}{2}} \sum_{j\in\NN} q^{-\frac{d(x,y)+2j}{2}} \, w(d(x,y)+2j+1) \, |\tilde\nabla_{\ZZ} k_{F(\Delta_{\ZZ})}(d(x,y)+2j+1)|,
\end{split}\]
as $w$ is increasing. By the estimate
\[
\sum_{x \colon d(x,y)=n} q^{\frac{\ell(x)-\ell(y)}{2}} \approx q^{n/2} (n+1)
\]
from \cite[Lemma 2.7]{MSV}, we then deduce
\[\begin{split}
&\sum_{x\in\Tq}|K_{F(\opL_{\Tq})}(x,y)| \, w(d(x,y)) \,\mq(x)\\
&\lesssim \sum_{n\in\NN} q^{n/2} \, (n+1) \sum_{j\in\NN} q^{-\frac{n+2j}{2}} w(n+2j+1) \, |\tilde\nabla_{\ZZ} k_{F(\Delta_{\ZZ})}(n+2j+1)| \\
&\leq \sum_{n\in\NN} \sum_{j\in\NN} q^{-j}(n+2j+1) \, w(n+2j+1) \, |\tilde\nabla_{\ZZ} k_{F(\Delta_{\ZZ})}(n+2j+1)| \\
&= \sum_{j\in\NN} \sum_{n\geq 2j+1} q^{-j} \, n \, w(n) \, |\tilde\nabla_{\ZZ} k_{F(\Delta_{\ZZ})}(n)| \\
&\lesssim \sum_{n\in\NN} n \, w(n)|\tilde\nabla_{\ZZ} k_{F(\Delta_{\ZZ})}(n)|,
\end{split}\]
where we used the change of indices $n+2j+1 \mapsto n$.  
\end{proof}

\subsection{Differentiable functional calculus and heat kernel estimates for the flow Laplacian}

Due to the uniformity in $q$ of the estimates for polynomial functions $F(\Lq)$ of the flow Laplacian on the homogeneous tree $\Tq$ in Proposition \ref{p: stimaL1pesataTq}, we can apply Theorem \ref{t: unif_transf} to obtain analogous estimates on any flow tree of bounded degree.

\begin{corollary}\label{c: stimaL1pesata_trans}
Let $(T,m)$ be a flow tree.
For every $F \in \Hol(1,2)$ and every increasing weight $w : \NN \to [0,\infty)$,
\begin{equation}\label{f: stimaL1pesata_trans}
\sup_{y\in T} \sum_{x\in T} |K_{F(\opL)}(x,y)| \, w(d(x,y)) \,m(x)
\lesssim \sum_{n\in\NN} n \, w(n) \, |\tilde\nabla_{\ZZ}k_{F(\Delta_{\ZZ})}(n)|,
\end{equation}
where the implicit constant does not depend on $(T,m)$.
\end{corollary}

We now combine the above inequality with the estimate on $\ZZ$ from Lemma \ref{l: FDeltaZsenzasupporto} to deduce a weighted $L^1$ estimate for sufficiently smooth functions of a flow Laplacian. 

\begin{prop}\label{p: stimaL1Tsenzasupporto}
Let $(T,m)$ be a flow tree.
For every $\alpha \geq 0$ and $\varepsilon >0$,
and for every $F\in L^2_{\alpha+3/2+\varepsilon}(\RR)$,
\begin{equation}\label{f: stimaL1Tsenzasupporto}
\sup_{y\in T} \sum_{x\in T} |K_{F(\opL)}(x,y)| \, (1+d(x,y))^{\alpha} \,m(x) \lesssim_{\alpha,\varepsilon} \|F\|_{L^2_{\alpha+3/2+\varepsilon}}.
\end{equation}
Moreover, for every increasing weight $w : \NN \to [0,\infty)$ satisfying $w(n) \lesssim_w (1+n)^\alpha$ for all $n \in \NN$, the estimate \eqref{f: stimaL1pesata_trans} holds for all $F\in L^2_{\alpha+3/2+\varepsilon}(\RR)$.
\end{prop}
\begin{proof}
Let $\chi \in C_c^\infty(\RR)$ be such that $\supp \chi \subseteq [-1,3]$ and $\chi|_{[0,2]} = 1$. In place of \eqref{f: stimaL1Tsenzasupporto}, we shall prove the apparently stronger estimate
\begin{equation}\label{f: apriori}
\sup_{y\in T} \sum_{x\in T} |K_{F(\opL)}(x,y)| \, (1+d(x,y))^{\alpha} \,m(x) \lesssim_{\alpha,\varepsilon} \|\chi F\|_{L^2_{\alpha+3/2+\varepsilon}}
\end{equation}
for all $F \in L^2_{\alpha+3/2+\varepsilon}(\RR)$.

Assume first that $F$ is a polynomial. By Corollary \ref{c: stimaL1pesata_trans}, the Cauchy--Schwarz inequality and Lemma \ref{l: FDeltaZsenzasupporto},
\begin{equation}\label{f: apriori2}
\begin{split}
&\sup_{y \in T} \sum_{x \in T} |K_{F(\opL)}(x,y)| \, (1+d(x,y))^{\alpha} \, \mq(x) \\
&\lesssim \sum_{n \in \NN} n \, (1+n)^{\alpha} \, |\tilde\nabla_{\ZZ} k_{F(\Delta_{\ZZ})}(n)| \\
&\lesssim \left(\sum_{n \in \NN} (1+n)^{2\alpha+3+2\varepsilon} \, |\tilde\nabla_{\ZZ} k_{(\chi F)(\Delta_{\ZZ})}(n)|^2 \right)^{1/2} \left(\sum_{n \in \NN} (1+n)^{-1-2\varepsilon}\right)^{1/2} \\
&\lesssim_{\alpha,\varepsilon} \|\chi F\|_{L^2_{\alpha+3/2+\varepsilon}},
\end{split}
\end{equation}
where we used that $k_{F(\Delta_{\ZZ})}$ only depends on $F|_{[0,2]}$ and therefore $k_{(\chi F)(\Delta_{\ZZ})} = k_{F(\Delta_{\ZZ})}$. This proves \eqref{f: apriori} when $F$ is a polynomial.

A density argument then allows us to extend the validity of \eqref{f: apriori} to any $F \in L^2_{\alpha+3/2+\varepsilon}(\RR)$. Indeed, from the Stone--Weierstrass Theorem one readily deduces that there exists a sequence $F_n$ of polynomials such that $\chi F_n \to \chi F$ in $L^2_{\alpha+3/2+\varepsilon}(\RR)$. Applying \eqref{f: apriori} to the polynomials $F_n-F_m$ then shows that the sequence of kernels $K_{F_n(\opL)}$ is a Cauchy sequence and converges (with respect to the weighted norm in the left-hand side of \eqref{f: apriori}, thus also pointwise) to a function $H_F$ on $T \times T$ satisfying
\[
\sup_{y\in T} \sum_{x\in T} |H_F(x,y)| \, (1+d(x,y))^{\alpha} \, m(x) \lesssim_{\alpha,\varepsilon} \|\chi F\|_{L^2_{\alpha+3/2+\varepsilon}} .
\]
On the other hand, by Sobolev's embedding, the $F_n$ converge uniformly to $F$ on $[0,2]$, thus the operators $F_n(\opL)$ converge to $F(\opL)$ in the $L^2$ operator norm; this implies the pointwise convergence of the corresponding convolution kernels, so the limit $H_F$ of the kernels $K_{F_n(\opL_{\Tq})}$ must be $K_{F(\opL_{\Tq})}$, and \eqref{f: apriori} follows.

Take now any increasing weight $w : \NN \to [0,\infty)$ satisfying $w(n) \lesssim_w (1+n)^\alpha$. From \eqref{f: apriori} we deduce in particular that
\begin{equation}\label{f: wapriori}
\sup_{y\in T} \sum_{x\in T} |K_{F(\opL)}(x,y)| \, w(d(x,y)) \,m(x) \lesssim_{w,\alpha,\varepsilon} \|\chi F\|_{L^2_{\alpha+3/2+\varepsilon}}
\end{equation}
for all $F \in L^2_{\alpha+3/2+\varepsilon}(\RR)$. Moreover, arguing as in \eqref{f: apriori2}, from Lemma \ref{l: FDeltaZsenzasupporto} and the Cauchy--Schwarz inequality we also deduce that
\begin{equation}\label{f: waprioriZ}
\sum_{n\in\NN} n \, w(n) \, |\tilde\nabla_{\ZZ}k_{F(\Delta_{\ZZ})}(n)| 
\lesssim_{w,\alpha,\varepsilon} \|\chi F\|_{L^2_{\alpha+3/2+\varepsilon}}
\end{equation}
for all $F \in L^2_{\alpha+3/2+\varepsilon}(\RR)$. In particular, for a given $F \in L^2_{\alpha+3/2+\varepsilon}(\RR)$, if we take as before a sequence $F_n$ of polynomials such that $\chi F_n \to \chi F$ in $L^2_{\alpha+3/2+\varepsilon}(\RR)$, then the kernels $K_{F_n(\opL)}$ and $k_{F_n(\Delta_\ZZ)}$ converge to $K_{F(\opL)}$ and $k_{F(\Delta_\ZZ)}$ in the sense of the weighted Lebesgue norms in the left-hand sides of \eqref{f: wapriori} and \eqref{f: waprioriZ}. We can therefore apply \eqref{f: stimaL1pesata_trans} to the polynomials $F_n$ and then pass to the limit as $n \to \infty$ to deduce the analogous estimate \eqref{f: stimaL1pesata_trans} for $F$.
\end{proof}

We can now prove Theorem \ref{t: multipliers}. 

\begin{proof}[Proof of Theorem \ref{t: multipliers}]
Suppose that $F\in L^2_s(\RR)$ for some $s>3/2$. Then, by Proposition \ref{p: stimaL1Tsenzasupporto},
\[
\sup_{y\in T} \sum_{x\in T} |K_{F(\opL)}(x,y)| \, m(x)
\lesssim_s \|F\|_{L^2_{s}}, 
\]
so $F(\opL)$ is bounded on $L^1(m)$ and 
\[ 
\|F(\opL)\|_{L^1 \to L^1} \lesssim_s \|F\|_{L^2_{s}}. 
\]
By applying the above estimate to the conjugate function $\overline{F}$, we obtain an analogous $L^\infty$ bound for $F(\opL)$, which is the adjoint of $\overline{F}(\opL)$. By interpolation it follows that $F(\opL)$ is bounded on $L^p(m)$ for $p\in (1,\infty)$. 
\end{proof}

As a consequence, we can complete the characterisation of the $L^p$ spectrum of flow Laplacians.

\begin{corollary}\label{c: spectrum}
Let $(T,m)$ be a flow tree and $p \in [1,\infty]$. The $L^p$ spectrum of $\opL$ is the interval $[0,2]$.
\end{corollary}
\begin{proof}
In light of Proposition \ref{p: spectrum}, it only remains to prove that the $L^p$ spectrum of $\opL$ is real. On the other hand, if $z \in \CC \setminus \RR$, then the function $F_z(\lambda) = (\lambda-z)^{-1}$ is smooth  and bounded on $\RR$ with all its derivatives, so by Theorem \ref{t: multipliers} the operator $F_z(\opL) = (\opL-z)^{-1}$ is bounded on $L^p(m)$ for all $p \in [1,\infty]$; thus $\CC \setminus \RR$ is in the $L^p$ resolvent of $\opL$, as desired.
\end{proof}

A further consequence of Proposition \ref{p: stimaL1Tsenzasupporto} is the following result, extending the class of compatible operators for a flow submersion to all sufficiently smooth functions $F(\opL)$ of the flow Laplacian; this should be compared with Proposition \ref{p: Intral} above, which only gives this result for analytic functions $F$.

\begin{prop}\label{p: Intral_smooth}
Let $\pi : (T_1,m_1) \to (T_2,m_2)$ be a flow submersion. Denote by $\opL_j$ the flow Laplacian on $(T_j,m_j)$ for $j=1,2$. Let $F \in L^2_{3/2+\varepsilon}(\RR)$ for some $\varepsilon > 0$. Then $F(\opL_1) \in \overline{\Comp_\fin(\pi)}$ and $\pi(F(\opL_1)) = F(\opL_2)$.
\end{prop}
\begin{proof}
Let $\chi \in C_c^\infty(\RR)$ be such that $\supp \chi \subseteq [-1,3]$ and $F|_{[0,2]}=1$. Much as in the proof of Proposition \ref{p: stimaL1Tsenzasupporto}, we can find a sequence of polynomials $F_n$ such that $\chi F_n \to \chi F$ in $L^2_{3/2+\varepsilon}(\RR)$. 
 Since the spectrum of $\opL_j$ is $[0,2]$, clearly $F(\opL_j) = (\chi F)(\opL_j)$ and $F_n(\opL_j) = (\chi F_n)(\opL_j)$, for $j=1,2$. 
In particular, from Theorem \ref{t: multipliers} it follows that $F_n(\opL_j) \to F(\opL_j)$ in $\Bdd(m_j)$ as $n \to \infty$. Since $F_n(\opL_1) \in \Comp_\fin(\pi)$ and $\pi(F_n(\opL_1)) = F_n(\opL_2)$ by Proposition \ref{p: Intral}, and moreover $\pi : \Comp(\pi) \to \Bdd(m_2)$ is continuous by Proposition \ref{p: Cpi}, we deduce that $F(\opL_1) \in \overline{\Comp_\fin(\pi)}$ and $\pi(F(\opL_1)) = F(\opL_2)$, as desired.
\end{proof}

Combining Proposition \ref{p: Intral_smooth} with Propositions \ref{p: stimeK2K1} and \ref{p: Lptransf} shows that $L^p$ and weighted kernel estimates for operators of the form $F(\opL)$ with $F$ sufficiently smooth can be transferred via flow submersions. As these estimates are also amenable to perturbative arguments by Corollary \ref{c: conv_hker}\ref{en: c_conv_hker_opL}, we conclude that the ``universal transference'' result of Theorem \ref{t: unif_transf} extends to sufficiently smooth (but not necessarily analytic) functions of the flow Laplacian.

\begin{theorem}\label{t: unif_transf_opL}
Let $(T,m)$ be a flow tree. 
Let $F\in L^2_{3/2+\varepsilon}(\RR)$ for some $\varepsilon > 0$. Then, for all $p \in [1,\infty]$,
\[
\|F(\opL)\|_{L^p(m) \to L^p(m)} \leq \liminf_{q \to \infty} \|F(\Lq)\|_{L^p(\mq) \to L^p(\mq)}.
\]
Moreover, for every weight $w:\NN \times \ZZ\times \ZZ \to [0,\infty)$ which is increasing in the first variable,
\begin{multline*}
\sup_{y \in T} \sum_{x \in T} w(d(x,y),\ell(x),\ell(y)) \, |K_{F(\opL)}(x,y)| \, m(x) \\
\le \liminf_{q \to \infty} \sup_{y \in \Tq} \sum_{x \in \Tq} w(d_{\Tq}(x,y),\ell(x),\ell(y)) \, |K_{F(\Lq)}(x,y)| \, \mq(x).
\end{multline*}
\end{theorem}

Finally, we establish the heat kernel bounds for flow Laplacians on arbitrary flow trees stated in the Introduction.

\begin{proof}[Proof of Theorem \ref{t: est on T}]
In the case $(T,m) = (\Tq,\mq)$ and $t \geq 1$, the desired estimates are established in \cite[Theorem 1.1 and Proposition 2.8]{MSV}; as these estimates hold uniformly in $q$, and moreover the operators $\exp(-t\opL)$, $\nabla \exp(-t\opL)$, $\exp(-t\opL) \nabla^*$, $\nabla \exp(-t\opL) \nabla^*$ are of the form $F(\Sigma,\Sigma^*)$ for appropriate choices of $F \in \Hol(1,2)$ (see Section \ref{ss: noncommutative}), we can directly apply the universal transference result of Theorem \ref{t: unif_transf} to deduce the analogous estimates on any flow tree $(T,m)$.

It remains to prove the bounds \eqref{f: levelestimates} when $0<t <1$. These follow directly from the fact that $\nabla e^{-t \opL}$ is bounded on $L^1(m)$ and $L^\infty(m)$, uniformly in $t$, as both $\nabla$ and $e^{-t \opL}$ are. More precisely, 
\begin{align*}
\sup_{x \in T} \sum_{z \in T \colon \ell(z)=l} |K_{\nabla e^{-t\opL}}(x,z)| \, m(z) &\le  \|\nabla e^{-t\opL} \|_{L^\infty \to L^\infty}\lesssim 1, \\ 
\sup_{x \in T} \sum_{z \in T \colon \ell(z)=l} |K_{\nabla e^{-t\opL}}(z,x)| \, m(z) &\le  \|\nabla e^{-t\opL} \|_{L^1\to L^1}\lesssim 1,
\end{align*}
as required.
\end{proof}

\section{Singular integrals}

In this final section we discuss boundedness results for spectral multipliers and Riesz transforms on flow trees, thus completing the proofs of Theorems \ref{t: Riesz} and \ref{t: MHmultipliers}. The universal transference results of the previous sections, as well as the heat kernel estimates and the $L^1$ bounds for functions of the flow Laplacian, are here combined with an appropriate singular integral theory adapted to the nondoubling setting of flow trees.

\subsection{Calder\'on--Zygmund theory on locally doubling flow trees}\label{ss: CZ}

A basic tool that we shall use to establish boundedness properties of singular integrals on a locally doubling flow tree is the following result, which is based on the Calder\'on--Zygmund and Hardy space theory developed in \cite{LSTV} on locally doubling flow trees, on the basis of the Calder\'on--Zygmund theory for certain nondoubling spaces in \cite{hs}. Specifically, the statement below can be obtained from \cite[Theorem 1.2]{hs} and \cite[Theorem 5.8]{LSTV}; in the case of a homogeneous tree, the statement can be found in \cite[Proposition 3.1]{MSV}. We refer to \cite{ATV2,ATV1,LSTV,San} for the definitions of the atomic Hardy space $H^1(m)$ and the dual space $BMO(m)$ on a flow tree $(T,m)$ used throughout.

\begin{theorem}\label{t: lemmahebisch}
Let $(T,m)$ be a locally bounded flow tree.
Let $\cO$ be a bounded operator on $L^2(m)$, whose integral kernel decomposes as
\[
K_{\cO}(x,y) = \sum_{n \in \ZZ} K_n(x,y) \qquad \forall x,y \in T,
\]
in the sense of pointwise convergence. Assume that there exist constants $C>0$, $c \in (0,1)$ and $a>0$ such that
\begin{align*}
    \sup_{y \in T} \sum_{x \in T} |K_{n}(x,y)| \, (1+c^{n}d(x,y))^a \, m(x) &\le C , \\ 
    \sup_{y \in T} \sum_{x \in T} |\nabla_y K_{n}(x,y) | \, m(x) &\le C\, c^{n} . 
\end{align*}
Then, the operator $\cO$ is of weak type $(1,1)$, bounded on $L^p(m)$ for every $p\in (1,2]$, and bounded from $H^1(m)$ to $L^1(m)$.
\end{theorem}

The above result is not enough for implementing our transference strategy and obtaining $L^p$ bounds for singular integrals on arbitrary flow trees, beyond the locally doubling case. Indeed, by following the proofs in \cite{hs,LSTV} in the case $(T,m) = (\Tq,\mq)$, one would deduce $L^p$ bounds which may depend on $q$ (for example, several $q$-dependent bounds appear in the construction of Calder\'on--Zygmund decompositions in the proof of \cite[Theorem 3.1]{hs}, and similar issues appear in \cite{LSTV}). Nevertheless, with some adjustments of the arguments in \cite{hs,LSTV} we can establish the following sharper version of Theorem \ref{t: lemmahebisch} in the case of homogeneous trees.

\begin{theorem}\label{t: lemmahebisch_sharp}
Let $q \in \NN_+$, $q \geq 2$. Let $\cO$ be a bounded operator on $L^2(\mq)$, whose integral kernel decomposes as
\[
K_{\cO}(x,y) = \sum_{n \in \ZZ} K_n(x,y) \qquad \forall x,y \in \Tq,
\]
in the sense of pointwise convergence. Assume that there exist constants $C>0$, $c \in (0,1)$ and $a>0$ such that
\begin{equation}\label{f: lemmahebisch_sharp_kernelbds}
\begin{aligned}
    \sup_{y \in \Tq} \sum_{x \in \Tq} |K_{n}(x,y)| \, (1+c^{n}d(x,y))^a \, \mq(x) &\le C , \\ 
    \sup_{y \in \Tq} \sum_{x \in \Tq} |\nabla_y K_{n}(x,y) | \, \mq(x) &\le C\, c^{n} .
\end{aligned}
\end{equation}
Then, the operator $\cO$ is of weak type $(1,1)$ and bounded on $L^p(\mq)$ for every $p\in (1,2]$, where the bounds on $\cO$ only depend on $C,a,c$ and $\|\cO\|_{L^2(\mq) \to L^2(\mq)}$, but not on $q$.
\end{theorem}
\begin{proof}
By a careful inspection of the proofs in \cite{LSTV}, applied to the case of $(\Tq,\mq)$, one realises that the source of $q$-dependence in the Calder\'on--Zygmund decompositions (particularly in the constants in \cite[Theorem 3.6]{LSTV}) lies in the construction of admissible trapezoids and their dyadic splitting in \cite[Sections 3.1 and 3.2]{LSTV}: namely, the maximum number of dyadic children of a trapezoid grows linearly in $q$, and a corresponding growth appears in the maximum ratio between the measures of a trapezoid and one of its children. We now explain how to modify that construction, by expanding the collection of admissible trapezoids, so that the bounds on the number of children and the father-child measure ratio are uniform in $q$.

For any interval $I = \{ x \in \ZZ \colon a \leq x < b \}$ in $\ZZ$ (here $a,b \in \ZZ$ and $b > a$), we shall denote by $I_-$ and $I_+$ its approximate left and right halves, defined by $I_- = \{ x \in \ZZ \colon a \leq x < \lfloor (a+b)/2 \rfloor \}$ and $I_+ = \{ x \in \ZZ \colon \lfloor (a+b)/2 \rfloor \leq x < b\}$. If $I$ is not a singleton, then the sets $I_-$ and $I_+$ are disjoint nonempty intervals in $\ZZ$ whose union is $I$, and we declare them to be the \emph{children} of $I$; instead, if $I$ is a singleton, then we declare $I$ to be the only child of $I$.

Let $\Int_q = \{0,\dots,q-1\}$. We define $\Dy_q$ to be the smallest collection of subintervals of $\Int_q$ including $\Int_q$ and all its descendants (children, children of children, and so on). 
Fix an enumerator $\ord : \Tq \to \Int_q$ of $\Tq$, and, for any $x \in \Tq$ and $j \in \Int_q$, let $\succ_j(x)$ be the only successor of $x$ with $\ord(\succ_j(x)) = j$. 
For any $x \in \Tq$, any $h',h'' \in \Npos$ with $h''> h'$, and any $I \in \Dy_q$, we define
\[
R_{h',h'',I}(x) = \{ y \in \Tq \colon h' \leq d(x,y) < h'' \} \cap \bigcup_{ j \in I } \Delta_{\succ_j(x)}.
\]
Notice that, when $I = \Int_q$, we recover the sets $R_{h'}^{h''}(x)$ defined in \cite[Section 3.1]{LSTV}; on the other hand, if $I = \{j\}$ is a singleton, then $R_{1,2,I}(x) = \{\succ_j(x)\}$ is a singleton too. We declare $R_{h',h'',I}(x)$ to be an \emph{admissible trapezoid} if either $I = \Int_q$ and $2 \leq h''/h' \leq 12$, or if $2 \leq h''/h' < 4$ and $I \in \Dy_q$ is arbitrary. It is readily checked that the collection of admissible trapezoids defined here is finer than that constructed in \cite[Section 3.1]{LSTV} with parameter $\beta = 12$.

Much as in \cite[Section 3.2]{LSTV}, we now describe the dyadic children of an admissible trapezoid $R_{h',h'',I}(x)$. We distinguish a few cases:
\begin{itemize}
\item if $4 \leq h''/h' \leq 12$ (so necessarily $I = \Int_q$), then we cut the trapezoid horizontally and declare $R_{h',2h',I}(x)$ and $R_{2h',h'',I}(x)$ to be the children of $R_{h',h'',I}(x)$;
\item if $2 \leq h''/h' < 4$ and $I \in \Dy_q$ is not a singleton, then we cut the trapezoid vertically and declare $R_{h',h'',I_-}(x)$ and $R_{h',h'',I_+}(x)$ to be the children of $R_{h',h'',I}(x)$;
\item if $2 \leq h''/h' < 4$, $I = \{j\}$ is a singleton and $h_1 \geq 2$, then 
$R_{h',h'',I}(x) = R_{h'-1,h''-1,\Int_q}(\succ_j(x))$
and we reduce to the previous cases;
\item if $I = \{j\}$ is a singleton, $h_1 = 1$ and $h_2 = 3$, then we cut the trapezoid horizontally and declare $R_{1,2,I}(x)$ and $R_{2,3,I}(x) = R_{1,2,\Int_q}(\succ_j(x))$ to be the children of $R_{1,3,I}(x)$;
\item if $I = \{j\}$ is a singleton, $h_1 = 1$ and $h_2 = 2$, then $R_{1,2,I}(x)$ is a singleton and is the only child of itself.
\end{itemize}

A comparison of the above construction with that in \cite[Section 3.2]{LSTV} shows that here we have simply introduced a few ``intermediate generations'' between the original admissible trapezoids: specifically, when making vertical cuts, in place of immediately splitting a trapezoid $R_{h',h'',\Int_q}(x)$ into the $q$ trapezoids $R_{h'-1,h''-1,\Int_q}(y)$ with $y \in \succ(x)$ as in \cite{LSTV}, here we introduce as intermediate steps the trapezoids $R_{h',h'',I}(x)$ where $I \in \Dy_q$. The advantage of the present construction is that any admissible trapezoid has now at most two dyadic children, and moreover the ratio of the measures between an admissible trapezoid and a child is bounded uniformly in $q$. In addition, if $R = R_{h',h'',I}(x)$ is an admissible trapezoid, then clearly $\diam(R) \lesssim h'$ with a $q$-independent implicit constant; furthermore, if we set $R^* = \{ y \in \Tq \colon d(y,R) < h'\}$, then $R^* \subseteq R_{1,h'+h''-1,I}(x)$, and the ratio of the measures of $R^*$ and $R$ is also $q$-uniformly bounded.

Thanks to these $q$-uniform bounds in the construction of admissible trapezoids, we can now follow the arguments in \cite[Sections 3.3 and 3.4]{LSTV} to construct Calder\'on--Zygmund decompositions of functions in $L^1(\mq)$ with $q$-independent bounds. By combining this with the arguments in \cite[Section 1]{hs}, we finally deduce, for any operator $\cO$ satisfying the assumptions of Theorem \ref{t: lemmahebisch_sharp}, a weak type $(1,1)$ bound that may depend on $\|\cO\|_{2 \to 2}$ and the constants in \eqref{f: lemmahebisch_sharp_kernelbds}, but not otherwise on $q$; the same $q$-independence is then shared by the corresponding $L^p$ bounds for $p \in (1,2)$ derived from the Marcinkiewicz Interpolation Theorem.
\end{proof}

\subsection{The Riesz transform on a flow tree}

Let $(T,m)$ be a flow tree. We denote by $\Rz$ the Riesz transform on $T$, formally defined as the operator $\Rz\defeq\nabla \opL^{-1/2}$, where $\opL$ is the flow Laplacian associated to $m$. From the definition of the flow Laplacian in Section \ref{ss: spectrum} it is clear that
\begin{equation}\label{f: Rz_L2}
\|\opL^{1/2} f \|_{L^2(m)}^2 = \langle \opL f, f \rangle = \frac{1}{2} \|\nabla f\|_{L^2(m)}^2 \qquad \forall f\in L^2(m),
\end{equation}
whence it follows that $\Rz$ is bounded on $L^2(m)$ with norm $\sqrt{2}$. The purpose of this section is studying $L^p$ boundedness properties of $\Rz$ for $p \neq 2$.

\begin{lemma}\label{l: Riesz_kernel}
Let $(T,m)$ be a flow tree.  The integral kernel of $\Rz$ is 
\begin{equation}\label{f:KR}
K_{\Rz}(x,y)=\frac{1}{\sqrt{\pi}}\int_0^\infty K_{\nabla e^{-t\opL}}(x,y) \frac{dt}{\sqrt{t}} \qquad \forall x,y\in T,
\end{equation}
where the integral converges absolutely for any $x,y \in T$.
\end{lemma} 
\begin{proof}
Let $\{F_n\}_{n \in \NN}$ be the sequence of functions defined on $(0,\infty)$ by 
\begin{equation}\label{f: truncated_Riesz}
F_n(s)= \frac{1}{\sqrt{\pi}}\int_0^n e^{-ts} \frac{dt}{\sqrt{t}},\qquad s>0.
\end{equation}
Then, it is clear that $F_n \le F_{n+1}$ for every $n \in \NN$ and $F_n$ converges pointwise on $(0,\infty)$ to the function $F(s)=\frac{1}{\sqrt{s}}$. Moreover, thanks to \eqref{f: Rz_L2}, an application of the Borel functional calculus allows us to conclude that $\nabla F_n(\opL) \to \Rz$ in the strong $L^2$ operator topology (see also \cite[Proposition 4.1]{MaVa}). Hence for every $x,y\in T$,
\[
|K_{\nabla {F_{n}}(\opL)}(x,y)- K_{\Rz}(x,y)|^2 m(x) m(y) \leq \|[\nabla F_n(\opL)-\Rz] \chr_{\{y\}}\|_2^2,
\]
which tends to $0$ as $n\to\infty$. It follows that 
\[
\lim_{n\to \infty}K_{\nabla {F_{n}}(\opL)}(x, y)=K_{\Rz}(x, y),\qquad \forall x,y\in T. 
\] 
On the other hand, by the estimates \eqref{f: levelestimates} we deduce that 
\[
\lim_{n\to \infty}  \frac{1}{\sqrt{\pi}}\int_0^n K_{\nabla e^{-t\opL}}(x,y) \frac{dt}{\sqrt{t}}= \frac{1}{\sqrt{\pi}}\int_0^\infty K_{\nabla e^{-t\opL}}(x,y) \frac{dt}{\sqrt{t}},
\]
with absolute converge of the integrals, and \eqref{f:KR} follows.
\end{proof}

We now proceed with the proof of the $L^p$ boundedness properties for $p \leq 2$ of the Riesz transform $\Rz$ stated in Theorem \ref{t: Riesz}, as well as the corresponding endpoint bounds.

\begin{proof}[Proof of Theorem \ref{t: Riesz}, case $p \leq 2$ and endpoint bounds]
Let $(T,m)$ be a flow tree.
Denote by $\Rz^{(0)}$ and $\Rz^{(\infty)}$ the operators defined by 
\begin{align*}
\Rz^{(0)} &=\frac{1}{\sqrt{\pi}}\int_0^1 \nabla e^{-t\opL} \frac{dt}{\sqrt{t}}, \\ 
\Rz^{(\infty)} &=\frac{1}{\sqrt{\pi}}\int_1^\infty \nabla e^{-t\opL} \frac{dt}{\sqrt{t}},
\end{align*}
so that $\Rz=\Rz^{(0)}+\Rz^{(\infty)}$. 

Since $\nabla e^{-t\opL}$ is bounded on $L^1(m)$ and $L^\infty(m)$ uniformly in $t \in (0,\infty)$, as both $\nabla$ and $e^{-t\opL}$ are, it follows by Minkowski's integral inequality that $\Rz^{(0)}$ is bounded on $L^p(m)$ for every $p \in [1,\infty]$. Hence $\Rz^{(\infty)} = \Rz - \Rz^{(0)}$ is $L^2$-bounded.

Assume now that $m$ is locally doubling. In this case, we can follow the proof given in \cite{MSV} for the case of a homogeneous tree to deduce the required boundedness properties of $\Rz^{(\infty)}$. Namely, we consider the dyadic decomposition
\[
\Rz^{(\infty)} = \sum_{n=0}^\infty \frac{1}{\sqrt{\pi}}\int_{2^n}^{2^{n+1}} \nabla e^{-t\opL} \frac{dt}{\sqrt{t}} \eqdef \sum_{n=0}^\infty \Rz^{(\infty)}_n,
\]
and from the heat kernel bounds of Theorem \ref{t: est on T} we deduce the following kernel estimates for the dyadic pieces:
\begin{equation}\label{f: Rz_dyad_bds}
\begin{aligned}
   \sup_{y \in T} \sum_{x \in T} |K_{\Rz^{(\infty)}_n}(x,y)| \, e^{\varepsilon d(x,y)/2^{n/2}} m(x) &\lesssim 1, \\
    \sup_{y \in T} \sum_{x \in T} |\nabla_y K_{\Rz^{(\infty)}_n}(x,y)| \, e^{\varepsilon d(x,y)/2^{n/2}} m(x) &\lesssim 2^{-n/2}
\end{aligned}
\end{equation}
(see \cite[Lemma 3.5]{MSV}). Thus, from Theorem \ref{t: lemmahebisch} we deduce that $\Rz^{(\infty)}$ is bounded from $H^1(m)$ to $L^1(m)$, is of weak type $(1,1)$ and bounded on $L^p(m)$ for all $p\in (1,2]$.	

It remains to prove the $L^p$ boundedness for $p \in (1,2)$ of $\Rz^{(\infty)}$ on an arbitrary flow tree $(T,m)$. We will achieve this by means of the transference results in the previous section.
Namely, if we set
\[
\Rz^{(N)} = \frac{1}{\sqrt{\pi}}\int_1^{2^N} \nabla e^{-t\opL} \frac{dt}{\sqrt{t}},
\]
then $\Rz^{(N)} \to \Rz^{(\infty)}$ as $N \to \infty$ in the strong operator topology of $L^2(m)$; thus
\[
\left\| \Rz^{(\infty)} \right\|_{L^p(m) \to L^p(m)} \leq \sup_{N \in \Npos}  \left\| \Rz^{(N)} \right\|_{L^p(m) \to L^p(m)}.
\]
On the other hand, 
we can write $\Rz^{(N)} = \nabla G_N(\opL)$, where
\[
G_N(z) = \frac{1}{\sqrt{\pi}}\int_1^{2^N}  e^{-tz} \frac{dt}{\sqrt{t}}
\]
is an entire function; thus, the universal transference result of Theorem \ref{t: unif_transf} applies, and 
\[
\sup_{N \in \Npos} \left\| \Rz^{(N)} \right\|_{L^p(m) \to L^p(m)} 
\leq \sup_{N \in \Npos} \sup_{q} \left\| \Rq^{(N)} \right\|_{L^p(\mq) \to L^p(\mq)},
\]
where $\Rq^{(N)} = \nabla G_N(\opL)$. Thus, we are reduced to proving, for $p \in (1,2]$, an $L^p$ bound for the $\Rq^{(N)}$ which is uniform in both $q$ and $N$.

On the other hand, much as before, we can consider the dyadic decomposition
\[
\Rq^{(N)} =\frac{1}{\sqrt{\pi}}\int_1^{2^N} \nabla_{\Tq} e^{-t\Lq} \frac{dt}{\sqrt{t}} = \sum_{n=0}^{N-1} \frac{1}{\sqrt{\pi}}\int_{2^n}^{2^{n+1}} \nabla_{\Tq} e^{-t\Lq} \frac{dt}{\sqrt{t}},
\]
and for each dyadic piece we have the kernel estimates of \eqref{f: Rz_dyad_bds}, which hold uniformly in $N$ and $q$. Moreover trivially $\|\Rq^{(N)}\|_{L^2(\mq) \to L^2(\mq)} \leq \sqrt{2}$ (cf., e.g., \cite[Proposition 4.1]{MaVa}). So the operators $\Rq^{(N)}$ satisfy the assumptions of Theorem \ref{t: lemmahebisch_sharp} uniformly in $q$ and $N$, whence we deduce their $L^p$ boundedness for $p \in (1,2]$ with the required uniformity.
\end{proof}

To prove the case $p>2$ of Theorem \ref{t: Riesz}
we shall employ a different strategy, which is partly an adaptation of the strategy used in \cite{LMSTV} in the case of homogeneous trees.

We start by showing that the integral kernel of the Riesz transform on a flow tree is related to the corresponding kernel on a quotient. 

\begin{prop}\label{p: rel Ker R}
Let $\pi : (T_1,m_1) \to (T_2,m_2)$ be a flow submersion. Let $\Rz_j$ denote the Riesz transform on $(T_j,m_j)$ for $j=1,2$.
Then, for every $x,y\in T_2$ and $\overline{x} \in \pi^{-1}\{x\}$,
\begin{equation}\label{f: rel Ker R}
\begin{aligned}
K_{\Rz_2}(x,y) &= \frac{1}{m_2(y)} \sum_{z \in \pi^{-1}\{y\}} K_{\Rz_1}(\overline{x},z) \, m_1(z), \\ 
K_{\Rz_2^*}(x,y) &= \frac{1}{m_2(y)} \sum_{z \in \pi^{-1}\{y\}} K_{\Rz_1^*}(\overline{x},z) \, m_1(z).
\end{aligned}
\end{equation}
\end{prop} 
\begin{proof}
The idea is to exploit Proposition \ref{p: K2K1}, however the Riesz transforms $\Rz_j$ are unbounded on $L^1(m_j)$ and $L^\infty(m_j)$, so the theory of $\pi$-compatible operators does not directly apply to them. On the other hand, for any $N \in \Npos$, the truncated integrals $F_N$ from \eqref{f: truncated_Riesz} are entire functions, thus the operators
\[
\Rz_j^{(N)} \defeq \nabla_j F_N(\opL_j) = \frac{1}{\sqrt{\pi}} \int_0^N \nabla_j e^{-t\opL_j} \,\frac{dt}{\sqrt{t}}
\]
are in $\Bdd(m_j)$. Therefore, from Proposition \ref{p: Intral} we deduce that $\Rz_1^{(N)} \in \overline{\Comp_\fin(\pi)}$ and $\pi(\Rz_1^{(N)}) = \Rz_2^{(N)}$.

Thus, by Proposition \ref{p: K2K1},
\[\begin{split}
K_{\Rz^{(N)}_2}(x,y) 
&= \frac{1}{m_2(y)}\sum_{z \in \pi^{-1}\{y\}} K_{\Rz_1^{(N)}}(\overline{x},z) \, m_1(z) \\
&= \frac{1}{m_2(x)}\sum_{z \in \pi^{-1}\{x\}} K_{\Rz_1^{(N)}}(z,\overline{y}) \, m_1(z)
\end{split}\]
for all $x,y \in T$, $\overline{x}\in\pi^{-1}\{x\}$ and $\overline{y} \in \pi^{-1}\{y\}$,
i.e.,
\begin{equation}\label{f: truncated_id}
\begin{split}
\int_0^N  K_{\nabla_2 e^{-t{\opL_2}}}(x,y) \, \frac{dt}{\sqrt{t}} 
&= \frac{1}{m_2(y)}\sum_{z \in \pi^{-1}\{y\}} \int_0^N  K_{\nabla_1 e^{-t\opL_1}}(\overline{x},z) \, m_1(z) \, \frac{dt}{\sqrt{t}} \\
&= \frac{1}{m_2(x)}\sum_{z \in \pi^{-1}\{x\}} \int_0^N  K_{\nabla_1 e^{-t\opL_1}}(z,\overline{y}) \, m_1(z) \, \frac{dt}{\sqrt{t}}.
\end{split}
\end{equation}
Notice now that, by \eqref{f: levelestimates},
\[
\sum_{z \in \pi^{-1}\{y\}} |K_{\nabla e^{-t\opL_1}}(\overline{x},z)| \, m_1(z), \ \sum_{z \in \pi^{-1}\{y\}} |K_{\nabla e^{-t\opL_1}}(z,\overline{y})| \, m_1(z) \lesssim \frac{1}{1+t}.
\]
Hence, by Fubini's Theorem and the Dominated Convergence Theorem, we can take the limit as $N \to \infty$ in \eqref{f: truncated_id} and deduce that
\[\begin{split}
\int_0^\infty  K_{\nabla_2 e^{-t{\opL_2}}}(x,y) \, \frac{dt}{\sqrt{t}} 
&= \frac{1}{m_2(y)}\sum_{z \in \pi^{-1}\{y\}} \int_0^\infty  K_{\nabla_1 e^{-t\opL_1}}(\overline{x},z) \, m_1(z) \, \frac{dt}{\sqrt{t}} \\
&= \frac{1}{m_2(x)}\sum_{z \in \pi^{-1}\{x\}} \int_0^\infty  K_{\nabla_1 e^{-t\opL_1}}(z,\overline{y}) \, m_1(z) \, \frac{dt}{\sqrt{t}},
\end{split}\]
that is, by Lemma \ref{l: Riesz_kernel},
\[\begin{split}
K_{\Rz_2}(x,y) 
&= \frac{1}{m(y)} \sum_{z \in \pi^{-1}\{y\}} K_{\Rz_1}(\overline{x},z) \, m_1(z) \\
&= \frac{1}{m(x)} \sum_{z \in \pi^{-1}\{x\}} K_{\Rz_1}(z,\overline{y}) \, m_1(z).
\end{split}\]
In light of the relation between the kernels of an operator and its adjoint, this proves \eqref{f: rel Ker R}.
\end{proof}

In the following statement, we use the notation
\[
\tilde\Sigma_n = \begin{cases}
\Sigma^n &\text{if } n \geq 0,\\
(\Sigma^*)^{-n} &\text{if } n < 0.
\end{cases}
\]
Moreover,
\begin{equation}\label{f: conv_Rzdiff_Z}
\tilde k_{\ZZ}(n)=\frac{2\sqrt{2}}{\pi}\frac{n}{n^2-1/4}, \qquad n \in \ZZ,
\end{equation}
is the convolution kernel of $\tilde \Rz_{\ZZ} \defeq \tilde\nabla_\ZZ \Delta_\ZZ^{-1/2}$, i.e., the skewsymmetric part of the discrete Hilbert transform on $\ZZ$ (see \cite{ADP} or \cite[Equation (2.5)]{LMSTV}).

\begin{prop}\label{p: Riesz_somma}
Let $(T,m)$ be a flow tree. Then, for all $f \in c_{00}(T)$,
\begin{equation}\label{f: R-R*_somma}
(\Rz-\Rz^*) f =\sum_{n \in \ZZ} \tilde k_{\ZZ}(n) \tilde \Sigma_{-n} f
\end{equation}
in the sense of pointwise convergence on $T$.
\end{prop}
\begin{proof}
It is straightforward to check that the identity \eqref{f: R-R*_somma} can be equivalently rewritten in terms of integral kernels as follows:
\begin{equation}\label{f: kRR*T}
K_{\Rz - \Rz^*}(x,y) = \begin{cases}
\frac{\Tilde{k}_{\ZZ}(\ell(x)-\ell(y))}{m(y)} & \text{if } x < y, \\ 
\frac{\Tilde{k}_{\ZZ}(\ell(x)-\ell(y))}{m(x)} & \text{if } x > y, \\
0 & \text{otherwise}.
\end{cases}
\end{equation}
Our proof will focus on verifying this kernel identity.

First of all, from \cite[Equation (4.4)]{LMSTV} we already know that the result holds true on any homogeneous tree $\Tq$, thus
\begin{equation}\label{f: kRR*Tq}
K_{\Rq - \Rq^*}(x,y) = \begin{cases}
\frac{\Tilde{k}_{\ZZ}(\ell(x)-\ell(y))}{\mq(y)} & \text{if } x < y, \\ 
\frac{\Tilde{k}_{\ZZ}(\ell(x)-\ell(y))}{\mq(x)} & \text{if } x > y, \\
0 & \text{otherwise}.
\end{cases}
\end{equation}

Assume now that $(T,m)$ is a $q$-uniformly rational flow tree. Then, by Proposition \ref{p: comp unif rat}, there exists a level-preserving flow submersion $\pi : (\Tq, \mq) \to (T,m)$.
So, by \eqref{f: kRR*Tq} and Proposition \ref{p: rel Ker R}, given $x,y \in T$ and $\overline{x} \in \pi^{-1}\{x\}$,
\[\begin{split}
&K_{\Rz- \Rz^*}(x,y)\\
&= \frac{1}{m(y)}\left[\sum_{z \in \pi^{-1}\{y\} \colon z > \overline{x}}\Tilde{k}_{\ZZ}(\ell(\overline{x})-\ell(z))+\sum_{z \in \pi^{-1}\{y\} \colon z < \overline{x}}\Tilde{k}_{\ZZ}(\ell(\overline{x})-\ell(z)) \frac{\mq(z)}{\mq(\overline{x})}\right] \\
&= \frac{\Tilde{k}_{\ZZ}(\ell(x)-\ell(y))}{m(y)}\left[\sum_{z \in \pi^{-1}\{y\} \colon z > \overline{x}} 1 +\sum_{z \in \pi^{-1}\{y\} \colon z < \overline{x}} \frac{\mq(z)}{\mq(\overline{x})}\right].
\end{split}\]
Now, if $x \not < y$ and $y \not < x$, by the strict monotonicity of $\pi$ there cannot exist a $z \in \pi^{-1}(y)$ such that $\overline{x} < z$ or $z < \overline{x}$, and therefore from the above formula we deduce that $K_{\Rz-\Rz^*}(x,y) = 0$ in this case. If instead $y > x$, then $y = \pred^{\ell(y)-\ell(x)}(x)$, so the only $z \in \pi^{-1}\{y\}$ comparable with $\overline{x}$ is given by $z = \pred^{\ell(y)-\ell(x)}(\overline{x})$, and the above formula gives
\begin{equation}\label{f: RR*diffpart}
K_{\Rz- \Rz^*}(x,y) = \frac{\Tilde{k}_{\ZZ}(\ell(x)-\ell(y))}{m(y)}.
\end{equation}
Finally, if $y < x$, as $\Rz-\Rz^*$ is skewadjoint,
\[
K_{\Rz- \Rz^*}(x,y) = -\overline{K_{\Rz- \Rz^*}(y,x)} = -\overline{\frac{\Tilde{k}_{\ZZ}(\ell(y)-\ell(x))}{m(x)}} = \frac{\Tilde{k}_{\ZZ}(\ell(x)-\ell(y))}{m(x)},
\]
where we applied \eqref{f: RR*diffpart} with $x$ and $y$ swapped, and used that $\tilde k_{\ZZ}$ is real and odd.
This concludes the proof of \eqref{f: kRR*T} in the case $(T,m)$ is uniformly rational.

Finally, consider an arbitrary flow tree $(T,m)$. By Corollary \ref{c: point approx}, there exists an approximating sequence $((T_j,m_j))_j$ where each $m_j$ is uniformly rational. So we know that \eqref{f: kRR*T} holds for each $(T,m_j)$, i.e.,
\begin{equation}\label{f: kRR*Tj}
K_{\Rz_j - \Rz_j^*}(x,y) = \begin{cases}
\frac{\Tilde{k}_{\ZZ}(\ell(x)-\ell(y))}{m_j(y)} & \text{if } x < y, \\ 
\frac{\Tilde{k}_{\ZZ}(\ell(x)-\ell(y))}{m_j(x)} & \text{if } x > y, \\
0 & \text{otherwise}
\end{cases}
\end{equation}
for all $x,y \in T_j$. Since the right-hand side of \eqref{f: kRR*Tj} clearly converges to that of \eqref{f: kRR*T} as $j \to \infty$, it only remains to check that $K_{\Rz_j - \Rz_j^*} \to K_{\Rz - \Rz^*}$ pointwise on $T \times T$, or equivalently, that $K_{\Rz_j} \to K_{\Rz}$ pointwise. Indeed, by Lemma \ref{l: Riesz_kernel},
\[
K_{\Rz}(x,y)=\frac{1}{\sqrt{\pi}}\int_0^\infty K_{\nabla e^{-t\opL}}(x,y) \frac{dt}{\sqrt{t}}, \quad K_{\Rz_j}(x,y)=\frac{1}{\sqrt{\pi}}\int_0^\infty K_{\nabla e^{-t\opL_j}}(x,y) \frac{dt}{\sqrt{t}}
\]
and we know from Proposition \ref{p: conv_hker_adv} that $K_{\nabla e^{-t\opL_j}}(x,y) \to K_{\nabla e^{-t\opL}}(x,y)$ as $j \to \infty$ for any $x,y \in T$ and $t > 0$. Thus, in light of the bound \eqref{f: levelestimates}, the desired convergence result follows from the Dominated Convergence Theorem.
\end{proof}

After establishing the crucial identity \eqref{f: R-R*_somma}, we can now proceed by following closely the strategy of \cite{LMSTV} to derive $L^p$ boundedness properties of $\Rz-\Rz^*$ from those of $\tilde\Rz_\ZZ$. For the reader's convenience, we briefly describe the main steps.

We need to introduce some more notation, analogous to the one in \cite[Section 3]{LMSTV}. Let $(T,m)$ be a flow tree. The flow measure $m$ determines a Borel measure $\nu$ on the punctured boundary $\Omega \defeq \partial T \setminus \{\myth\}$, which is uniquely defined by the condition
\[
\nu(\Omega_x) = m(x) \qquad\forall x \in T,
\]
where $\Omega_x \defeq \{\omega\in \Omega \colon x \in (\omega,\myth)\}$.
Correspondingly, we equip the product $\Omega \times \ZZ$ with the product measure $\nu \times \#$.

We define the lifting operator $\Psi: \CC^{T} \to \CC^{\Omega \times \ZZ}$ by 
\[
\Psi f(\omega,n)=f(\omega_n),
\]
where $\omega_n$ is the unique vertex of level $n$ which belongs to $(\omega, \myth)$. Furthermore, we define the shift operator $\sigma : \CC^{\Omega \times \ZZ} \to \CC^{\Omega \times \ZZ}$ by 
\[
\sigma g(\omega,n) = g(\omega,n+1) \qquad \omega \in \Omega, \ n \in \ZZ.
\]
Easy computations show the following properties of the lifting operator $\Psi$ and the shift operator $\sigma$, which can be found in \cite[Propositions 3.1, 3.2 and 3.4]{LMSTV} in the case of homogeneous trees.

\begin{prop}\label{p: PhiPhi*}
The following properties hold.
\begin{enumerate}[label=(\roman*)]
\item\label{en:phiphiiso} $\Psi$ is an isometric embedding from $L^p(m)$ to $L^p(\nu \times \#)$ for every $p\in [1,\infty]$. Correspondingly, the adjoint operator $\Psi^*$ is bounded from $L^p(\nu \times \#)$ to $L^p(m)$ with norm $1$. 
\item\label{en:phiphiLp} For any $p \in [1,\infty]$, an operator $A$ is bounded on $\ell^p(\ZZ)$ if and only if $\id_\Omega \otimes A$ is bounded on $L^p(\nu \times \#)$, and their norms are the same.
\item\label{en:phiphishift} For any $n \in \ZZ$,
\[
\tilde\Sigma_n = \Psi^* \sigma^n \Psi.
\]
\end{enumerate}
\end{prop}

We can now proceed to prove the $L^p$ boundedness of the skewsymmetric part of the Riesz transform.

\begin{prop}\label{p: Rzdiff_Lp}
Let $(T,m)$ be a flow tree. Then, for all $p \in (1,\infty)$, the operator $\Rz-\Rz^*$ is bounded on $L^p(m)$, and more precisely
\begin{equation}\label{f: Rzdiff_Lp}
\|\Rz-\Rz^*\|_{L^p(m) \to L^p(m)} \leq \| \tilde\Rz_{\ZZ} \|_{\ell^p(\ZZ) \to \ell^p(\ZZ)}.
\end{equation}
\end{prop}
\begin{proof}
Recall that $\tilde k_\ZZ$ is the convolution kernel of $\tilde\Rz_{\ZZ}$.
In light of Proposition \ref{p: PhiPhi*}\ref{en:phiphishift}, we can rewrite the identity from Proposition \ref{p: Riesz_somma} as
\[
(\Rz-\Rz^*)f = \sum_{n \in \ZZ} \Tilde{k}_\ZZ(n) \Psi^* \sigma^{-n} \Psi f = \Psi^*(\id_{\Omega} \otimes \tilde\Rz_{\ZZ})\Psi f.
\]
From Proposition \ref{p: PhiPhi*}\ref{en:phiphiiso}-\ref{en:phiphiLp} we therefore deduce the bound \eqref{f: Rzdiff_Lp}.
As $\tilde\Rz_{\ZZ}$ is bounded on $\ell^p(\ZZ)$ for any $p \in (1,\infty)$ (see, e.g., \cite{HSC} or \cite[Proposition 2.3]{LMSTV}), the desired result follows.
\end{proof}

Finally, we complete the proof of Theorem \ref{t: Riesz}.

\begin{proof}[Proof of Theorem \ref{t: Riesz}, case $p>2$]
We already know from the case $p\leq 2$ of Theorem \ref{t: Riesz} that $\Rz$ is bounded on $L^p(m)$ for all $p \in (1,2)$, which means that $\Rz^*$ is bounded on $L^p(m)$ for all $p \in (2,\infty)$. As $\Rz-\Rz^*$ is also bounded on $L^p(m)$ for $p \in (2,\infty)$ by Proposition \ref{p: Rzdiff_Lp}, the sum $\Rz = \Rz^* + (\Rz-\Rz^*)$ must also be bounded on $L^p(m)$ for $p \in (2,\infty)$.
\end{proof}

We end the section with a negative boundedness result for $\Rz$ and its adjoint.

\begin{prop}\label{p: RunboundedLinftyBMO}
Let $(T,m)$ be a locally doubling flow tree which is not isomorphic to $\ZZ$, i.e., such that $q(x) > 1$ for some $x \in T$. Then, $\Rz^*$ is unbounded from $H^1(m)$ to $L^1(m)$ and, consequently, $\Rz$ is unbounded from $L^\infty(m)$ to $BMO(m)$.
\end{prop}
\begin{proof}
Since $\Rz$ is bounded from $H^1(m)$ to $L^1(m)$ by Theorem \ref{t: Riesz}, it suffices to prove the unboundedness from $H^1(m)$ to $L^1(m)$ of $\Rz-\Rz^*$.

By our assumption on $T$, we can choose $x_1 \in T$ such that $q(\pred(x_1)) \ge 2$, and denote by $x_2$ a vertex in $\succ(\pred(x_1))$ different from $x_1$.  Define $f=\chr_{\Delta_{x_1}}$ and $g=\frac{\chr_{\{x_1\}}}{m(x_1)}-\frac{\chr_{\{x_2\}}}{m(x_2)}$, and observe that $f \in L^\infty(m)$ and $g\in H^1(m)$. By \eqref{f: kRR*T} and \eqref{f: conv_Rzdiff_Z},
\[\begin{split}
|\langle f, (\Rz-\Rz^*) g \rangle|
&= \left| \sum_{x \colon x \le x_1} \sum_{y \in T}K_{\Rz-\Rz^*}(x,y) \left(\frac{\chr_{\{x_1\}}(y)}{m(x_1)}-\frac{\chr_{\{x_2\}}(y)}{m(x_2)}\right) \, m(y) \, m(x) \right| \\ 
&= \left| \sum_{x \colon x \le x_1} K_{\Rz-\Rz^*}(x,x_1) \, m(x) \right| \\ 
&\approx \sum_{x \colon x \le x_1}\frac{1}{(d(x,x_1)+1)m(x_1)} \, m(x) \\ 
&= \sum_{n=0}^\infty \frac{1}{n+1}
=\infty.
\end{split}\]
This shows that $(\Rz-\Rz^*) g$ does not lie in $L^1(m)$, as desired.
\end{proof}

\subsection{Spectral multipliers of the flow Laplacian}

In this section we prove Theorem \ref{t: MHmultipliers}, establishing $L^p$ boundedness properties for functions $F(\opL)$ of the flow Laplacian under suitable scale-invariant smoothness conditions on $F$.

One of the fundamental ingredients in the proof is the following scale-invariant version of the weighted $L^1$ estimate of Proposition \ref{p: stimaL1Tsenzasupporto}, which corresponds to the estimate on $\ZZ$ from Lemma \ref{l: FtDeltaZconsupporto}.

\begin{prop}\label{p: stimaL1Tconsupporto}
Let $(T,m)$ be a flow tree.
For every $\alpha \in [0,\infty)$ and $\beta>\alpha+3/2$, for every $F\in L^2_{\beta}(\RR)$ supported in $[1/4,7/4]$, 
\[
\sup_{t\geq 1} \sup_{y\in T} \sum_{x\in T} |K_{F(t\opL)}(x,y)| \left(1+\frac{d(x,y)}{\sqrt{t}}\right)^{\alpha}\,m(x)
\lesssim_{\alpha,\beta} \|F\|_{L^2_{\beta}}.
\]
\end{prop}
\begin{proof}
Let $\varepsilon > 0$ be such that $\beta = \alpha+3/2+\varepsilon$.

Let $F \in L^2_\beta(\RR)$ and $t\geq 1$. By Proposition \ref{p: stimaL1Tsenzasupporto}, we can apply the estimate \eqref{f: stimaL1pesata_trans} to the function $F(t\cdot) \in L^2_\beta(\RR)$ and the weight $w(n) = (1+n/\sqrt{t})^\alpha$, thus
\[\begin{split}
&\sup_{y \in T} \sum_{x \in T} |K_{F(t\opL)}(x,y)| \left(1+\frac{d(x,y)}{\sqrt{t}}\right)^{\alpha} \,m(x) \\
&\lesssim \sum_{n\in\NN} n \left(1+\frac{n}{\sqrt t}\right)^{\alpha} |\tilde\nabla_{\ZZ} k_{F(t\Delta_{\ZZ})}(n)| \\
&\leq \sqrt{t} \sum_{n\in\NN} \left(1+\frac{n}{\sqrt t}\right)^{\alpha+1} |\tilde\nabla_{\ZZ} k_{F(t\Delta_{\ZZ})}(n)|.
\end{split}\]
As $\varepsilon > 0$, 
the Cauchy--Schwarz inequality and Lemma \ref{l: FtDeltaZconsupporto} then give
\[\begin{split}
&\sup_{y \in T} \sum_{x\in T} |K_{F(t\opL)}(x,y)| \left(1+\frac{d(x,y)}{\sqrt t}\right)^{\alpha} \,m(x)\\
&\leq \sqrt{t} \left(\sum_{n\in\NN}  \left(1+\frac{n}{\sqrt t}\right)^{-1-2\varepsilon} \right)^{1/2} \left(\sum_{n\in\NN} \left[ \left(1+\frac{n}{\sqrt t}\right)^{3/2+\alpha+\varepsilon}|\tilde\nabla_{\ZZ}k_{F(t\Delta_{\ZZ})}(n)|\right]^2  \right)^{1/2}\\
&\lesssim_{\alpha,\beta} \|F\|_{L^2_\beta},
\end{split}\]
where we used that
\[
\sum_{n\in\NN} \left(1+\frac{n}{\sqrt t}\right)^{-1-2\varepsilon} \approx_\varepsilon t^{1/2}
\]
for $t \geq 1$.
\end{proof}

The above estimate, combined with the gradient heat kernel bounds of Theorem \ref{t: est on T}, gives an $L^1$ estimate for the gradient of the integral kernel of $F(t\opL)$. 

\begin{corollary}\label{c: stimagradiente}
Let $(T,m)$ be a flow tree.
For every $\varepsilon >0$, every $F\in L^2_{3/2+\varepsilon}(\RR)$ supported in $[1/4,7/4]$ and every $t\geq 1$,
\[
\sup_{y\in T} \sum_{x\in T} |\nabla_yK_{F(t\opL)}(x,y)| \, m(x)
\lesssim_{\varepsilon} t^{-1/2} \, \|F\|_{L^2_{3/2+\varepsilon}}.
\]
\end{corollary}
\begin{proof}
Define $G(\lambda)=F(\lambda) \, e^{\lambda}$ and write $F(t\opL)=G(t\opL) \, e^{-t\opL}$, so that
\[
K_{F(t\opL)}(x,y)=\sum_{z\in T}K_{G(t\opL)}(x,z)K_{ e^{-t\opL}}(z,y) \, m(z),
\]
and
\[
\nabla_y K_{F(t\opL)}(x,y)=\sum_{z\in T}K_{G(t\opL)}(x,z)\nabla_yK_{ e^{-t\opL}}(z,y) \, m(z).
\]
It follows that 
\[\begin{split}
&\sup_{y\in T}\sum_{x\in T }|\nabla_yK_{F(t\opL)}(x,y)| \, m(x)\\
&\leq \sup_{z\in T}\sum_{x\in T }|K_{G(t\opL)}(x,z)| \, m(x) 
\, \sup_{y\in T}\sum_{z\in T}|\nabla_yK_{e^{-t\opL}}(z,y)| \,m(z)\\
&\lesssim_\varepsilon \|G\|_{L^2_{3/2+\varepsilon}} \, t^{-1/2}\\
&\lesssim_\varepsilon t^{-1/2}\,\|F\|_{L^2_{3/2+\varepsilon}},
\end{split}\]
where we applied Proposition \ref{p: stimaL1Tconsupporto} to the function $G$ with $\alpha = 0$ and $\beta = 3/2+\varepsilon$, and the gradient heat kernel bound from Theorem \ref{t: est on T}.
\end{proof}

We can now prove Theorem \ref{t: MHmultipliers}. To do so, we shall use the modulation operator $\modE : \CC^T \to \CC^T$ defined by
\[
(\modE f)(x)=(-1)^{\ell(x)}f(x)\qquad x\in T, f\in \CC^T.
\]
It is easily seen that $\modE$ is selfadjoint and involutive, and preserves the norms in $L^p(m)$, $p\in [1,\infty]$, and $L^{1,\infty}(m)$. In addition,
\[
\modE \Sigma \modE= - \Sigma, \qquad \modE \Sigma^* \modE = -\Sigma^*, \qquad \modE \opAv \modE = -\opAv
\]
where $\opAv = (\Sigma+\Sigma^*)/2$ is as in \eqref{f: mathcalA}, which implies that
\begin{equation}\label{f: modul_opL}
\modE \opL \modE = 2I - \opL.
\end{equation}
The latter identity shows why it is natural that the assumption on $F$ in Theorem \ref{t: MHmultipliers}\ref{en: MHmultipliers_twist} is invariant under the change of variables $\lambda \mapsto 2-\lambda$.

\begin{proof}[Proof of Theorem \ref{t: MHmultipliers}]
Let us first prove part \ref{en: MHmultipliers_std}; here we assume that $(T,m)$ is locally doubling. Using smooth cutoff functions, we write $F$ as a sum $F=F^{(0)}+F^{(1)}$, with $F^{(0)}$ supported in $(-\infty,1/2)$ and $F^{(1)}$ supported in $(1/4,\infty)$.
From the assumption \eqref{f: MH_cond_std} on $F$ it follows that
\begin{equation}\label{eq:MHcond_red}
\sup_{t>0} \|F^{(0)}(t\cdot)\chi\|_{L^2_s} < \infty, \quad \|F^{(1)}\|_{L^2_s} < \infty,
\end{equation}
In particular, by Theorem \ref{t: multipliers}, $F^{(1)}(\opL)$ is bounded on $L^p(m)$ for every $p\in [1,\infty]$. 

We now consider $F^{(0)}$ and choose a function $\phi\in C^\infty_c(\RR)$ such that $\supp \phi \subseteq (1/4,1)$ and $\sum_{\ell\geq 0}\phi(2^\ell \lambda)=1$ for every $\lambda \in (0,1/2)$. Define $F^{(0)}_\ell =F^{(0)} (2^{-\ell}\cdot)\phi$. By Proposition \ref{p: stimaL1Tconsupporto} and Corollary \ref{c: stimagradiente} applied with $t=2^\ell$,	
\begin{equation}\label{f: CZ_cond_mult}
\begin{aligned}
\sup_{y\in T}\sum_{x\in T}|K_{F^{(0)}_\ell(2^\ell\opL)}(x,y)| \left(1+\frac{d(x,y)}{2^{\ell/2}}\right)^{\varepsilon} \, m(x)
&\lesssim_\varepsilon \|F^{(0)}(2^{-\ell}\cdot)\phi \|_{L^2_{s}}, \\
\sup_{y\in T}\sum_{x\in T}|\nabla_yK_{F^{(0)}_\ell(2^{\ell}\opL)}(x,y)| \, m(x)
&\lesssim_{\varepsilon} 2^{-\ell/2} \,\|  F^{(0)}(2^{-\ell}\cdot)\phi \|_{L^2_{s}}
\end{aligned}
\end{equation}
if $s>3/2+\varepsilon$. As
\[
\sup_{\ell \geq 0} \| F^{(0)}(2^{-\ell}\cdot)\phi \|_{L^2_{s}} \lesssim \sup_{t>0} \|F^{(0)}(t\cdot) \chi\|_{L^2_s} < \infty,
\]
it follows that the decomposition $F^{(0)}(\opL)=\sum_{\ell\geq 0} F^{(0)}_\ell(2^\ell\opL)$ satisfies the integral conditions of Theorem \ref{t: lemmahebisch}, whence we deduce that  $F^{(0)}(\opL)$ is bounded from $H^1(m)$ to $L^1(m)$, from $L^1(m)$ to $L^{1,\infty}(m)$, and on $L^p(m)$ for $p\in (1,2]$; thus the same boundedness properties are shared by $F(\opL) = F^{(0)}(\opL) + F^{(1)}(\opL)$. As $F(\opL)^* = \overline{F}(\opL)$ and $\overline{F}$ satisfies the same smoothness assumptions as $F$, by duality we also deduce the boundedness of $F(\opL)$ from $L^{\infty}(m)$ to $BMO(m)$ and on $L^p(m)$, $p\in [2,\infty)$.

Let us now prove part \ref{en: MHmultipliers_twist} when $(T,m)$ is locally doubling. Here, by using cutoff functions, we can decompose $F = F_0 + F_2$, where $\supp F_0 \subseteq (-\infty,5/4)$ and $\supp F_2 \subseteq (3/4,\infty)$. From \eqref{f: MH_cond_twist} it then follows that
\[
\sup_{t>0} \|F_0(t\cdot) \chi\|_{L^2_s} < \infty, \qquad \sup_{t>0} \|F_2(2-t\cdot) \chi\|_{L^2_s} < \infty. 
\]
From part \ref{en: MHmultipliers_std} we then deduce that $F_0(\opL)$ and $F_2(2-\opL)$ are of weak type $(1,1)$ and bounded on $L^p(m)$ for all $p \in (1,\infty)$. On the other hand, from \eqref{f: modul_opL} we deduce that $F_2(\opL) = \modE F_2(2-\opL) \modE$; since $\modE$ preserves the norms in $L^p(m)$, $p\in [1,\infty]$, and $L^{1,\infty}(m)$, the operator $F_2(\opL)$ inherits the corresponding boundedness properties of $F_2(2-\opL)$. Therefore also $F(\opL) = F_0(\opL) + F_2(\opL)$ is of weak type $(1,1)$ and bounded on $L^p(m)$ for $1 < p < \infty$, as desired.

It remains to discuss part \ref{en: MHmultipliers_twist} for any arbitrary flow tree $(T,m)$; in this case we only need to prove the $L^p$ boundedness of $F(\opL)$, $1 < p < \infty$, whenever $F$ satisfies the condition \eqref{f: MH_cond_twist}, and actually, by duality, it is enough to consider the case $p \leq 2$. Moreover, the same argument as above, using the decomposition $F = F_0 + F_2$ and the modulation operator $\modE$, allows us to reduce to the case where $F$ satisfies the stronger condition \eqref{f: MH_cond_std}.

In this case, as in the proof of part \ref{en: MHmultipliers_std} above, we decompose $F = F^{(0)} + F^{(1)}$, and deduce from Theorem \ref{t: multipliers} the $L^p$ boundedness of $F^{(1)}(\opL)$ for all $p \in [1,\infty]$. Moreover,
\[
F^{(0)}(\opL) = \sum_{\ell \geq 0} F^{(0)}_\ell(2^\ell \opL) = \lim_{N \to \infty} F^{(0)}_{(N)}(\opL),
\]
in the sense of the strong operator topology on $L^2(m)$, where
\begin{equation}\label{eq:F0partialsum}
F^{(0)}_{(N)} = \sum_{\ell = 0}^N F^{(0)}_\ell(2^\ell \cdot) = F \sum_{\ell=0}^N \phi(2^\ell \cdot).
\end{equation}
In particular
\[
\|F^{(0)}(\opL)\|_{L^p(m) \to L^p(m)} \leq \sup_{N \in \NN} \| F^{(0)}_{(N)}(\opL) \|_{L^p(m) \to L^p(m)},
\]
so in order to conclude it is enough to check that, for any $p \in (1,2]$, the truncations $F^{(0)}_{(N)}(\opL)$ of $F^{(0)}(\opL)$ are bounded on $L^p(m)$ uniformly in $N$. On the other hand, from \eqref{eq:F0partialsum} and \eqref{eq:MHcond_red} it is clear that each $F^{(0)}_{N}$ is in $L^2_s(\RR)$; thus, by applying Theorem \ref{t: unif_transf_opL}, we deduce that
\[
\|F^{(0)}(\opL)\|_{L^p(m) \to L^p(m)} \leq \sup_{N \in \NN} \sup_q \| F^{(0)}_{(N)}(\Lq) \|_{L^p(\mq) \to L^p(\mq)}.
\]
So we are reduced to proving an analogous bound for the truncations on the homogeneous tree $\Tq$, which however must be uniform both in $N$ and $q$.

An $L^p$ bound for each $F^{(0)}_{(N)}(\Lq)$ can be deduced following the proof of part \ref{en: MHmultipliers_std} above on the tree $(\Tq,\mq)$. Indeed, here we have the dyadic decomposition $F^{(0)}_{(N)}(\Lq)=\sum_{\ell = 0}^N F^{(0)}_\ell(2^\ell \Lq)$, and each dyadic piece satisfies the kernel estimates \eqref{f: CZ_cond_mult}, which hold uniformly in $N$ and $q$. Moreover, clearly
\[
\|F^{(0)}_{(N)}(\Lq)\|_{L^2(\mq) \to L^2(\mq)} \leq \|F^{(0)}_{(N)}\|_\infty \leq \|F^{(0)}\|_\infty.
\]
In other words, the truncations $F^{(0)}_{(N)}(\Lq)$ satisfy the assumptions of Theorem \ref{t: lemmahebisch_sharp} uniformly in $N$ and $q$; thus, for any $p \in (1,2]$, we deduce their $L^p(\mq)$ boundedness with the same uniformity in $q$ and $N$, as required.
\end{proof}

The assumption on the multiplier $F$ in Theorem \ref{t: MHmultipliers}\ref{en: MHmultipliers_twist} does not imply the boundedness of $F(\opL)$ from $H^1(m)$ to $L^1(m)$. This follows from the next result.

\begin{prop}\label{p: twist_MH_counterex}
Let $(T,m)$ be a locally doubling flow tree. For any $s > 3/2$, there exist functions $F : \RR \to \CC$ satisfying
\begin{equation}\label{f: MH_cond_onlytwist}
\sup_{t>0} \| F(2-t \cdot) \chi \|_{L^2_s} < \infty
\end{equation}
and such that $F(\opL)$ is not bounded from $H^1(m)$ to $L^1(m)$.
\end{prop}
\begin{proof}
Arguing by contradiction, assume instead that there exists $s>3/2$ such that, for any function $F$ satisfying \eqref{f: MH_cond_onlytwist}, the operator $F(\opL)$ is bounded from $H^1(m)$ to $L^1(m)$ . An application of the Closed Graph Theorem then shows that the bound
\begin{equation}\label{f: H1L1Sob_bd_contr}
\|F(\opL)\|_{H^1(m) \to L^1(m)} \leq C \sup_{t>0} \| F(2-t \cdot) \chi \|_{L^2_s}
\end{equation}
holds for some $C \in (0,\infty)$. We shall now contradict the validity of this bound.

Fix $\alpha \in \RR \setminus \{0\}$, and consider the function 
\begin{equation}\label{f: potimm}
G(\lambda)=\lambda^{i\alpha},
\end{equation}
which satisfies the condition 
\[
\sup_{t>0} \|G(t\cdot)\chi\|_{L^2_s} < \infty
\]
for every $s>0$.
Define $F=G(2 -\cdot)$. Then obviously 
\begin{equation}\label{f: twMHbdG}
\sup_{t>0} \|F(2-t\cdot)\chi\|_{L^2_s} < \infty.
\end{equation}

Consider the flow tree $\ZZ$ equipped with the counting measure. It is not difficult to see that the operator $G(\Delta_{\ZZ})= \Delta_\ZZ^{i\alpha}$ is $L^1$-unbounded; indeed
\[
k_{\Delta_\ZZ^{i\alpha}}(n) = \frac{2^{i\alpha}}{\sqrt{\pi}} \frac{\Gamma(1/2+i\alpha)}{\Gamma(-i\alpha)} \frac{\Gamma(|n|-i\alpha)}{\Gamma(|n|+1+i\alpha)}, \qquad n \in \ZZ,
\]
(see, e.g., \cite[eq.\ (1.12)]{CRSTV}), and known asymptotics for the Gamma function (see, e.g., \cite[eq.\ (5.11.12)]{DLMF}) show that $|k_{\Delta_\ZZ^{i\alpha}}(n)| \approx_\alpha |n|^{-1}$ for large $|n|$, so $G(\Delta_\ZZ) \delta_0 = k_{\Delta_\ZZ^{i\alpha}} \notin L^1(\ZZ)$. On the other hand, the function $a=\delta_0-\delta_{-1}$ is an atom in $H^1(\ZZ)$ and, by Theorem \ref{t: MHmultipliers}\ref{en: MHmultipliers_std}, $G(\Delta_{\ZZ}) a\in L^1(\ZZ)$. Then
\[\begin{split}
F(\Delta_{\ZZ})a
&= \modE G(\Delta_{\ZZ}) \modE a \\
&= \modE G(\Delta_{\ZZ})(\delta_0+\delta_{-1}) \\
&= \modE G(\Delta_{\ZZ})(2\delta_0-a) \\
&= 2 \modE G(\Delta_{\ZZ})\delta_0 - \modE G(\Delta_{\ZZ}) a.
\end{split}\]
Since $\modE G(\Delta_{\ZZ}) a \in L^1(\ZZ)$ and $\modE G(\Delta_{\ZZ}) \delta_0 \notin L^1(\ZZ)$, we conclude that $F(\Delta_{\ZZ}) a \notin L^1(\ZZ)$. Hence $F$ satisfies \eqref{f: twMHbdG}, but $F(\Delta_{\ZZ})$ does not map $H^1(\ZZ)$ into $L^1(\ZZ)$, thus contradicting the bound \eqref{f: H1L1Sob_bd_contr} in the case of the flow tree $\ZZ$.

We shall now use transference results to construct a similar counterexample for every locally doubling flow tree. To do so, we first multiply the function $G$ defined in \eqref{f: potimm} by a smooth cutoff function, thus obtaining a new multiplier $H$ supported in $[0,3]$ such that $H|_{[0,2]}= F|_{[0,2]}$ and, for every $s>0$,
\[
\sup_{t>0} \|H(t\cdot)\chi\|_{L^2_s} < \infty.
\]
Given $\eta\in C^{\infty}_c(\RR)$ supported in $(-2,2)$ such that $\eta(\lambda)=1$ for $\lambda\in [-1,1]$ and $0\leq \eta\leq 1$, define $H_n(\lambda)=H(\lambda)(1-\eta(n\lambda))$, for $n\in \NN$. Then $H_n$ is supported in $[0,3]\setminus (-1/n,1/n)$, it is smooth and, for every $s>0$,
\begin{equation}\label{f: Hn}
\sup_{n\in\NN} \sup_{t>0} \|H_n(t\cdot)\chi\|_{L^2_s} < \infty.
\end{equation}
Moreover, $H_n$ tends to $H$ pointwise and boundedly on $(0,\infty)$ as $n \to \infty$, thus $H_n(2-\Delta_{\ZZ})$ tends to $H(2-\Delta_{\ZZ})$ in the strong operator topology on $L^2(\ZZ)$. In particular, $H_n(2-\Delta_{\ZZ})a$ tends to $H(2-\Delta_{\ZZ})a$ in $L^2(\ZZ)$ as $n \to \infty$. Since $H(2-\Delta_{\ZZ})a = F(\Delta_{\ZZ})a$ is not in $L^1(\ZZ)$, it follows that $\sup_n \|H_n(2-\Delta_{\ZZ})a\|_{L^1(\ZZ)} = \infty$. 

Let now $(T,m)$ be a locally doubling flow tree and let $\pi : T \to \ZZ$ be the submersion defined by $\pi(x)=\ell(x)$, $x\in T$. Then, by Proposition \ref{p: Intral_smooth}, $H_n(2-\opL) \in \Comp(\pi)$ and $\pi(H_n(2-\opL)) = H_n(2-\Delta_{\ZZ})$.
 Define $b=\chr_{\{o\}} -\Sigma \chr_{\{o\}}$, which is in $H^1(m)$ and such that $\Phi_\pi^*b=a$, where $\Phi_\pi$ is the lifting operator associated to $\pi$, and $\Phi_\pi^*$ is its adjoint, as in \eqref{f: adj_lifting}. From \eqref{f: M1M2comp_adj} it follows that $\Phi_\pi^* H_n(2-\opL)b=H_n(2-\Delta_{\ZZ})a$; as $\Phi_\pi^* : L^1(m) \to L^1(\ZZ)$ has norm $1$, 
\[
\sup_{n \in \NN} \|H_n(2-\opL)b\|_{L^1(m)} \geq \sup_{n \in \NN} \|H_n(2-\Delta_{\ZZ})a\|_{L^1(\ZZ)} = \infty.
\]
Thus, if we set $F_n = H_n(2-\cdot)$, then  
\[
\sup_{n\in \NN} \|F_n(\opL)\|_{H^1(m) \to L^1(m)}=\infty;
\]
together with \eqref{f: Hn}, this shows that the $F_n$ provide a counterexample to \eqref{f: H1L1Sob_bd_contr}.
\end{proof}

Finally, we discuss the optimality of the threshold $3/2$ in the above multiplier theorems in the case of the homogeneous trees. We start with a preliminary lemma.

\begin{lemma}\label{l: nabla_L1H1}
Let $(T,m)$ be a locally doubling flow tree. Then, the flow gradient $\nabla$ is bounded from $L^1(m)$ to $H^1(m)$.
\end{lemma}
\begin{proof}
For all $f \in L^1(m)$, we can write
\[
f = \sum_{x \in T} f(x) \, m(x) \, b_x, \qquad b_x \defeq \frac{\chr_{\{x\}}}{m(x)},
\]
thus
\[
\nabla f = \sum_{x \in T} f(x) m(x) \nabla b_x.
\]
Moreover, from \cite[Definition 4.4]{LSTV} it is clear that
\[
\nabla b_x = \frac{\chr_{\{x\}}-\chr_{\succ(x)}}{m(x)}
\]
belongs to $H^1(m)$, with $H^1$-norm uniformly bounded in $x \in T$. Thus,
\[
\|\nabla f\|_{H^1(m)} \leq \sum_{x \in T} |f(x)| \, m(x) \, \|\nabla b_x\|_{H^1(m)} \lesssim \|f\|_{L^1(m)},
\]
as required.
\end{proof}

We now show that the smoothness threshold $3/2$ in Theorems \ref{t: multipliers} and \ref{t: MHmultipliers} cannot be replaced by any smaller quantity. As explained in Remark \ref{r: sharpness}, this is a consequence of the following result, where $L^\infty_s(\RR)$ denotes the $L^\infty$ Sobolev space of order $s$ on $\RR$; much as in \cite{MMu}, the idea is to test the above multiplier theorems on a truncated version of the Schr\"odinger propagator.

\begin{prop}\label{p: mult_sharpness}
Let $\chi_0 \in C_c^\infty(\RR)$ be such that $\supp \chi_0 \subseteq (-1/2,1/2)$ and $\chi_0|_{[-1/4,1/4]} = 1$. For all $t \in \RR$, let $F_t(\lambda) = e^{it\lambda} \chi_0(\lambda)$. Then, for all $s \geq 0$,
\begin{equation}\label{f: above_est_mult_schr}
\|F_t\|_{L^\infty_{s}} \lesssim_s (1+|t|)^{s} \qquad\forall t \in \RR.
\end{equation}
Moreover, there exists $t_0 > 0$ such that
\begin{equation}\label{f: below_est_mult_schr}
\|F_t(\Lq)\|_{H^1(\mq) \to L^1(\mq)} \gtrsim_q t^{3/2} \qquad \forall t \geq t_0,\ q \geq 2.
\end{equation}
\end{prop}
\begin{proof}
An elementary computation shows the validity of \eqref{f: above_est_mult_schr}; indeed, clearly
\[
\|F_t^{(k)}\|_\infty \lesssim_k (1+|t|)^k
\]
for all $k \in \NN$, which implies the result for integer $s$; the case of fractional $s$ follows by interpolation.

It remains to prove the lower bound \eqref{f: below_est_mult_schr}. By Lemma \ref{l: nabla_L1H1}, it will be enough to prove that
\[
\|F_t(\Lq) \nabla_{\Tq}\|_{L^1(\mq) \to L^1(\mq)} \gtrsim_q t^{3/2} \qquad \forall t \geq t_0,\ q \geq 2
\]
for a sufficiently large $t_0 > 0$.

From Corollary \ref{c: KF(L)} it follows that
\[
K_{F_t(\Lq)}(x,y) = q^{-(\ell(x)+\ell(y))/2} E_{F_t}(d(x,y)),
\]
thus also
\[\begin{split}
K_{F_t(\Lq) \nabla_{\Tq}}(x,y) 
&= K_{F_t(\opL)} (x,y) - \frac{1}{q} \sum_{z \in \succ(y)} K_{F_t(\opL)}(x,z) \\
&= q^{-(\ell(x)+\ell(y))/2} \left[ E_{F_t}(d(x,y)) - q^{-1/2} \sum_{z \in \succ(y)} E_{F_t}(d(x,z)) \right].
\end{split}\]
In particular,
\begin{equation}\label{f: kernel_FtNabla}
K_{F_t(\Lq) \nabla_{\Tq}}(x,y) = q^{-(\ell(x)+\ell(y))/2} \tilde E_{F_t}(d(x,y)) \qquad\text{if } x \not< y, 
\end{equation}
where, in light of \eqref{f: E_F} and Lemma \ref{l: kernelFZ},
\begin{equation}\label{f: osc_comp}
\begin{split}
\tilde E_{F_t}(k) 
&= E_{F_t}(k) - q^{1/2} E_{F_t}(k+1) \\
&= \sum_{j\geq 0} q^{-(k+2j)/2} \, [\tilde{\nabla}_{\ZZ} k_{F_t(\Delta_{\ZZ})}(k+2j+1)-\tilde{\nabla}_{\ZZ} k_{F_t(\Delta_{\ZZ})}(k+2j+2)] \\
&= \frac{1}{i\pi} \sum_{j\geq 0} q^{-(k+2j)/2}  \int_{-\pi}^{\pi} \sin\theta \, F_t(1-\cos\theta) \, (1-e^{i\theta}) \, e^{i(k+2j+1)\theta} \, d\theta\\
&= \frac{q^{-k/2}}{i\pi} \int_{-\pi}^{\pi} \sin\theta \, F_t(1-\cos\theta) \, (1-e^{i\theta}) \, \left(1-\frac{e^{2i\theta}}{q}\right)^{-1} \, e^{i(k+1)\theta} \, d\theta \\
&= \frac{q^{-k/2}}{i\pi} I_q\left(\frac{k+1}{t};t\right)
\end{split}
\end{equation}
for all $k \in \ZZ$ and $t > 0$; here
\begin{align*}
I_q(\xi;t) &= \int_{\RR} e^{it \phi_\xi(\theta)} A_q(\theta) \,d \theta,\\
\phi_\xi(\theta) &= 1-\cos\theta + \xi\theta,\\
A_q(\theta) &=   (1-e^{i\theta}) \, \sin \theta  \, \left(1-\frac{e^{2i\theta}}{q}\right)^{-1} \, \chi_0(1-\cos\theta) \, \tilde\chi_0(\theta)
\end{align*}
and $\tilde\chi_0 \in C^\infty_c(\RR)$ is such that $\supp \tilde\chi_0 \subseteq (-\pi/2,\pi/2)$ and $\tilde\chi_0|_{[-\pi/3,\pi/3]} = 1$.

We now study the oscillatory integral $I_q(\xi;t)$. Notice that
\[
\phi_\xi'(\theta) = \sin\theta + \xi, \qquad \phi_\xi''(\theta) = \cos\theta.
\]
Thus, if $|\xi| \leq 1/2$, then the phase function $\phi_\xi$ has critical points $\pi k - (-1)^k \arcsin \xi$ where $k \in \ZZ$, and the only critical point lying in the support of the amplitude $A_q$ is the one for $k=0$, i.e., $\theta_c(\xi) \defeq -\arcsin \xi$. Moreover, clearly $\phi_\xi''(\theta_c(\xi)) = \sqrt{1-\xi^2} \approx 1$, i.e., the critical point is nondegenerate. The method of stationary phase (see, e.g., \cite[Theorem 7.7.6]{Hor1}) then shows that
\[
I_q(\xi;t) = \sqrt{2\pi i} \frac{t^{-1/2}}{\sqrt{1-\xi^2}} e^{it\phi_\xi(\theta_c(\xi))} \, A_q(\theta_c(\xi))  + O(t^{-3/2}) 
\quad\text{as } t \to \infty,
\]
uniformly in $|\xi| \leq 1/2$. In particular, in the region where $1/4 \leq |\xi| \leq 1/2$, we have
\[
|A_q(\theta_c(\xi))| \approx 1
\]
and
\[
|I_q(\xi;t)| \gtrsim t^{-1/2} \qquad\forall t \geq t_0
\]
for some sufficiently large $t_0 > 0$.

By \eqref{f: osc_comp}, this shows that, for all $t \geq t_0$, and all $k \in \NN$ such that $t/4 \leq k+1 \leq t/2$, 
\[
|\tilde E_{F_t}(k)| \gtrsim q^{-k/2} t^{-1/2}.
\]
Therefore, by \eqref{f: kernel_FtNabla}, for all $t \geq t_0$,
\[\begin{split}
&\|F_t(\Lq) \nabla_{\Tq}\|_{L^1(\mq) \to L^1(\mq)} \geq \sup_{y \in \Tq} \sum_{x \in \Tq \colon x \not< y} |K_{F_t(\Lq) \nabla_{\Tq}}(x,y)| \, q^{\ell(x)} \\
&\approx \sum_{k \in \NN} |\tilde E_{F_t}(k)| \, (k+1) \,q^{k/2} \gtrsim \sum_{k \colon t/4 \leq k+1 \leq t/2} t^{-1/2} \, (k+1) \approx t^{3/2},
\end{split}\]
as required. In the middle step, we used the fact that, for all $y \in \Tq$ and $k \in \NN$,
\[
\sum_{x \in \Tq \colon d(x,y) = k, \, x \not< y} q^{(\ell(x)-\ell(y))/2} \approx q^{k/2} (k+1),
\]
which is proved much in the same way as \cite[eq.\ (2.18)]{MSV}.
\end{proof}

\begin{remark}\label{r: sharpness}
Clearly \eqref{f: below_est_mult_schr} implies an analogous lower bound for the $L^1$ operator norm of $F_t(\Lq)$, while \eqref{f: above_est_mult_schr} implies an analogous upper bound for $\|F_t\|_{L^2_s}$, because $F_t$ is supported in $[-1/2,1/2]$; this shows that the bound \eqref{f: multipliers_bound} cannot hold for any $s < 3/2$, thus proving the optimality of Theorem \ref{t: multipliers} on $(\Tq,\mq)$ with $q \geq 2$. Moreover,
\[
\sup_{0 < t \leq 2} \|F(t\cdot) \chi\|_{L^2_s} \lesssim_s \sup_{0 < t \leq 2} \|F(t\cdot) \chi\|_{L^\infty_s} \lesssim_s \|F\|_{L^\infty_s}
\]
for all $s \geq 0$, where $\chi$ is as in Theorem \ref{t: MHmultipliers}; thus, similar considerations prove that no bound of the form
\[
\|F(\Lq)\|_{H^1 \to L^1} \lesssim_s \sup_{0 < t \leq 2} \|F(t\cdot) \chi\|_{L^2_s}
\]
may hold when $s < 3/2$ and $q \geq 2$; an application of the Closed Graph Theorem (as in the proof of Proposition \ref{p: twist_MH_counterex}) then shows the optimality of the $H^1 \to L^1$ bound of Theorem \ref{t: MHmultipliers}\ref{en: MHmultipliers_std}. This discussion actually shows that the threshold $3/2$ remains optimal even when the smoothness conditions in Theorems \ref{t: multipliers} and \ref{t: MHmultipliers} are strengthened by replacing $L^2_s$ with $L^\infty_s$.
\end{remark}

\end{document}